\documentclass[11pt]{amsart}

\usepackage{amsmath}
\usepackage{amssymb}
\usepackage{amscd}

\usepackage{xypic}
\xyoption{all}

\newcommand{\nc}{\newcommand}

\nc{\exto}[1]{\stackrel{#1}{\longrightarrow}}
\nc{\dlim}{{\mathop{\lim\limits_{\longrightarrow}\,}}}
\nc{\ilim}{{\mathop{\lim\limits_{\longleftarrow}\,}}}
\nc{\hocolim}{{\mathop{\sf hocolim}\,}}
\nc{\holim}{{\mathop{\sf holim}}}
\nc{\lan}{\big\langle}
\nc{\ran}{\big\rangle}

\nc{\kk}{{\mathsf{k}}}

\nc{\C}{{\mathbb{C}}}
\nc{\HH}{{\mathbf{H}}}
\nc{\LL}{{\mathbb{L}}}
\nc{\PP}{{\mathbb{P}}}
\nc{\QQ}{{\mathbb{Q}}}
\nc{\RR}{{\mathbb{R}}}
\nc{\ZZ}{{\mathbb{Z}}}

\nc{\CA}{{\mathcal{A}}}
\nc{\CB}{{\mathcal{B}}}
\nc{\CC}{{\mathcal{C}}}
\nc{\D}{{\mathcal{D}}}
\nc{\CE}{{\mathcal{E}}}
\nc{\CF}{{\mathcal{F}}}
\nc{\CG}{{\mathcal{G}}}
\nc{\CH}{{\mathcal{H}}}
\nc{\CI}{{\mathcal{I}}}
\nc{\CL}{{\mathcal{L}}}
\nc{\CM}{{\mathcal{M}}}
\nc{\CN}{{\mathcal{N}}}
\nc{\CO}{{\mathcal{O}}}
\nc{\CP}{{\mathcal{P}}}
\nc{\CQ}{{\mathcal{Q}}}
\nc{\CR}{{\mathcal{R}}}
\nc{\CS}{{\mathcal{S}}}
\nc{\CT}{{\mathcal{T}}}
\nc{\CU}{{\mathcal{U}}}
\nc{\CV}{{\mathcal{V}}}
\nc{\CW}{{\mathcal{W}}}
\nc{\CX}{{\mathcal{X}}}
\nc{\CY}{{\mathcal{Y}}}
\nc{\CMo}{{\mathcal{M}^\circ}}
\nc{\Co}{{{C}^\circ}}

\nc{\BY}{{\overline{Y}}}
\nc{\BYD}{{\overline{Y}{}^{|D|}}}
\nc{\OZ}{{\overline{Z}}}
\nc{\bg}{{\bar{g}}}

\nc{\bq}{{\mathbf{q}}}
\nc{\BB}{{\mathbf{B}}}
\nc{\BC}{{\mathbf{C}}}
\nc{\BD}{{\mathbf{D}}}
\nc{\BG}{{\mathbf{G}}}
\nc{\BH}{{\mathbf{H}}}
\nc{\BK}{{\mathbf{K}}}
\nc{\BL}{{\mathbf{L}}}
\nc{\BM}{{\mathbf{M}}}
\nc{\BP}{{\mathbf{P}}}
\nc{\BT}{{\mathbf{T}}}
\nc{\BU}{{\mathbf{U}}}
\nc{\BZ}{{\mathbf{Z}}}
\nc{\BPr}{{\mathsf{P}}}
\nc{\BR}{{\mathbf{R}}}
\nc{\BW}{{\mathbf{W}}}
\nc{\BRO}[1]{{{\mathbf{R}}^{\circ}_{#1}}}
\nc{\BRD}[1]{{{\mathbf{R}}^{|D|}_{#1}}}
\nc{\BRP}[1]{{{\mathbf{R}}^{1}_{#1}}}
\nc{\BRTP}[1]{{{\mathbf{\tilde{R}}}{}^{1}_{#1}}}
\nc{\BS}{{\mathbf{S}}}
\nc{\BMS}{{{\mathbf{M}}^{{s}}}}
\nc{\BMSS}{{{\mathbf{M}}^{{ss}}}}
\nc{\BMZ}{{\mathbf{M}^{\circ}}}
\nc{\BCL}{{\mathbf{L}}}

\nc{\PCC}{{{}^\perp\CC}}

\nc{\Cl}{{\mathsf{Cliff}}}
\nc{\Clev}{{\mathop{\mathsf{Cliff}}^{\circ}}}

\nc{\FA}{{\mathfrak{A}}}
\nc{\FB}{{\mathfrak{B}}}

\nc{\fd}{{\mathfrak{d}}}
\nc{\fa}{{\mathfrak{a}}}
\nc{\fb}{{\mathfrak{b}}}
\nc{\fg}{{\mathfrak{g}}}
\nc{\fn}{{\mathfrak{n}}}
\nc{\fp}{{\mathfrak{p}}}
\nc{\FD}{{\mathfrak{D}}}
\nc{\FE}{{\mathfrak{E}}}
\nc{\FL}{{\mathfrak{L}}}
\nc{\FM}{{\mathfrak{M}}}
\nc{\FS}{{\mathsf{S}}}
\nc{\FU}{{\mathsf{U}}}

\nc{\fr}{{{\mathop{\mathsf{fr}}}}}
\nc{\FR}{{{\mathop{\mathrm{FR}}}}}

\nc{\sfc}{{\mathsf{c}}}
\nc{\sfch}{{\mathsf{ch}}}
\nc{\sfh}{{\mathsf{h}}}

\nc{\SB}{{\mathsf{B}}}
\nc{\SBi}{\SB^\inn}
\nc{\SBo}{\SB^\out}
\nc{\HSB}{{\hat{\mathsf{B}}}}
\nc{\HSBi}{{\hat{\mathsf{B}}^\inn}}
\nc{\HSBo}{{\hat{\mathsf{B}}^\out}}
\nc{\SBS}{{\bar{\mathsf{B}}}}
\nc{\SBSi}{{\bar{\mathsf{B}}^\inn}}
\nc{\SBSo}{{\bar{\mathsf{B}}^\out}}
\nc{\BRS}{{\mathop{\BR}\limits^\circ}}
\nc{\sj}{{\mathsf{j}}}
\nc{\SJ}{{\mathsf{J}}}
\nc{\SK}{{\mathsf{K}}}
\nc{\SM}{{\mathsf{M}}}
\nc{\SO}{{\mathsf{O}}}
\nc{\SQ}{{\mathsf{Q}}}
\nc{\SPV}{{\mathsf{S}^+\mathsf{V}}}
\nc{\SMV}{{\mathsf{S}^-\mathsf{V}}}
\nc{\SPMV}{{\mathsf{S}^\pm\mathsf{V}}}
\nc{\SX}{{S_X}}
\nc{\SY}{{S_Y}}
\nc{\phipsi}{{q}}
\nc{\eps}{\varepsilon}

\nc{\pim}{{\pi_-}}
\nc{\pip}{{\pi_+}}

\nc{\BE}{{\overline{\CE}}}
\nc{\TE}{{\tilde{\CE}}}
\nc{\TQ}{{\tilde{Q}}}
\nc{\TCF}{{\tilde{\CF}}}
\nc{\TCG}{{\tilde{\CG}}}
\nc{\TCL}{{\tilde{\CL}}}
\nc{\TF}{{\tilde{F}}}
\nc{\TR}{{\tilde{R}}}
\nc{\TW}{{\tilde{W}}}
\nc{\TCA}{{\widetilde{\CA}}}
\nc{\TCC}{{\tilde{\CC}}}
\nc{\TCX}{{\tilde{\CX}}}
\nc{\TCY}{{\tilde{\CY}}}
\nc{\TPi}{{\tilde{\Pi}}}
\nc{\TPhi}{{\tilde{\Phi}}}
\nc{\OPhi}{{\bar{\Phi}}}
\nc{\txi}{{\tilde{\xi}}}
\nc{\TXi}{{\widetilde{\Xi}}}
\nc{\tp}{{\tilde{p}}}
\nc{\tq}{{\tilde{q}}}
\nc{\tzeta}{{\tilde{\zeta}}}
\nc{\tpi}{{\tilde{\pi}}}
\nc{\tlambda}{{\tilde{\lambda}}}
\nc{\tnu}{{\tilde{\nu}}}

\nc{\hgamma}{{\hat{\gamma}}}
\nc{\halpha}{{\hat{\alpha}}}
\nc{\HCA}{{\hat{\CA}}}
\nc{\HCB}{{\hat{\CB}}}
\nc{\HCC}{{\hat{\CC}}}
\nc{\HE}{{\widehat{\CE}}}
\nc{\HR}{{\hat{R}}}
\nc{\HX}{{\hat{X}}}
\nc{\hxi}{{\hat{\xi}}}
\nc{\hlambda}{{\hat{\lambda}}}
\nc{\hnu}{{\hat{\nu}}}

\nc{\UH}{{\mathcal{H}}}

\nc{\TM}{{\widetilde{M}}}
\nc{\TCM}{{\widetilde{\CM}}}
\nc{\TU}{{\widetilde{U}}}
\nc{\TX}{{\widetilde{X}}}
\nc{\TY}{{\widetilde{Y}}}
\nc{\TYO}{{{\widetilde{Y}}^\circ}}
\nc{\barf}{{\bar{f}}}
\nc{\te}{{\tilde{e}}{}}
\nc{\tf}{{\tilde{f}}}
\nc{\tg}{{\tilde{g}}}
\nc{\ti}{{\tilde{\imath}}}
\nc{\tj}{{\tilde{\jmath}}}
\nc{\ty}{{\tilde{y}}}
\nc{\tphi}{{\tilde{\phi}}}

\nc{\depth}{\mathop{\mathsf{depth}}}
\nc{\SSS}{{\mathfrak{S}}}

\nc{\urho}{{\underline{\rho}}}

\nc{\LRA}{\Leftrightarrow}
\nc{\RA}{\Rightarrow}
\nc{\lotimes}{\mathbin{\mathop{\otimes}\limits^{\mathbb{L}}}}
\nc{\CEnd}{\mathop{\mathcal{E}\mathit{nd}}\nolimits}
\nc{\CExt}{\mathop{\mathcal{E}\mathit{xt}}\nolimits}
\nc{\CHom}{\mathop{\mathcal{H}\mathit{om}}\nolimits}
\nc{\RH}{\mathop{{\mathsf{R}}\Gamma}\nolimits}
\nc{\RGamma}{\mathop{{\mathsf{R}}\Gamma}\nolimits}
\nc{\RHom}{\mathop{\mathsf{RHom}}\nolimits}
\nc{\RCHom}{\mathop{\mathsf{R}\mathcal{H}\mathit{om}}\nolimits}
\nc{\RG}{\mathop{\mathsf{R\Gamma}}\nolimits}
\nc{\mult}{\mathop{\mathsf{mult}}\nolimits}
\nc{\Hom}{\mathop{\mathsf{Hom}}\nolimits}
\nc{\Ext}{\mathop{\mathsf{Ext}}\nolimits}
\nc{\End}{\mathop{\mathsf{End}}\nolimits}
\nc{\Tor}{\mathop{\mathsf{Tor}}\nolimits}
\nc{\Tordim}{\mathop{\mathsf{Tor}\text{\rm-}\mathsf{dim}}\nolimits}
\nc{\Hilb}{\mathop{\mathsf{Hilb}}\nolimits}
\nc{\Spec}{\mathop{\mathsf{Spec}}\nolimits}
\nc{\Pic}{\mathop{\mathsf{Pic}}\nolimits}

\nc{\Tr}{\mathop{\mathsf{Tr}}\nolimits}
\nc{\Cone}{\mathop{\mathsf{Cone}}\nolimits}
\nc{\Conv}{\mathop{\mathsf{Conv}}\nolimits}
\nc{\Fiber}{\mathop{\mathsf{Fiber}}\nolimits}
\nc{\SR}{\mathop{\mathsf{SR}}\nolimits}
\nc{\Ker}{\mathop{\mathsf{Ker}}\nolimits}
\nc{\Coker}{\mathop{\mathsf{Coker}}\nolimits}
\nc{\codim}{\mathop{\mathsf{codim}}\nolimits}
\nc{\reg}{{\mathsf{reg}}}
\nc{\sing}{{\mathsf{sing}}}
\nc{\supp}{\mathop{\mathsf{supp}}}
\nc{\vol}{\mathop{\mathsf{vol}}\nolimits}
\nc{\ch}{\mathop{\mathsf{ch}}\nolimits}
\nc{\perf}{{\mathsf{perf}}}
\nc{\rk}{\mathop{\mathsf{rk}}}
\nc{\Pf}{{\mathsf{Pf}}}
\nc{\Gr}{{\mathsf{Gr}}}
\nc{\OGr}{{\mathsf{OGr}}}
\nc{\Flag}{{\mathsf{Fl}}}
\nc{\Kosz}{{\mathsf{Kosz}}}
\nc{\LGr}{{\mathsf{LGr}}}
\nc{\LFl}{{\mathsf{LFl}}}
\nc{\SGr}{{\mathsf{SGr}}}
\nc{\GTGr}{{\mathsf{G_2Gr}}}
\nc{\GTF}{{\mathsf{G_2F}}}
\nc{\OF}{{\mathsf{OF}}}
\nc{\Fl}{{\mathsf{Fl}}}
\nc{\Bl}{{\mathsf{Bl}}}
\nc{\Gm}{{\mathbb{G}_m}}
\nc{\GL}{{\mathsf{GL}}}
\nc{\PGL}{{\mathsf{PGL}}}
\nc{\SL}{{\mathsf{SL}}}
\nc{\SP}{{\mathsf{Sp}}}
\nc{\Spin}{{\mathsf{Spin}}}
\nc{\Tot}{{\mathsf{Tot}}}
\nc{\ev}{{\mathsf{ev}}}
\nc{\tr}{{\mathsf{tr}}}
\nc{\mn}{{\mathsf{min}}}
\nc{\mx}{{\mathsf{max}}}
\nc{\od}{{\mathsf{odd}}}
\nc{\coev}{{\mathsf{coev}}}
\nc{\id}{{\mathsf{id}}}
\nc{\opp}{{\mathsf{opp}}}
\nc{\PS}{{{\PP^3}}}
\nc{\Qu}{{{Q^3}}}
\nc{\tdim}{\mathop{\Tor\dim}}
\nc{\ecart}{{\fbox{$\scriptstyle\mathsf{EC}$}}}
\nc{\ad}{{\mathop{\mathsf ad}}}
\nc{\sg}{{\mathop{\mathsf sg}}}
\nc{\hf}{{\mathop{\mathsf hf}}}
\nc{\gr}{{\mathop{\mathsf gr}}}
\nc{\qgr}{{\mathop{\mathsf qgr}}}
\nc{\Coh}{{\mathop{{\mathsf{Coh}}}}}
\nc{\Ab}{{\mathop{\mathcal{A}\mathit{b}}}}
\nc{\Ccoh}{{\mathop{\mathsf Ccoh}}}
\nc{\Qcoh}{{\mathop{\mathsf Qcoh}}}
\nc{\At}{\mathop{\mathsf{At}}}
\nc{\tra}{{\mathsf{T}}}
\nc{\fsl}{{\mathfrak{sl}}}
\nc{\fso}{{\mathfrak{so}}}
\nc{\fgl}{{\mathfrak{gl}}}

\newcommand{\w}{{\boldsymbol w}}

\nc{\AAV}{{\mathcal{AAV}}}

\nc{\Rep}{\mathop{\mathsf{Rep}}}

\nc{\Cubics}{{{\mathcal{S}}_3}}
\nc{\VFT}{{{\mathcal{S}}_{14}}}
\nc{\VFTE}{{{\mathcal{N}}_{\mathrm{reg,sm}}}}
\nc{\MX}{{\CM_X}}
\nc{\MY}{{\CM_Y}}
\nc{\MYE}{{\CM_{Y,\CE}}}
\nc{\Yd}{{Y_d}}
\nc{\Yfive}{{Y_5}}
\nc{\Xg}{{X_{2g-2}}}
\nc{\Xtt}{{X_{22}}}
\nc{\Xst}{{X_{16}}}
\nc{\Xtw}{{X_{12}}}
\nc{\Xe}{{X_{8}}}
\nc{\Xf}{{X_{4}}}

\nc{\git}{{/\!\!/\!{}_\chi}}

\nc{\HOH}{{\mathsf H\mathsf H}}
\nc{\HHE}{{\mathsf H\mathsf E}}

\nc{\cu}{{\mathfrak{u}}}
\nc{\ce}{{\mathfrak{e}}}
\nc{\cf}{{\mathfrak{f}}}
\nc{\co}{{\mathfrak{o}}}

\nc{\VB}{{\mathsf{VB}}}

\nc{\pio}{\pi_\out}
\nc{\pii}{\pi_\inn}
\nc{\Po}{P_\out}
\nc{\Pin}{P_\inn}
\nc{\Do}{D_\out}
\nc{\Di}{D_\inn}
\nc{\lo}{\lambda_\out}
\nc{\li}{\lambda_\inn}
\nc{\RRo}{\RR^\out}
\nc{\RRi}{\RR^\inn}
\nc{\BLo}{{\BL_\out}}
\nc{\BLi}{{\BL_\inn}}
\nc{\BMo}{{\BM_\out}}
\nc{\BMi}{{\BM_\inn}}

\nc{\OP}{{{\mathsf{OP}}}}
\nc{\TP}{{\widetilde{P}}}
\nc{\tomega}{{\tilde{\omega}}}
\nc{\vi}{{v_{\text{\bf i}}}}
\nc{\vo}{{v_{\text{\bf o}}}}
\nc{\inn}{{{\text{\bf inn}}}}
\nc{\out}{{{\text{\bf out}}}}
\nc{\TL}{{\widetilde{L}}}
\nc{\TBL}{{\widetilde{\BL}}}
\nc{\TS}{{\widetilde{S}}}

\nc{\bd}{{\bar{d}}}

\theoremstyle{plain}

\newtheorem{theorem}{Theorem}[section]
\newtheorem{conjecture}[theorem]{Conjecture}
\newtheorem{lemma}[theorem]{Lemma}
\newtheorem{proposition}[theorem]{Proposition}
\newtheorem{corollary}[theorem]{Corollary}

\theoremstyle{definition}

\newtheorem{definition}[theorem]{Definition}

\newtheorem{example}[theorem]{Example}

\theoremstyle{remark}

\newtheorem{remark}[theorem]{Remark}
\newtheorem{question}[theorem]{Question}

\nc{\TBW}{{\widetilde{\BW}}}
\nc{\HBW}{{\widehat{\BW}}}

\nc{\TA}{{\widetilde{A}}}
\nc{\HA}{{\widehat{A}}}

\nc{\SW}{{\operatorname{{\mathsf{SW}}}}}

\nc{\ji}{{i}}
\nc{\jo}{{o}}

\nc{\ot}{\otimes}
\nc{\wt}{\widetilde}

\title{Exceptional collections on isotropic Grassmannians}
\author{Alexander Kuznetsov}
\address{\sloppy
\parbox{0.9\textwidth}{
{\bf A.K.: }Algebra Section, Steklov Mathematical Institute,
8 Gubkin str., Moscow 119991 Russia
\hfill\\[5pt]
\phantom{{\bf A.K.: }}The Poncelet Laboratory, Independent University of Moscow
\hfill\\[5pt]
\phantom{{\bf A.K.: }}Laboratory of Algebraic Geometry, SU-HSE, 7 Vavilova Str., Moscow, Russia, 117312
\hfill
}\bigskip}
\email{akuznet@mi.ras.ru}
\author{Alexander Polishchuk}
\address{{\bf A.P.: }Department of Mathematics, University of Oregon, Eugene, OR 97405\bigskip}
\email{apolish@uoregon.edu}
\date{}
\thanks{A.K.\ was partially supported by
RFFI grants 10-01-93110, 10-01-93113, 11-01-00393, 11-01-00568, \hbox{11-01-92613-KO-a}, NSh-4713.2010.1,
the grant of the Simons foundation, and
by AG Laboratory SU-HSE, RF government grant, ag.11.G34.31.0023.
A.P.\ was partially supported by the NSF grant DMS-1001364.}

\begin{document}

%

\nc{\NB}{{\operatorname{\mathsf{B}}}}
\nc{\Forget}{{\operatorname{\mathsf{Fg}}}}
\nc{\TCU}{{\tilde{\CU}}}

\begin{abstract}
We introduce a new construction of exceptional objects in the derived category of coherent sheaves
on a compact homogeneous space of a semisimple algebraic group and show that it produces exceptional
collections of the length equal to the rank of the Grothendieck group on homogeneous spaces of all classical groups.
\end{abstract}
\maketitle

\tableofcontents

\section{Introduction}

The study of derived categories of coherent sheaves on algebraic varieties has been an
increasingly popular subject in algebraic geometry. One of important devices relevant
for this study is the notion of an exceptional collection (see \ref{ss-oec} below). In the present paper
we give a new general construction of such collections in the derived categories
of compact homogeneous spaces of semisimple algebraic groups and show
that for classical groups it gives exceptional collections of maximal length.

\subsection{An overview of exceptional collections on homogeneous varieties}\label{ss-oec}

Let $\kk$ be a base field which we assume to be algebraically closed of characteristic 0.
Recall that an object $E$ of a $\kk$-linear triangulated category $\CT$
is {\sf exceptional}, if
$$
\Ext^\bullet(E,E) = \kk
$$
(that is $E$ is simple and has no higher self-$\Ext$'s).
An ordered collection $E_1,\dots,E_m$ in $\CT$ is an {\sf exceptional collection},
if each $E_i$ is exceptional and
$$
\Ext^\bullet(E_i,E_j) = 0
$$
for all $i > j$.
Finally, an exceptional collection $E_1,\dots,E_m$ is {\sf full},
if the smallest triangulated subcategory of $\CT$ containing all
the objects $E_1,\dots,E_m$ is $\CT$ itself.

The simplest geometrical example of a full exceptional collection
is the collection
$$
\CO,\CO(1),\dots,\CO(n-1),\CO(n)
$$
in the bounded derived category $\D(\PP^n)$ of coherent sheaves on $\PP^n$ constructed by Beilinson
in his pioneering work~\cite{Bei}. After that a vast number of exceptional collections
was constructed by Kapranov in~\cite{Kap}. In fact, he constructed full exceptional collections
of vector bundles on all homogeneous spaces of the simple algebraic groups of type $A$ and on quadrics
(which are special homogeneous spaces of types $B$ and~$D$).
This naturally led to the following conjecture.

\begin{conjecture}\label{gp-ec}
If $\BG$ is a semisimple algebraic group and $\BP \subset \BG$ is a parabolic subgroup of $\BG$
then there is a full exceptional collection of vector bundles in $\D(\BG/\BP)$.
\end{conjecture}

Up to now only partial results in this direction were obtained. Below we list
all minimal homogeneous varieties of simple groups (corresponding to maximal
parabolic subgroups) for which a full exceptional collection was constructed.
Recall that simple algebraic groups are classified by Dynkin diagrams
that fall into types $A$, $B$, $C$, $D$, $E$, $F$ and $G$.
Maximal parabolic subgroups correspond to vertices of Dynkin diagrams
for which we use the standard numbering (see~\cite{Bou}). Thus, we denote by $\BP_i$
the maximal parabolic subgroup corresponding to the vertex $i$.

\begin{description}
\item[Type $A_n$] A full collection was constructed by Kapranov in~\cite{Kap}.
\item[Type $B_n$] For $\BP = \BP_1$ (so that $\BG/\BP = Q^{2n-1}$, a quadric of dimension $2n-1$)
a full exceptional collection was constructed by Kapranov in~\cite{Kap}.
For $\BP = \BP_2$ (so that $\BG/\BP = \OGr(2,2n+1)$, the Grassmannian of lines on $Q^{2n-1}$)
a full exceptional collection was constructed in~\cite{K08}.
For $n = 4$ and $\BP = \BP_4$ (so that $\BG/\BP = \OGr(4,9) = \OGr_+(5,10)$)
a full exceptional collection was constructed in~\cite{K06}.
\item[Type $C_n$] For $\BP = \BP_1$ (so that $\BG/\BP = \PP^{2n-1}$) Beilinson's collection works.
For $\BP = \BP_2$ (so that $\BG/\BP = \SGr(2,2n)$, the Grassmannian of isotropic planes in a symplectic
vector space)
a full exceptional collection was constructed in~\cite{K08}.
For $n = 3,4,5$ and $\BP = \BP_n$ (so that $\BG/\BP = \SGr(n,2n)$, the Lagrangian Grassmannian)
full exceptional collections were constructed in~\cite{Sam} and~\cite{PS}.
\item[Type $D_n$] For $\BP = \BP_1$ (so that $\BG/\BP = Q^{2n-2}$, a quadric of dimension $2n-2$)
a full exceptional collection was constructed by Kapranov in~\cite{Kap}.
For $\BP = \BP_2$ (so that $\BG/\BP = \OGr(2,2n)$, the Grassmannian of isotropic lines on $Q^{2n-2}$)
an almost full exceptional collection was constructed in~\cite{K08}.
\item[Type $E_n$] For $n = 6$ and $\BP = \BP_1$ (or $\BP = \BP_6$)
an exceptional collection was constructed by Manivel in~\cite{Man}. The collection was proved to be full in~\cite{FM}.
\item[Type $F_4$] For $\BP = \BP_4$ (so that $\BG/\BP$ is a hyperplane section
of $E_6/\BP_1$) an exceptional collection can be constructed by restricting
Manivel's collection.
\item[Type $G_2$] For $\BP = \BP_1$ (so that $\BG/\BP = Q^5$) Kapranov's collection works.
For $\BP = \BP_2$ a full exceptional collection was constructed in~\cite{K06}.
\end{description}

\subsection{The statement of results}\label{ss-res}

The main result of the present paper can be formulated as follows.
Let us say that an exceptional collection in $\D(X)$, the bounded derived category of
coherent sheaves on an algebraic variety $X$, is {\sf of expected length},
if its length is equal to the rank of the Grothendieck group $\rk(K_0(X))$.
Note that if $K_0(X)$ is a free abelian group then
this implies that the corresponding classes generate $K_0(X)$.

Let us say that a simple group $\BG$ is of type $BCD$ if its type is either $B_n$, or $C_n$, or $D_n$.

\begin{theorem}\label{main-th} Let  $\BG$ be
a simply connected simple group of type $BCD$. Then for each maximal parabolic subgroup $\BP \subset \BG$
there exists an exceptional collection of expected length in $\D(\BG/\BP)$ consisting of 
objects that have a $\BG$-equivariant structure.
\end{theorem}

Note that the existence of a $\BG$-equivariant structure here is a general result (see
\cite[Lem.\ 2.2]{Pol2}) but also comes naturally from the construction.
The $\BG$-equivariant structure on objects of our collections allows to
construct a relative exceptional collection on any fibration with fiber $\BG/\BP$
(see \cite[Thm.\ 3.1]{Sam2}).


\begin{corollary}\label{dbfib}
Let $\BG$ and $\BP$ be as in Theorem \ref{main-th}, and let $\CG \to X$ be a principal $\BG$-bundle,
where $X$ is an algebraic variety.
Consider the corresponding fibration $Y = \CG \times_\BG (\BG/\BP) \to X$.
Then there exists a semiorthogonal decomposition of $\D^b(Y)$
consisting of $\rk(K_0(\BG/\BP))$ subcategories, each equivalent to $\D^b(X)$,
and possibly an additional subcategory.
In particular, if $X$ has an exceptional collection of expected length
then so does $Y$.
\end{corollary}

Both Theorem~\ref{main-th} and Corollary~\ref{dbfib} will be proved in Section~\ref{Proofs-sec}.

Note that for an arbitrary (not maximal) parabolic subgroup $\BP \subset \BG$ the homogeneous space
$\BG/\BP$ has a structure of an iterated fibration with fibers of the form $\BG_i/\BP_i$, where
$\BG_i$ are semisimple algebraic groups and $\BP_i \subset \BG_i$ are maximal parabolic
subgroups.
Moreover, if $\BG$ is a classical group then all $\BG_i$ are classical as well. So, applying
Corollary~\ref{dbfib} (or Kapranov's construction in type $A$) several times we conclude that

\begin{corollary}\label{bcd-all}
If $\BG$ is a simple group of type $BCD$ and $\BP \subset \BG$ is
a {\rm(}not necessary maximal{\rm)} parabolic subgroup then
there exists an exceptional collection of expected length in $\D(\BG/\BP)$.
\end{corollary}

We conjecture that the exceptional collections we construct are full and
possess further nice properties
that we checked in some special cases (see Conjecture \ref{strong-and-pure}).

Finally, we would like to stress that our construction of an exceptional collection
is quite general: we use special properties of types $BCD$ only in some
computations. So, we hope that the approach of this paper
can be used to construct full exceptional collections for all the remaining homogeneous spaces
(i.e., for the exceptional groups $E_6$, $E_7$, $E_8$ and $F_4$).

\subsection{An overview of the construction}\label{ss-constr}

The main part of any construction of an exceptional collection is to find sufficiently many exceptional objects.
For a homogeneous variety it is natural to try equivariant bundles.

Note that when we fix the type of a simple group we have several choices of the group itself,
ranging from simply connected to adjoint cases. The simply connected group has the richest
category of equivariant bundles. On the other hand, the variety $\BG/\BP$ does not change if we
replace $\BG$ by its simply connected covering. 
Because of this {\em from now on we will assume that $\BG$ is simply connected}.

Recall that there is a natural equivalence of the category of $\BG$-equivariant coherent sheaves on $\BG/\BP$
with the category of representations of $\BP$:
$$
\Coh^\BG(\BG/\BP) \cong \Rep\BP,
$$
see~\cite{BK}.
In fact, it is an equivalence of tensor abelian categories.
In particular, each representation of $\BP$ can be considered as a vector bundle on $X = \BG/\BP$.
The group $\BP$ is not reductive, so its representation theory is rather complicated.
Let us start by considering the semisimple part of the category, $\Rep^{\rm ss}\BP$,
i.e., the subcategory of representations on which the unipotent
radical $\BU$ of $\BP$ acts trivially. Thus, if
$$
\BL = \BP/\BU
$$
is the Levi quotient, then extending a representation of $\BL$
to a representation of $\BP$ via the projection $\BP \to \BL$
we get an equivalence $\Rep\BL\cong\Rep^{\rm ss}\BP$.
The Levi group $\BL$ is reductive, and its weight lattice $P_\BL$ is canonically isomorphic to the weight lattice $P_\BG$ of the group $\BG$. Let us choose a maximal torus $\BT\subset\BL$ and a Borel subgroup
$\BB$ in $\BP$ containing $\BT$, such that $\BB\cap\BL$ is a Borel subgroup in $\BL$.
We denote the corresponding cones of $\BL$-dominant and $\BG$-dominant weights
by $P_\BL^+ \subset P_\BL$ and by $P_\BG^+ \subset P_\BG$, respectively.
Irreducible representations of $\BL$ are parameterized by their highest weights which are $\BL$-dominant.
For each $\BL$-dominant weight $\lambda \in P_\BL^+$ we denote by $V_\BL^\lambda$ the corresponding irreducible representation of $\BL$,
as well as its extension to $\BP$, and by $\CU^\lambda$ the corresponding $\BG$-equivariant bundle on~$X = \BG/\BP$.

In type $A$  there are sufficiently many exceptional bundles
among the $\CU^\lambda$'s, so one can construct
an exceptional collection of expected length out of them. However,
for other types the situation is not so nice.
Although all the bundles $\CU^\lambda$ are exceptional as objects of the derived category of equivariant
sheaves $\D^\BG(X)$,
it turns out that only few of them are exceptional in $\D(X)$.
For example, in the case when $\BG$ is of type $C_n$ and $\BP = \BP_n$, so that $X = \SGr(n,2n)$
(the Lagrangian Grassmannian), one can check that $\CU^\lambda$ is exceptional if and only if
$$
\lambda = \omega_i + t\omega_n,
$$
where $\omega_i$ is the fundamental weight of the vertex $i$
of the Dynkin diagram and $t\in\ZZ$. Since the canonical bundle is $\omega_X = \CU^{-(n+1)\omega_n}$, one can deduce easily that
the maximal possible length of an exceptional collection in $\D(X)$ consisting of vector bundles
of the form $\CU^\lambda$ is $n(n+1)$ (we have $n$ choices for $i$ and $n+1$ choices for $t$
in the above formula for $\lambda$), whereas $\rk(K_0(X)) = 2^n$. So, for $n \ge 5$
we have no chance to find an exceptional collection of expected length consisting only of $\CU^\lambda$.
In other words, we need to introduce another class of $\BP$-modules. In fact, this is the most interesting
problem discussed in this paper.

To explain how we do it let us return to the example of the group $\BG$ of type $C_n$ and of $\BP = \BP_n$.
Recall that in this case the lattice of weights is
$$
P_\BL = P_\BG = \ZZ^n = \{(\lambda_1,\dots,\lambda_n)\},
$$
and the dominant cones can be described as
$$
P_\BG^+ = \{ \lambda_1 \ge \lambda_2 \ge \dots \ge \lambda_n \ge 0 \},
\qquad
P_\BL^+ = \{ \lambda_1 \ge \lambda_2 \ge \dots \ge \lambda_n \}
$$
(the Levi group $\BL$ in this case is isomorphic to $\GL_n$).
Take any integer $0 \le a \le n$ and consider a subset (a {\sf block})
$$
\SB_a = \{ n \ge \lambda_1 \ge \lambda_2 \ge \dots \ge \lambda_a \ge \lambda_{a+1} = \dots = \lambda_n = a \}.
$$
Its elements can be viewed as Young diagrams inscribed in $(n-a)\times a$ rectangle.
In particular,
$$
\#\SB_a = {n \choose a}.
$$
It turns out that for the weights $\lambda,\mu$ within such a block $\SB=\SB_a$ the following
amusing property is satisfied: the canonical map
$$
\bigoplus_{\nu \in \NB} \Ext^\bullet_\BG(\CU^\lambda,\CU^\nu) \otimes \Hom(\CU^\nu,\CU^\mu)\to
\Ext^\bullet(\CU^\lambda,\CU^\mu)\eqno{(\star)}
$$
is an isomorphism (here $\Ext_\BG$ stands for the $\Ext$ groups in the derived category $\D^\BG(X)$
of $\BG$-equivariant coherent sheaves on $X$, and the map is given by the composition of equivariant
$\Ext$'s with $\Hom$'s).

As we already mentioned above, all the objects $\CU^\lambda$ are exceptional when considered as objects
of the derived category of equivariant sheaves $\D^\BG(X)$ (and in fact form an exceptional collection), while
when considered as objects of $\D(X)$ (by forgetting the equivariant structure), they are not exceptional
in general. Now, having property $(\star)$ one can formally check that
\begin{itemize}
\item considering $\{\CU^\lambda\}_{\lambda \in \SB_a}$ as a (nonfull) exceptional collection in $\D^\BG(X)$,
\item passing to the {\em right dual exceptional collection} $\{\CE^\lambda\}_{\lambda \in \SB_a}$ in $\D^\BG(X)$,
and then
\item forgetting the equivariant structure on all $\CE^\lambda$,
\end{itemize}
one obtains an exceptional collection $\{\CE^\lambda \}_{\lambda \in \SB_a}$ in the non-equivariant
category $\D(X)$ that generates the same subcategory as
the original (non-exceptional) collection $\{\CU^\lambda\}$. This strange procedure (see details in 
Section \ref{ss-eb}) can be considered
as the central construction of the paper. To make it work in general we introduce the notion
of an {\sf exceptional block}\/ $\SB \subset P_\BL^+$. By definition, an exceptional block is a subset $\SB \subset P_\BL^+$
of $\BL$-dominant weights such that the morphism $(\star)$ is an isomorphism. The procedure described above produces an exceptional
collection $\{\CE^\lambda\}_{\lambda\in\SB}$ generating the subcategory
$$
\CA_\SB := \langle \CU^\lambda \rangle_{\lambda \in \SB}.
$$

However, in general one cannot find a single exceptional block of expected length.
To obtain an exceptional collection of expected length we combine several exceptional blocks 
in a semiorthogonal sequence of blocks, i.e. in such a way that $\Ext$'s between blocks
in the order-decreasing direction vanish.
For example, for $\BG$ of type $C_n$ and $\BP = \BP_n$ we take the blocks $\SB_a$ described above for all $a$ from $0$ to~$n$.
Note that the total number of exceptional objects in the blocks $\SB_a$ is 
$\sum_{a=0}^n {n \choose a} = 2^n$, which is the expected length in this case.

\subsubsection{The choices and the restriction}\label{ss-cnr}

Now let us describe the construction in the general case. The details can be found in Section \ref{ss-ckb}.
The construction depends on several choices (subject to one restriction) that
we are going to explain now.
Let $D = D_\BG$ be the Dynkin diagram of $\BG$. Denote
by $\beta$ the simple root (a vertex of $D$) corresponding to the maximal parabolic subgroup $\BP$,
and by $\xi$ the corresponding fundamental weight of $\BG$.
$$
\parbox{.75\textwidth}{We choose a connected component of $D\setminus\beta$, called the {\sf outer component}\/ and denoted by $\Do$. We also allow $\Do$ to be empty.}\leqno{\text{(C1)}}
$$
The restriction is
$$
\parbox{.75\textwidth}{If $\Do$ is nonempty then it is a Dynkin diagram of type $A$.}\leqno{\text{(R)}}
$$
We denote the complement of $\beta$ and $\Do$ by $\Di$
$$
\Di = D_\BG \setminus (\Do \cup \beta)
$$
and call it the {\sf inner component}\/ of $D_\BG$. We consider the simply connected subgroups
$$
\BLo,\ \BLi \subset \BL
$$
corresponding to the subdiagrams $\Do,\ \Di \subset D\setminus \beta = D_\BL$ and denote by
$$
i:\BLi \to \BL,
\qquad
o:\BLo \to \BL
$$
the embeddings. Abusing the notation we denote the embeddings of these subgroups into
$\BG$ by the same letters. Our restriction on $\Do$ means that
$\BLo\simeq\SL_k$ for some $k\ge 1$.

The next choice is the following.
$$
\parbox{.75\textwidth}{We choose a standard numbering of vertices in $\Do$.}
\leqno{\text{(C2)}}
$$
Since $\Do$ is of type $A$, there are two possibilities for this choice (unless $\Do$ is empty or
consists of one vertex).
Let $b$ be the number corresponding to the vertex in $\Do$ which is adjacent to $\beta$. The chain of vertices $1,2,\dots,b$ of $D_\BG$ 
will play an importaequationnt role in the construction below.
%
%
%

The possibilities of the choice of the outer component and of the numbering of its vertices are illustrated in the following Figure.
We take the Dynkin diagram of type $E_7$, the black circle marks the vertex corresponding to the parabolic subgroup $\BP$,
the thick lines mark the outer component of the diagram. So, there are 4 choices with nonempty outer component and the fifth
choice (not illustrated on the picture) when $D_\out$ is empty).
\begin{figure*}[h]\caption{Choices of the outer component and of the numbering}
\begin{picture}(100,60)
\multiput(0,20)(20,0){6}{\circle{4}}
\put(40,0){\circle{4}}
\put(60,20){\circle*{4}}
\multiput(2,20)(20,0){5}{\line(1,0){16}}
\put(40,2){\line(0,1){16}}
\thicklines
\multiput(2,20)(20,0){2}{\line(1,0){16}}
\put(40,2){\line(0,1){16}}
\put(-2,25){$\scriptstyle{}1$}
\put(18,25){$\scriptstyle{}2$}
\put(38,25){$\scriptstyle{}3$}
\put(38,-8){$\scriptstyle{}4$}
\put(58,25){$\scriptstyle{}5$}
\put(78,25){$\scriptstyle{}6$}
\put(98,25){$\scriptstyle{}7$}
\put(20,40){$b = 3,\ k = 5$}
\end{picture}
\qquad
\begin{picture}(100,50)
\multiput(0,20)(20,0){6}{\circle{4}}
\put(40,0){\circle{4}}
\put(60,20){\circle*{4}}
\multiput(2,20)(20,0){5}{\line(1,0){16}}
\put(40,2){\line(0,1){16}}
\thicklines
\multiput(2,20)(20,0){2}{\line(1,0){16}}
\put(40,2){\line(0,1){16}}
\put(-2,25){$\scriptstyle{}4$}
\put(18,25){$\scriptstyle{}3$}
\put(38,25){$\scriptstyle{}2$}
\put(38,-8){$\scriptstyle{}1$}
\put(58,25){$\scriptstyle{}5$}
\put(78,25){$\scriptstyle{}6$}
\put(98,25){$\scriptstyle{}7$}
\put(20,40){$b = 2,\ k = 5$}
\end{picture}
\qquad
\begin{picture}(100,60)
\multiput(0,20)(20,0){6}{\circle{4}}
\put(40,0){\circle{4}}
\put(60,20){\circle*{4}}
\multiput(2,20)(20,0){5}{\line(1,0){16}}
\put(40,2){\line(0,1){16}}
\thicklines
\put(82,20){\line(1,0){16}}
\put(-2,25){$\scriptstyle{}4$}
\put(18,25){$\scriptstyle{}5$}
\put(38,25){$\scriptstyle{}6$}
\put(38,-8){$\scriptstyle{}7$}
\put(58,25){$\scriptstyle{}3$}
\put(78,25){$\scriptstyle{}2$}
\put(98,25){$\scriptstyle{}1$}
\put(20,40){$b = 2,\ k = 3$}
\end{picture}
\qquad
\begin{picture}(100,60)
\multiput(0,20)(20,0){6}{\circle{4}}
\put(40,0){\circle{4}}
\put(60,20){\circle*{4}}
\multiput(2,20)(20,0){5}{\line(1,0){16}}
\put(40,2){\line(0,1){16}}
\thicklines
\put(82,20){\line(1,0){16}}
\put(-2,25){$\scriptstyle{}4$}
\put(18,25){$\scriptstyle{}5$}
\put(38,25){$\scriptstyle{}6$}
\put(38,-8){$\scriptstyle{}7$}
\put(58,25){$\scriptstyle{}3$}
\put(78,25){$\scriptstyle{}1$}
\put(98,25){$\scriptstyle{}2$}
\put(20,40){$b = 1,\ k = 3$}
\end{picture}
\end{figure*}
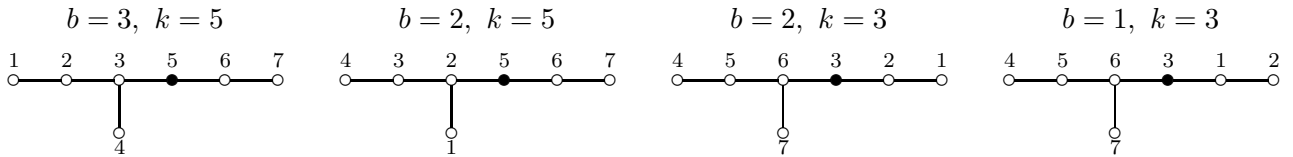

\medskip

We have the following decreasing chain of Dynkin subdiagrams in $\D_{\BG}$:
$$
D_a = D_\BG \setminus \{1,\dots,a\}
$$
for $a=0, 1, \ldots, b$ (so $D_0=D_\BG$).
Let
$$
h_a:\BH_a \to \BG
$$
be the embedding of the simply connected subgroup
corresponding to the subdiagram $D_a$.
Note that $\BLi \subset \BH_a$ since $\Di \subset D_a$, so the embedding $\ji:\BLi\to \BG$
factors through an embedding $\BLi \to \BH_a$ that we will also denote by $\ji$.
If $\BK$ is any of the groups $\BG$, $\BL$, $\BLi$ , $\BLo$, $\BH_a$ then we denote by
$P_\BK$ (resp., $\BW_\BK$)
the corresponding weight lattice (resp., Weyl group).

The third choice is the following.
$$
\parbox{.75\textwidth}{For each $a=0, 1, \ldots, b$ we choose a strictly dominant weight $\delta_a \in P_{\BH_a}^+$.}\leqno{\text{(C3)}}
$$
For each $a=0, 1, \ldots, b$ we define a polyhedron in $P_{\BH_a}\otimes {\mathbb{R}}$ by
$$
\BR_{\delta_a} = \{ \lambda \in P_{\BH_a}\otimes{\mathbb{R}}\ |\ \forall w\in\BW_{\BH_a}\quad (\lambda,w\delta_a) \le (\rho_{\BH_a},\delta_a) \},
$$
where $\rho_{\BH_a}$ is the sum of fundamental weights of $\BH_a$.
We will refer to $\BR_{\delta_a}$ as {\sf the core} in $P_{\BH_a}\otimes {\mathbb{R}}$.

\subsubsection{The indexing set}

The exceptional blocks that we construct are indexed by the set
$$
\SJ = \{ \sj \in (\theta,P_\BL)\ |\ 0 \le \sj < r \}.
$$
Here $\theta$ is the unique element of $P_\BL\otimes\QQ$ such that
$$
\theta \in \langle \omega_1,\dots,\omega_{k-1}\rangle ^\perp \cap \Ker i^*,
\qquad\text{and}\qquad
(\theta,\xi) = 1,
$$
where $\omega_t$ is the fundamental weight of the vertex $t \in D_\BG$, $i^*:P_{\BL}\to P_\BLi$ is the natural
restriction map, and
$r$ is the index of the Grassmannian $\BG/\BP$ (the integer such that 
$\CU^{-r\xi}$ is the canonical class of $\BG/\BP$). Note that the scalar 
product with $\theta$ defines a linear map $(\theta,-):P_\BL \to \QQ$, 
its image $(\theta,P_\BL)$ is a finitely generated subgroup of $\QQ$ containing $\ZZ$,
so $\SJ$ is a finite totally ordered set.


\subsubsection{The construction of the blocks}

Recall that we have a chain of subgroups
$\BH_b \subset \dots \BH_1 \subset \BH_0 = \BG$.
For each subgroup $\BH_a$ denote by $r_a$ the index of the Grassmannian $\BH_a/(\BP\cap \BH_a)$.
We prove that the sequence of integers $r_a$ is strictly decreasing
$$
r = r_0 > r_1 > \dots > r_{b-1} > r_b > r_{b+1} := 0,
$$
so it gives a subdivision of the indexing set $\SJ = \SJ_0 \sqcup \SJ_1 \sqcup \dots \sqcup \SJ_b$, where
\begin{equation*}
\SJ_a = \{ \sj \in \SJ\ |\ r-r_a \le \sj < r - r_{a+1} \}.
\end{equation*}
We denote by $a$ the function $\SJ \to \ZZ$ equal to $a$ on the subset $\SJ_a$. In the other words,
it is defined by the inequalities
$$
r - r_{a(\sj)} \le \sj < r - r_{a(\sj)+1},
$$
For brevity we will write $\BH_\sj = \BH_{a(\sj)}$ , $h_\sj = h_{a(\sj)}$ and $\BR_\sj=\BR_{\delta_{a(\sj)}}$.

Now  we are ready to describe the blocks.
First, we construct for each $\sj\in\SJ$ a block $\SB_\sj$ of the form
$$
\SB_\sj = \SBo_\sj + \sj\xi + i_*(\SBi_\sj) \subset P_\BL
$$
with $\SBo_\sj \subset \Ker h_\sj^* = \langle \omega_1,\omega_2,\dots,\omega_{a(\sj)-1}\rangle$ (called the {\sf outer part of the block}),
and $\SBi_\sj \subset P_\BLi$ (called the {\sf inner part of the block}).
The inner part is given by
$$
\SBi_\sj = \left\{ \nu \in P_\BLi^+\ \left|\
\begin{array}{ll}
(1) & \rho_{\BH_\sj} \pm 2 i_* (w \nu) \in \BR_{\sj}\ \text{for all $w \in \BW_{\BLi}$}\\
(2) & \sj\xi + i_*\nu \in P_\BL
\end{array}
\right\}\right.
$$
%
%
%
and then the outer part is defined by
$$
\SBo_\sj = \left\{ \mu \in \Ker h_\sj^* \cap P_\BG^+\ \left|\
\begin{array}{l}
\rho_{\BH_\sj} - h_\sj^*(w_\BLo\mu) - i_*(w_\BLi\nu) + i_*(w'_\BLi\nu') \in \BR_{\sj}\\
\qquad\text{for all $\nu,\nu' \in \SBi_\sj$, $w_\BLo \in \BW_\BLo$, and $w_\BLi,w'_\BLi \in \BW_\BLi$}
\end{array}
\right\}\right..
$$
Note that by definition of $\theta$ we have
$(\theta,\SB_\sj) = \sj$. So, the pairing with $\theta$ gives the ordering
of the blocks.

We check that the blocks $\SB_\sj$ constructed above are exceptional
provided the group $\BG$ is of type $BCD$
(for other types one has to modify slightly the definition of the outer part, 
see details in section~\ref{sbap}).
It follows that for each $\sj\in\SJ$ the subcategory
$\langle \CU^\lambda\ \rangle_{\lambda \in \SB_\sj}$
is generated by an exceptional collection.

\subsubsection{Modification of the blocks and the main result}

It turns out that the subcategories $\langle \CU^\lambda\ \rangle_{\lambda \in \SB_\sj}$ are not semiorthogonal, so we
have to make our blocks slightly smaller.
Let $\BR^*_{\sj}$ denote the interior of the core $\BR_{\sj}$. We define the subsets
$\SBSi_\sj\subset \SBi_\sj$ for $\sj\in\SJ$ recursively (starting from $\sj=0$) by
$$
\SBSi_\sj = \left\{ \nu \in \SBi_\sj\ \left|\
\begin{array}{l}
\text{for all $\sj' < \sj$, $\nu' \in \SBSi_{\sj'}$, and $w_\BLi,w'_\BLi \in \BW_\BLi$}\\[2pt]
\text{one has $\rho_{\BH_{\sj'}} - (\sj - \sj')\xi - w_\BLi\ji_*\nu + w'_\BLi\ji_*\nu' \in \BR^*_{\sj'}$}
\end{array}
\right\}\right..
$$
Then we set
$$
\SBSo_\sj = \left\{ \lambda_0 \in \SBo_\sj\ \left|\
\begin{array}{l}
\text{for all $\sj' < \sj$, $\nu \in \SBSi_{\sj}$, $\nu' \in \SBSi_{\sj'}$, $w_\BLi,w'_\BLi \in \BW_\BLi$, and $w_\BL \in \BW_\BL$}\\[2pt]
\text{one has $\rho_{\BH_{\sj'}} - h_{\sj'}^*(w_\BL\lambda_0 + (\sj - \sj')\xi) - w_\BLi\ji_*\nu + w'_\BLi\ji_*\nu' \in \BR^*_{\sj'}$}
\end{array}
\right\}\right.
$$
and, as before,
$$
\SBS_\sj = \SBSo_\sj + \sj\xi + \ji_*\SBSi_\sj.
$$
The subcategories 
$$
\CA_\sj = \langle \CU^\lambda \rangle_{\lambda \in \SBS_\sj},
$$
generated by these smaller blocks, are semiorthogonal.


This construction looks intimidating. However, we show in Section \ref{ss-ed} that the definition of the blocks
$\SBSo_\sj$ and $\SBSi_\sj$ can be rewritten in terms of simple inequalities, and in Section \ref{ss-ex}
we describe these blocks for classical groups.

Here is a more precise version of our main result.
Note that $\SBSo_\sj$ is a set of linear combinations of fundamental weights $\omega_1,\dots,\omega_{a(\sj)}$
with nonnegative coefficients. These can be considered as Young diagrams --- a weight
$x_1\omega_1 + \dots + x_a\omega_a$ corresponds to the Young diagram with $x_i$ columns of length~$i$.
Let us say that the set $\SBSo_\sj$ is {\sf closed under passing to Young subdiagrams}
if the corresponding set of Young diagrams is.

\begin{theorem}\label{mainthm}
$(i)$ 
Let $\BG$ be a simple simply connected group. For any choices {\rm(C1)}, {\rm(C2)}, {\rm(C3)} subject to the restriction {\rm(R)} 
the collection of subcategories
%
$$
\{\CA_\sj\}_{\sj \in \SJ}
$$
constructed above is semiorthogonal.

\noindent
$(ii)$ For $\sj\in\SJ$ such that $\SBSo_\sj$ is closed under passing to Young subdiagrams,
the block $\SBS_\sj$ is exceptional.

\noindent
$(iii)$ If $\BG$ is a group of type $BCD$ then the choices {\rm(C1)}, {\rm(C2)}, {\rm(C3)} can be made in such a way 
that the assumption of $(ii)$ is satisfied for all $\sj\in\SJ$ and the resulting exceptional collection
\begin{equation}\label{main-exc-coll-eq}
\{\CE^\lambda\}_{\lambda \in \SBS_\sj,\ \sj\in\SJ}
\end{equation}
in $\D(X)$ is of expected length.
\end{theorem}

We will describe explicit choices in Theorem~\ref{mainthm}$(iii)$ in Section \ref{ss-ex} along with
the explicit description of the blocks $\SBS_\sj$. The Theorem is proved in~Section~\ref{Proofs-sec}.
%
%
%
Note that Theorem~\ref{main-th} follows from this.


\begin{conjecture}\label{full}
The exceptional collections constructed in the Theorem~\ref{mainthm}$(iii)$ are full.
\end{conjecture}

\begin{remark}\label{nv-full}
We conjecture that in fact every exceptional collection of expected length on $\BG/\BP$ is full.
%
The more general Nonvanishing Conjecture of \cite{K09} stating that every exceptional collection of expected
length is full turned out to be false --- counterexamples were constructed in~\cite{BBS},
\cite{AO}, \cite{GS}, \cite{BBKS}. Nevertheless, we believe that the conjecture is still true for
homogeneous spaces.
\end{remark}

\subsubsection{Properties of the exceptional collection}

Recall that an exceptional collection $E_1,\dots,E_m$ in a triangulated category $\CT$ is {\sf strong}, if
$$
\Ext^{\ne 0}(E_i,E_j) = 0
$$
for all $i,j$. An advantage of a full strong exceptional collection is that it gives
an equivalence of the category $\CT$ with the derived category of modules
over an algebra $\End(\oplus E_i)$ (for a non-strong collection one has to deal with a DG-algebra).
Let us say that an exceptional collection is {\sf pure}\/
if all $E_i$ are vector bundles.

\begin{theorem}\label{sp-equiv}
For the blocks of the collections constructed in Theorem~$\ref{mainthm}(i)$ strongness and purity are equivalent.
\end{theorem}

The proof will be given in Proposition \ref{E-vec-bun-prop}.
In fact, we conjecture that

\begin{conjecture}\label{strong-and-pure}
The collections constructed in Theorem~$\ref{mainthm}(iii)$ are pure and their blocks are strong.
\end{conjecture}

We verify this conjecture for all maximal isotropic Grassmannians (symplectic and orthogonal).

\subsection{Further questions}\label{ss-qu}

There are several questions to be investigated.

\begin{question}\label{qu-choices}
Is there a way to make choices (C1)--(C3) in a canonical way?
Is restriction (R) really necessary for the construction?
\end{question}

It seems that our constructions should work under
a certain weakening of the restriction (R) which would allow to construct
many interesting exceptional collections even for $\D^b(\Gr(k,n))$.
In particular, it would allow to construct an exceptional collection in $\D^b(\Gr(k,2k))$
which is invariant under the outer automorphism. For more details see section~\ref{s-typea}.



\begin{question}\label{qu-ef}
Assume that $\BG$ is an exceptional group (types $E_6$, $E_7$, $E_8$ and $F_4$).
Is it possible to make the choices (C1)--(C3) in a way analogous to Theorem~\ref{mainthm}(iii),
so as to get an exceptional collection of expected length?
\end{question}

Note that in the case of groups of type $BCD$  the equality of the length of the constructed
collection with the rank of the Grothendieck group is a result of direct calculation
without an a priori explanation. It would be nice to understand the combinatorics
behind this coincidence. Recall that the rank of the Grothendieck group $K_0(\BG/\BP)$
is equal to $|\BW_\BG/\BW_\BL|$, so the following question seems natural.

\begin{question}\label{qu-comb}
Find a decomposition of the set $\BW_\BG/\BW_\BL = \bigsqcup_{\sj \in \SJ} W_\sj$
and a bijection between the sets $W_\sj$ and the sets $\SBS_\sj$.
\end{question}

The above decomposition should depend on a chain of subgroups
$\BH_b \subset \ldots \BH_1 \subset \BH_0 = \BG$.

\begin{question}\label{qu-char}
What happens with our exceptional collections in positive characteristic?
\end{question}

For the case of the Grassmannians of type $A$ this was studied in \cite{BLV}.


\subsection{The structure of the paper}\label{ss-str}

We start by collecting in section~\ref{s-prelim} the notation and basic facts about representation
theory of algebraic groups.

In section~\ref{ss-eb} we define exceptional blocks, prove that they produce exceptional
collections, investigate their properties, and state a criterion of exceptionality of a block.

In section~\ref{s-snp} 
we discuss strongness and purity of the collection obtained from an exceptional block.

In section~\ref{ss-ckb} we define the blocks $\SB_{\sj}$
and $\SBS_{\sj}$ and show that $(\CA_\sj)$
is a semiorthogonal collection of subcategories.


In section~\ref{s-np} 
we verify the first part of the exceptionality criterion from section~\ref{ss-eb}
--- the invariance condition --- for the blocks $\SB_{\sj}$ and $\SBS_{\sj}$.

In section~\ref{ss-aw} we verify the second part of the criterion --- the compatibility condition ---
modulo a technical assumption (that the outer part of each block is closed under passing to 
Young subdiagrams).

In section~\ref{ss-ed} we rewrite the definition of the blocks in a
more explicit form.

In section~\ref{ss-ex} we write down the precise choices for classical groups
and prove that they give exceptional collections of expected length.

Finally, in the Appendix (section~\ref{ss-kp}) we prove a certain property of representations of the general
linear group which is used for the proof of the exceptionality of the blocks.


\subsection{Acknowledgements}\label{ss-ackn}

A.K. is very grateful to Misha Finkelberg and Serezha Loktev for numerous patient explanations of representation theory. We also thank the anonimous referee for many useful remarks which in particular helped us to
revise the proof of the key technical Proposition in the Appendix.


\section{Preliminaries}\label{s-prelim}

\subsection{Notation}\label{ss-not}

\begin{enumerate}
\item Groups
\begin{itemize}
\item $\BG$, a simple simply connected algebraic group;
\item $\BP \subset \BG$, a maximal parabolic subgroup;
\item $\BU \subset \BP$, the unipotent radical;
\item $\BL = \BP/\BU$, the Levi quotient, there is also an embedding $\BL \subset \BP \subset \BG$;
\item $\BLi \subset \BL$, the inner part of $\BL$, see section~\ref{s-oi};
\item $\BLo \subset \BL$, the outer part of $\BL$, see section~\ref{s-oi};
\item $\BH_a \subset \BG$, $\BLi \subset \BH_a$, a semisimple subgroup, see section~\ref{s-oi};
\item $\BM_a = \BL \cap \BH_a$, the Levi of $\BH_a$;
\item $\BM_{a,\inn} = \BLi \cap \BH_a = \BLi$, the inner part of the Levi of $\BH_a$;
\item $\BM_{a,\out} = \BLo \cap \BH_a$, the outer part of the Levi of $\BH_a$;
\end{itemize}

\item Roots, weights
\begin{itemize}
\item $D = D_\BG$, $D_\BLi = \Di \subset D$, $D_\BLo = \Do \subset D$, $D_{\BH_a} = D_a \subset D$, the Dynkin diagrams;
\item $Q_\BG$, $Q_\BL$, $Q_\BLi$, $Q_\BLo$, $Q_{\BH_a}$, the root lattices;
\item $Q_\BG^+$, $Q_\BL^+$, $Q_\BLi^+$, $Q_\BLo^+$, $Q_{\BH_a}^+$, the cones generated by simple roots;
\item $P_\BG$, $P_\BL = P_\BL$, $P_\BLi$, $P_\BLo$, $P_{\BH_a}$, the weight lattices;
\item $P_\BG^+ \subset P_\BG$, $P_\BL^+ \subset P_\BL$, $P_\BLo^+ \subset P_\BLo$, $P_\BLi^+ \subset P_\BLi$, $P_{\BH_a}^+ \subset P_{\BH_a}$, the dominant cones;
\item $\alpha_i$, the simple roots;
\item $\omega_i$, the fundamental weights;
\item $\beta$, the simple root corresponding to the maximal parabolic $\BP$;
\item $\xi$, the fundamental weight corresponding to the maximal parabolic $\BP$;
\item $\rho = \rho_\BG = \sum_{i \in D_\BG} \omega_i \in P_\BG$;
\item $\rho_{\BH_a} = \sum_{i \in D_a} \omega_i \in P_{\BH_a}$;
\item $(-,-)$, the scalar product on the root/weight lattices;
\end{itemize}

\item Weyl groups
\begin{itemize}
\item $\BW_\BG$, $\BW_\BL$, $\BW_\BLi$, $\BW_\BLo$, $\BW_{\BH_a}$, the Weyl groups;
\item $s_\alpha$, $s_i = s_{\alpha_i}$, $s_\beta$, the simple reflections corresponding to simple roots;
\item $\ell:\BW \to \ZZ_{\ge 0}$, the length function on a Weyl group;
\item $w_0^\BG$, $w_0^\BL$, $w_0^\BLi$, $w_0^\BLo$, $w_0^{\BH_a}$, the longest elements in the corresponding Weyl groups;
\item $\SR_\BG^\BL$, the set of special representatives of left $\BW_\BL$-cosets in $\BW_\BG$, see section~\ref{s-sr};
\item $\SR_\BH^\BM$, the set of special representatives of left $\BW_\BM$-cosets in $\BW_\BH$;
\end{itemize}

\item Maps
\begin{itemize}
\item $\ji:\BLi \to \BL$, $\BLi \to \BH_a$, $\BLi \to \BG$, the natural embeddings;
\item $\jo:\BLo \to \BL$, $\BLo \to \BG$, the natural embeddings;
\item $h_a:\BH_a \to \BG$, the natural embedding;
\item $\ji^*:P_\BG \to P_\BLi$, $\jo^*:P_\BG \to P_\BLo$, $h_a^*:P_\BG \to P_{\BH_a}$, the restriction of weights;
\item $\ji_*:Q_\BLi \to Q_\BG$, $\jo_*:Q_\BLo \to Q_\BG$, ${h_a}_*:Q_{\BH_a} \to Q_\BG$, the embedding of roots;
\end{itemize}

\item Representations and bundles
\begin{itemize}
\item $V_\BG^\lambda$, the irreducible representation of $\BG$ with the highest weight $\lambda \in P_\BG^+$;
\item $V_\BL^\lambda$, the irreducible representation of $\BL$ with the highest weight $\lambda \in P_\BL^+$;
\item $\CU^\lambda$, the $\BG$-equivariant vector bundle on $\BG/\BP$ corresponding to $(V_\BL^\lambda)_{|\BP}$;
\end{itemize}

\item Other
\begin{itemize}
\item $X = \BG/\BP$, the generalized Grassmannian associated with a parabolic subgroup $\BP$;
\item $\SGr(k,2n)$ (resp., $\OGr(k,n)$), symplectic (resp., orthogonal) isotropic Grassmannian;
\item $\D(X)$, the bounded derived category of coherent sheaves on $X$;
\item $\D^\BG(X)$, the bounded derived category of $\BG$-equivariant coherent sheaves on $X$;
\end{itemize}
\end{enumerate}

\subsection{Roots and weights}\label{ss-rw}

Let $\BG$ be a simple algebraic group, $\BP$ be a maximal parabolic subgroup, $\BG/\BP = X$.
Let $\beta$ be the corresponding simple root of $\BG$ and $\xi$ the corresponding fundamental weight.

We denote by $\BU \subset \BP$ the unipotent radical of $\BP$ and by $\BL = \BP/\BU$
the Levi quotient. Recall that the projection $\BP \to \BL$ admits a splitting. We choose such a splitting
and consider $\BL$ as a subgroup of $\BP$, and hence of $\BG$. 
We also choose a maximal torus $\BT\subset\BL$ and a Borel subgroup $\BB$ in $\BP$ such that
$\BT\subset\BB$ and $\BL\cap\BB$ is a Borel subgroup in $\BB$.
Note that
the set of simple roots of $\BL$ is the complement of $\beta$ in the set of simple roots of $\BG$.

The embedding of groups $\BL \subset \BG$ induces an isomorphism
of weight lattices $P := P_\BG \stackrel{\sim}{\to} P_\BL$. We use this isomorphism
to identify the lattices. Let $P_\BL^+$ and $P_\BG^+$ denote the dominant
cones in $P$ of $\BL$ and $\BG$ respectively.

We identify simple roots of the group $\BG$ with the vertices of the Dynkin diagram $D_\BG$.
In particular, we say that simple roots $\alpha$ and $\alpha'$ are {\sf adjacent}
if the corresponding vertices are connected by an edge, or equivalently if $\alpha \ne \alpha'$
but $(\alpha,\alpha') \ne 0$.

The fundamental weight of $\BG$ corresponding to the vertex $i \in D_\BG$ is denoted by $\omega_i$.
Also, we denote by $\rho=\rho_\BG$ half the sum of simple roots of~$\BG$, or equivalently, the sum of fundamental weights.

We consider the root lattice $Q_\BG$ of $\BG$ as a sublattice of the weight lattice
(roots are weights in the adjoint representation). We denote by $(-,-)$ the scalar product
on the weight lattice. This scalar product is defined uniquely up to a multiplicative
constant. We choose the standard scaling as in \cite{Bou}. Note that with this choice all scalar
products of roots are integers and scalar products of weights are rational.

\subsection{Weyl group action}\label{ss-weyl}

The simple reflection corresponding to a root $\alpha = \alpha_i$
is denoted by $s_\alpha = s_{\alpha_i} = s_i$. Note that
\begin{equation}\label{sioj}
s_i(\omega_j) = \omega_j - \delta_{ij}\alpha_j,
\end{equation}
which means that
\begin{equation}\label{ojai}
(\omega_j,\alpha_i) = \delta_{ij}\alpha_i^2/2.
\end{equation}
It follows that
\begin{equation}\label{sirho}
s_i\rho = \rho - \alpha_i
\end{equation}
for all $i$.

We identify $\BW_\BL$ with the subgroup in $\BW_\BG$ generated by all simple reflections
$s_{\alpha_i}$ with $\alpha_i\neq\beta$.
Together with \eqref{sioj} this immediately implies the following

\begin{lemma}\label{xi-inv}
The weight $\xi$ is invariant under the action of $\BW_\BL$.
\end{lemma}

The length function on the Weyl group is denoted by $\ell$ (recall that $\ell(w)$ is the length of a minimal representation
of $w$ as a product of simple reflections). The following Lemma is well-known
(see \cite{Hum}, Lemma 10.3A and its proof).

\begin{lemma}\label{simple-positive-lem}
If $w\in \BW_\BG$ and $s_j$ is a simple reflection corresponding to the simple root $\alpha_j$
then one has $\ell(w s_j)>\ell(w)$ if and only if the root $w(\alpha_j)$ is positive.
\end{lemma}

Recall that the dominant cone $P_\BG^+$ is a fundamental domain for the action of $\BW_\BG$ on $P_\BG$.
In particular, for each $\lambda \in P_\BG$ there is an element $w \in \BW_\BG$ such that $w\lambda \in P_\BG^+$.
Moreover, such $w$ is unique unless $\lambda$ is orthogonal to a root of $\BG$ (i.e.,
unless $\lambda$ lies on a wall of a Weyl chamber).

Let us denote by $Q_\BG^+\subset P_\BG$ the cone of all linear combinations of simple roots with
nonnegative integer coefficients.
The following Lemma is also well known but we provide a proof for completeness.

\begin{lemma}\label{w-dom}
If $\lambda$ is dominant then for any $w \in \BW_\BG$ one has $\lambda - w\lambda \in Q_\BG^+$.
\end{lemma}
\begin{proof}
Since $\lambda$ is a positive linear combination of fundamental weights, it is enough to check
that for every $\omega_i$ and every $w$ the weight $\omega_i-w\omega_i$ is a sum of positive roots.
This can be checked by induction on the length of $w$. When $w$ is a simple reflection $s_j$ this follows from~\eqref{sioj}.
Let $s = s_j$ be a simple reflection, and assume $\ell(ws_j)=\ell(w)+1$. Then
$$
\omega_i-ws_j\omega_i=\omega_i-w\omega_i+w(\omega_i-s_j\omega_i) = \omega_i-w\omega_i + w(\delta_{ij}\alpha_j).
$$
Now the assertion follows from the induction assumption and from Lemma~\ref{simple-positive-lem}.
\end{proof}

The following consequence of this Lemma will be extremely important for us.

\begin{corollary}\label{max-scalar-prod-cor}
For a pair of weights $\lambda$ and $\mu$ the maximum {\rm(}resp., minimum{\rm)} of the scalar product $(w\lambda,\mu)$ 
when $w$ runs through the Weyl group $\BW$ is achieved when $w\lambda$ and $\mu$ lie in the same Weyl chamber 
{\rm(}resp., opposite Weyl chambers{\rm)}.
\end{corollary}
\begin{proof}
Since the scalar product is $\BW$-invariant, we can assume that $\mu$ is dominant. To prove the
assertion about the maximum we have to check that if $\lambda$
is also dominant then $(w\lambda,\mu) \le (\lambda,\mu)$ for any $w \in \BW$. But this follows easily from Lemma~\ref{w-dom} since the scalar product of a positive root with a dominant weight is
nonnegative by~\eqref{ojai}. The assertion about the minimum follows
as well since $(w\lambda,\mu)$ is minimal exactly when $(w\lambda,-\mu)$ is maximal.
\end{proof}

Assume that $\BH$ is a simply connected semisimple algebraic group and let $\BH = \BH_1 \times \dots \times \BH_k$ be
its decomposition into the product of simple groups. Then $P_\BH = P_{\BH_1} \oplus \dots \oplus P_{\BH_k}$
and $P_\BH^+ = P_{\BH_1}^+ \times \dots \times P_{\BH_k}^+$. Denote by $\lambda_i$ the component of a weight $\lambda \in P_\BH$
in the summand $P_{\BH_i}$.

\begin{definition}\label{def-sdom}
A weight $\lambda \in P_\BH^+$ is {\sf strictly dominant} if all its components $\lambda_i \in P_{\BH_i}^+$ are nonzero.
\end{definition}

\begin{lemma}\label{scal-dom}
If $\lambda,\mu \in P_\BH^+$ then $(\lambda,\mu) \ge 0$.
Moreover, if $\lambda$ is strictly dominant and $\mu \ne 0$ then $(\lambda,\mu) > 0$.
In particular, if $\BH$ is simple then the latter inequality holds for any pair of nonzero dominant weights.
\end{lemma}
\begin{proof}
This follows immediately from the fact that all scalar products of fundamental weights of a simple group are strictly positive.
\end{proof}

Let $\beta$ be the simple root corresponding to $\BP$. The $\BW_\BL$-orbit of $\beta$
has the following nice description.

\begin{lemma}\label{beta-orbit-lem}
The $\BW_\BL$-orbit of $\beta$ consists of all roots of $\BG$ that have the coefficient of
$\beta$ equal to $1$, when expressed as a linear combination of simple roots,
and have the same length as $\beta$.
\end{lemma}
\begin{proof} The coefficient of $\beta$ in a root $\alpha$ is given by $(\xi,\alpha)/(\xi,\beta)$,
where $\xi$ is the fundamental weight corresponding to $\beta$. Since $\xi$ is invariant
under the action of $\BW_\BL$, we have
$$
(\xi,w_\BL\beta) = (w_\BL^{-1}\xi,\beta) = (\xi,\beta)
$$
for all $w_\BL \in \BW_\BL$, so the coefficient of $\beta$ is equal to $1$ for all roots in
the $\BW_\BL$-orbit of $\beta$.

Conversely, let us check that if a positive root $\alpha$ has the coefficient of $\beta$ equal
to $1$ and $(\alpha,\alpha)=(\beta,\beta)$ then $\alpha$ is in the $\BW_\BL$-orbit of $\beta$.
Let us write $\alpha=\sum c_i \alpha_i$, where $\alpha_i$ are simple roots.
We will use induction on $\sum c_i$. If $\sum c_i=1$ then $\alpha=\beta$, so the statement
is true. Now assume that $\sum c_i>1$. It is enough to prove that there exists a simple
root $\alpha_i\neq\beta$ such that $(\alpha,\alpha_i)>0$. Indeed, then
$s_i\alpha$ will have a smaller sum of coefficients and by the induction assumption, we would deduce
that $s_i\alpha$ is in $\BW_\BL$-orbit of $\beta$. Suppose $(\alpha,\alpha_i)\le 0$ for all
$\alpha_i\neq\beta$. Then
$$
(\alpha,\alpha) =
(\beta + \sum_{\alpha_i \ne \beta} c_i\alpha_i,\alpha) =
(\beta,\alpha) + \sum_{\alpha_i \ne \beta} c_i(\alpha_i,\alpha) \le (\beta,\alpha).
$$
Since $(\alpha,\alpha) = (\beta,\beta)$ by assumption, we get $(\beta,\beta) \le (\beta,\alpha)$.
But $s_\beta(\alpha)=\alpha-2\frac{(\alpha,\beta)}{(\beta,\beta)}\beta$ should be a positive root
(since $\alpha\neq \beta$). Looking at the coefficient of $\beta$ in $s_\beta(\alpha)$ we obtain
$$2\frac{(\alpha,\beta)}{(\beta,\beta)}\le 1$$
which contradicts the previous inequality.
\end{proof}

We denote by $w_0^\BG$ and $w_0^\BL$ the longest elements of the Weyl groups $\BW_\BG$
and $\BW_\BL$ respectively. Note that
$$
(w_0^\BL)^2 = (w_0^\BG)^2 = 1.
$$
Note also that $w_0^\BG$ takes any simple root of $\BG$ to minus a simple root,
and hence any fundamental weight to minus a fundamental weight. In particular,
\begin{equation}\label{worho}
w_0^\BG\rho_\BG = - \rho_\BG
\end{equation}
and $w_0^\BG(P_\BG^+) = - P_\BG^+$, $w_0^\BL(P_\BL^+) = - P_\BL^+$.

\subsection{Representations}\label{ss-reps}

For each dominant weight $\lambda \in P_\BG^+$ (resp., $\lambda \in P_\BL^+$)
we denote by $V_\BG^\lambda$ (resp., $V_\BL^\lambda$) the corresponding irreducible
representation of $\BG$ (resp., $\BL$).

The dual of any irreducible representation is also irreducible. To be more precise we have
\begin{equation}\label{vldual}
(V_\BL^\lambda)^\vee = V_\BL^{-w_0^\BL\cdot\lambda}.
\end{equation}
Indeed, if $\lambda$ is the highest weight of an irreducible representation of $\BL$ then
$w_0^\BL\lambda$ is the lowest weight, so $-w_0^\BL\lambda$ is the highest weight of the dual.

Since the group $\BL$ is reductive the tensor product of two irreducible representations
of $\BL$ is a direct sum of irreducibles. We denote by $\mult(V_\BL^\nu,V_\BL^\lambda \otimes V_\BL^\mu)$
the multiplicity of $V_\BL^\nu$ in the tensor product. The following simple result will be useful.

\begin{lemma}\label{mult-hom}
We have
$$
\mult(V_\BL^\nu,V_\BL^\lambda \otimes V_\BL^\mu) = \dim \Hom_\BL(V_\BL^\nu,V_\BL^\lambda \otimes V_\BL^\mu) =
\dim \Hom_\BL(V_\BL^\lambda \otimes V_\BL^\mu, V_\BL^\nu).
$$
In particular,
$\mult(V_\BL^\nu,V_\BL^\lambda \otimes V_\BL^\mu) = \mult((V_\BL^\mu)^\vee,V_\BL^\lambda \otimes (V_\BL^\nu)^\vee)$.
\end{lemma}
\begin{proof}
The first part follows from the fact that there are no maps between different irreducibles
and a one-dimensional space of maps between isomorphic irreducibles. 
The second part follows from the canonical isomorphism
$\Hom_\BL(V_\BL^\nu,V_\BL^\lambda \otimes V_\BL^\mu) \cong \Hom_\BL((V_\BL^\mu)^\vee,V_\BL^\lambda \otimes (V_\BL^\nu)^\vee)$.
\end{proof}

We also need the following standard result that gives restrictions on the possible highest weights of irreducible summands of the tensor product of two irreducible representations (obtained e.g. by combining
\cite[Thm.\ 14.18]{FH} with \cite[\textsection 131, Thm.\ 5]{Zhel};
see also \cite[Exer.\ 24.12]{Hum} and \cite[Exer.\ 25.33]{FH}).

\begin{lemma}\label{tens-conv}
If $\mult(V_\BL^\nu,V_\BL^\lambda \otimes V_\BL^\mu) > 0$ then $\nu \in \Conv(\lambda + w\mu)_{w \in \BW_\BL}$,
where $\Conv$ stands for the convex hull and $\BW_\BL$ is the Weyl group of $\BL$. Similarly, if
$\mult(V_\BL^\nu,V_\BL^\lambda \otimes (V_\BL^\mu)^\vee) > 0$ then $\nu \in \Conv(\lambda - w\mu)_{w \in \BW_\BL}$.
\end{lemma}

\subsection{Special representatives}\label{s-sr}

For any $w\in\BW$ the set $w(P_\BG^+)$ belongs to a unique $\BW_\BL$-chamber, so
in the coset $\BW_\BL w\subset\BW$ of $\BW_\BL$ there is a unique representative which takes
the $\BG$-dominant cone to the $\BL$-dominant cone.
We call it the {\sf $\BL$-special representative of the coset} and denote the set of all
$\BL$-special representatives in $\BW$ by $\SR_\BG^\BL$. 
Note that the $\BW_\BL$-chamber containing $w(P_\BG^+)$ is determined by $w(\rho)$, hence,
the $\BL$-special representative $w_1$ is determined by the condition $w_1(\rho)\in P_\BL^+$.

The elements of $\SR_\BG^\BL$ can also be characterized as follows.

\begin{lemma}\label{sr-min}
The set $\SR_\BG^\BL\subset\BW$ consists of the elements that have minimal length in their
left $\BW_\BL$-cosets.
\end{lemma}
\begin{proof} Let $w\in\BW$ be an element of minimal length in its left $\BW_\BL$-coset.
Then $\ell(w^{-1}s_j) = \ell(s_jw) > \ell(w) = \ell(w^{-1})$ for every simple reflection $s_j$ in $\BW_\BL$. Hence,
by Lemma \ref{simple-positive-lem}, the root $w^{-1}(\alpha_j)$ is positive for every simple root
$\alpha_j$
that belongs to the root system of $\BL$. Thus,
$(w\rho,\alpha_j)=(\rho,w^{-1}\alpha_j)>0$ for every such simple root, i.e., $w\rho$ is $\BL$-dominant.
Hence, $w$ is a special representative.
\end{proof}


\begin{lemma}\label{srlg}
\noindent$(0)$
The only element of length $0$ in $\SR_\BG^\BL$ is $1$.

\noindent$(1)$
The only element of length $1$ in $\SR_\BG^\BL$ is $s_\beta$.

\noindent$(2)$
All elements of length $2$ in $\SR_\BG^\BL$ are of the form $s_\beta s_\alpha$, where $\alpha$ is a simple root of $\BG$ adjacent to $\beta$.
%
\end{lemma}
\begin{proof}
Part $(0)$ is clear. For $(1)$ we note that elements of length 1 in $\BW_\BG$ are just simple reflections and
for $\alpha \ne \beta$ the reflection $s_\alpha$ is in the same $\BW_\BL$-coset as $1$, which has smaller length.
Similarly, all elements of length 2 are products $s_{\alpha_1}s_{\alpha_2}$ of simple reflections.
If $\alpha_1 \ne \beta$ then $s_{\alpha_2}$ is in the same coset and has smaller length, hence $\alpha_1 = \beta$.
And if $\alpha := \alpha_2$ is not adjacent to $\beta$ then reflections $s_\alpha$ and $s_\beta$ commute, so
$s_\beta s_\alpha = s_\alpha s_\beta$ is in the same coset as $s_\beta$, which has smaller length.
%
\end{proof}


Take any reductive subgroup $\BH \subset \BG$ compatible with the torus and the Borel subgroups $\BT\subset \BB\subset \BG$,
i.e.\ such that $\BH\cap\BT\subset\BH\cap\BB$ is a maximal torus and a Borel subgroup in $\BH$,
and such that $\BM = \BH \cap \BL$ is the Levi subgroup in a parabolic subgroup of $\BH$. Let $\BW_\BH$ and $\BW_\BM$ be the corresponding
Weyl groups. Note that $\BW_\BM = \BW_\BH \cap \BW_\BL$.
It follows that $\BW_\BH/\BW_\BM \subset \BW_\BG/\BW_\BL$.
Actually, the same inclusion holds for the sets of special representatives.

\begin{lemma}\label{srlg-wh}
We have $\SR_\BG^\BL \cap \BW_\BH = \SR_\BH^\BM$.
\end{lemma}
\begin{proof}
The inclusion $\SR_\BG^\BL \cap \BW_\BH \subset \SR_\BH^\BM$ is clear.
Now let $w \in \SR_\BH^\BM$. We have to show that
$(w\rho, \alpha_i)\ge 0$, where $\alpha_i$ is any simple root of $\BL$. If $\alpha_i$ belongs
to the root system of $\BH\cap \BL=\BM$ then this follows from the definition of $\SR_\BH^\BM$.
Otherwise, the simple reflection $s_i$ associated with $\alpha_i$ is different from all simple reflections
in $W_\BH$, so $\ell(w^{-1}s_i)>\ell(w^{-1})$.
Hence, by Lemma \ref{simple-positive-lem}, $w^{-1}\alpha_i$ is a positive root, and so
$(w\rho, \alpha_i)=(\rho,w^{-1}\alpha_i)\ge 0$.
\end{proof}

The following inequality is very important for us.

\begin{lemma}\label{xil}
Assume that $v \in \SR_\BG^\BL$. Then
$$
(\xi,\rho - v\rho) \ge \ell(v)(\xi,\beta),
$$
where $\beta$ is the simple root corresponding to $\xi$. If $\ell(v) = 1$ then this inequality becomes
an equality.
\end{lemma}
\begin{proof}
Let us prove this by induction on the length of $v$. In the case $v=1$ both sides of our
inequality are equal to zero. Now assume that $\ell(v)\ge 1$.
Recall that $v$ is the representative of minimal length in the coset $\BW_\BL v$. Thus, we
can write $v=u s_i$, where $s_i$ is a simple reflection, $\ell(u)=\ell(v)-1$, and
$u\in \SR_\BG^\BL$. We have
$$(\xi,\rho-v\rho)=(\xi,\rho-u\rho)+(\xi,u(\rho-s_i\rho))=(\xi,\rho-u\rho)+(\xi,u(\alpha_i)).$$
The first summand in the right-hand side is $\ge \ell(u)(\xi,\beta)$ by the induction assumption.
Thus, it suffices to check that $(\xi, u(\alpha_i))\ge (\xi,\beta)$.

Since $\ell(u s_i)=\ell(u)+1$,
the root $u(\alpha_i)$ is positive (by Lemma \ref{simple-positive-lem}), so we only have to check that
$\beta$ appears in $u(\alpha_i)$ with nonzero coefficient, i.e., that $(\xi, u(\alpha_i))\neq 0$.
Suppose $(\xi,u(\alpha_i))=0$. Then $u(\alpha_i)=\sum_{\alpha_j\ne \beta}n_j\alpha_j$ with
$n_j\ge 0$.
The fact that $v$ has minimal length in its
right $W_\BL$-cosets implies that $\ell(v^{-1}s_j)>\ell(v^{-1})$ for every $j$ such that $\alpha_j\ne \beta$. Hence,
all the roots $v^{-1}(\alpha_j)$ are positive, and therefore,
$$-\alpha_i=s_i\alpha_i=v^{-1}u(\alpha_i)=\sum_{\alpha_j\ne \beta}n_j v^{-1}(\alpha_j)$$
should be positive, so we get a contradiction.

If $\ell(v) = 1$ then $v = s_\beta$ by Lemma~\ref{srlg}, hence $\rho - v\rho = \beta$, and both sides are equal to $(\xi,\beta)$.
\end{proof}

\begin{remark}\label{cominuscule-rem-1}
Note that if the root $\beta$ is {\sf cominuscule}, which means that the coefficient of $\beta$
in any root of $\BG$ does not exceed $1$, then
$$(\xi,\rho - v\rho) = \ell(v)(\xi,\beta)$$
for all $v \in \SR_\BG^\BL$. Indeed, in the argument above we conclude that the coefficient
of $\beta$ in $u(\alpha_i)$ is precisely $1$, hence we obtain an inductive proof of the equality.
\end{remark}

\subsection{Equivariant bundles $\CU^\lambda$ and Borel--Bott--Weil Theorem}

Since $X = \BG/\BP$ is a homogeneous variety, the category $\Coh^\BG(X)$ of $\BG$-equivariant coherent sheaves on $X$
is equivalent to the category of representations of $\BP$:
\begin{equation}\label{GP-equiv}
\Coh^\BG(X) \cong \Rep-\BP,
\end{equation}
see~\cite{BK}, \cite{Hille} and references therein.
This equivalence is compatible with the structures of tensor abelian categories on both sides,
i.e.\ it preserves tensor products and duals.

For each $\lambda \in P_\BL^+$, a dominant weight of the Levi quotient $\BL = \BP/\BU$,
we consider $V_\BL^\lambda$, the corresponding irreducible representation of $\BL$.
Extending $V_\BL^\lambda$ to $\BP$ (via the projection $\BP \to \BL$) we obtain a representation of $\BP$,
and hence a $\BG$-equivariant vector bundle on $X$ which we denote by $\CU^\lambda$.
Since the above equivalence preserves the tensor structure, we deduce from Lemma~\ref{mult-hom} and~\eqref{vldual} that
\begin{equation}\label{tp-ul}
\CU^\lambda\otimes\CU^{\lambda'} = \bigoplus_{\mu \in P_\BL^+} \Hom(V_\BL^\mu,V_\BL^\lambda\otimes V_\BL^{\lambda'}) \otimes \CU^\mu,\qquad
(\CU^\lambda)^\vee \cong \CU^{-w_0^\BL\lambda}.
\end{equation}

Note that $V_\BL^\xi$ is a one-dimensional representation of $\BL$, hence $\CU^\xi$ is a line bundle on $X$.
Moreover, it is the ample generator of $\Pic X = \ZZ$, so we will denote it by $\CO_X(1)$.
Thus,
\begin{equation}\label{oxt}
\CO_X(t) = \CU^{t\xi}.
\end{equation}
Similarly, we will denote the bundle $\CU^{\lambda + t\xi}$ by $\CU^\lambda(t)$.


The cohomology groups of bundles $\CU^\lambda$ can be computed via the Borel--Bott--Weil Theorem.
Recall that a weight $\lambda \in P_\BG$ is called {\sf $\BG$-singular}\/ if it lies on a wall
of a Weyl chamber of $\BG$ (equivalently, if it is orthogonal to some root of $\BG$).
If a weight does not lie on a wall of a Weyl chamber it is called {\sf $\BG$-regular}.
If the group $\BG$ is clear from the context we will write just singular and regular.
The sets of $\BG$-singular and $\BG$-regular weights are invariant under the natural action
of the Weyl group $\BW_\BG$ on $P_\BG$.

\begin{theorem}\label{bbw}{\em (\cite[Thm.\ IV']{Bott})}\
Take any $\lambda \in P_\BL^+ \subset P_\BL = P_\BG$.
If $\lambda + \rho_\BG$ is $\BG$-singular then $H^\bullet(X,\CU^\lambda) = 0$.
If~$\lambda + \rho_\BG$ is $\BG$-regular then there exists a unique $w \in \BW_\BG$ such that $w(\lambda + \rho_\BG)$
is dominant. In this case
$$
H^{\ell(w)}(X,\CU^\lambda) = V_\BG^{w(\lambda+\rho_\BG) - \rho_\BG}
$$
and the other cohomology groups vanish. In particular, if $\lambda$ is $\BG$-dominant then $H^0(X,\CU^\lambda) = V_\BG^\lambda$.
\end{theorem}

Let $P_\BG^\reg$ denote the set of all regular weights of $\BG$ and $P_\BG^\reg - \rho_\BG$ denote the
set of all weights $\mu \in P_\BG$ such that $\mu + \rho_\BG \in P_\BG^\reg$. Further, for each $\mu \in P_\BG^\reg - \rho_\BG$
denote by $w_\mu$ the unique element of the Weyl group $\BW_\BG$ such that $w_\mu(\mu + \rho_\BG)$ is $\BG$-dominant.
Combining the theorem above with~\eqref{tp-ul} and Lemma~\ref{tens-conv} we deduce

\begin{corollary}\label{extuu}
We have
$$
\Ext^\bullet(\CU^\lambda,\CU^{\lambda'}) =
\bigoplus_{\mu \in \Conv(\lambda' - w\lambda)_{w\in\BW_\BL} \cap P_\BL^+ \cap (P_\BG^\reg - \rho_\BG)}
\Hom(V_\BL^\mu,V_\BL^{\lambda'}\otimes V_\BL^{-w_0^\BL\lambda}) \otimes V_\BG^{w_\mu(\mu + \rho_\BG) - \rho_\BG}[-\ell(w_\mu)],
$$
where $[-\ell(w_\mu)]$ stands for cohomological shift.
\end{corollary}


We will also need a way to compute $\Ext$-groups in the derived category $\D^\BG(X)$ of $\BG$-equivariant
coherent sheaves on $X$.
Let us denote these $\Ext$ groups between $F,F' \in \D^\BG(X)$ by $\Ext_\BG^i(F,F') = \Hom_{\D^\BG(X)}(F,F'[i])$.

\begin{proposition}\label{extuu-g}
One has

\noindent$(i)$
$\Ext^i_\BG(F,F') = (\Ext^i(F,F'))^\BG$, the space of $\BG$-invariants in the $\Ext$-group between $F$ and $F'$ in $\D(X)$.

\noindent$(ii)$
$\Ext^\bullet_\BG(\CU^\lambda,\CU^{\lambda'}) = \bigoplus_{v \in \SR_\BG^\BL} \Hom(V_\BL^{v\rho - \rho},V_\BL^{\lambda'}\otimes V_\BL^{-w_0^\BL\lambda})[-\ell(v)]$.

\noindent$(iii)$
$\Ext^1_\BG(\CU^\lambda,\CU^{\lambda'}) = \Hom(V_\BL^{-\beta},V_\BL^{\lambda'}\otimes V_\BL^{-w_0^\BL\lambda})$.
\end{proposition}
\begin{proof}
(i) This follows from $\Hom_\BG(F,F') = \Hom(F,F')^\BG$ because the functor of invariants is exact
(since the group $\BG$ is reductive).

\noindent
(ii) Note that $(V_\BG^\nu)^\BG$ is zero for $\nu \ne 0$
and $\kk$ for $\nu = 0$, hence $\mu$ from the formula of Corollary~\ref{extuu} contributes to $\Ext_\BG$
if and only if $w_\mu(\mu + \rho) - \rho = 0$, that is if $\mu = v\rho - \rho$ for some $v \in \BW_\BG$.
Since $\mu$ should be $\BL$-dominant, the element $v$ should be a special representative, that is $v \in \SR_\BG^\BL$.
Of course, if $v\rho - \rho \not\in \Conv(\lambda'-w\lambda)$ then $\Hom$ is zero, so we can forget
this restriction.

\noindent
(iii) This follows from (ii) using the fact that by Lemma~\ref{srlg}(1) the only special representative
of length 1 is $s_\beta$ and $s_\beta\rho = \rho - \beta$.
\end{proof}


%

\subsection{The canonical class}\label{ss-can}

Let $\BG$ be a semisimple algebraic group (not necessarily simple).
Recall that by~\cite[sec.~1.5]{Hille}, the canonical class of $X = \BG/\BP$ is the line bundle corresponding to the weight
equal to minus the sum of all positive roots of $\BG$ which are not roots of $\BL$. The following formula is also
well known but we add a proof for completeness.

\begin{lemma}\label{canclass}
The canonical class $\omega_X$ of $X=\BG/\BP$ is isomorphic to the line bundle $\CU^{w_0^\BL w_0^\BG \rho - \rho}$.
\end{lemma}
\begin{proof}
Recall that $\rho$ is half the sum of all positive roots of $\BG$. As $w_0^\BG$ takes all positive roots of $\BG$ to negative roots
and $w_0^\BL$ takes all negative roots of $\BL$ to positive roots of $\BL$, it follows that $w_0^\BL w_0^\BG\rho$ is half the sum of all
positive roots of $\BL$ minus half the sum of all positive roots of $\BG$ which are not roots of $\BL$. So, subtracting $\rho$ we obtain
minus the sum of all positive roots of $\BG$ which are not the roots of $\BL$.
\end{proof}

We will also need the following more explicit formula.

\begin{lemma}\label{index}
Let $\beta$ be the simple root corresponding to $\BP$ and $\xi$ the corresponding fundamental weight.
There exists a maximal root in the $\BW_\BL$-orbit of $\beta$, i.e., a positive
root $\bar\beta\in \BW_\BL\beta$ satisfying
$\bar\beta-w\bar\beta\in Q_\BG^+$ for any $w\in\BW_\BL$.
Then $\omega_X = \CO_X(-r)=\CU^{-r\xi}$, where
$$
r = (\rho,\bar\beta + \beta)/(\xi,\beta).
$$
\end{lemma}
\begin{proof}
The Picard group of $\BG/\BP$ is generated by $\CU^\xi$, hence $\omega_X \cong \CU^{w_0^\BL w_0^\BG \rho - \rho} \cong \CU^{-k\xi}$
for some $k \in \ZZ$. To find $k$ we compute the scalar product with $\beta$. We get
$$
k = (\rho - w_0^\BL w_0^\BG \rho,\beta)/(\xi,\beta).
$$
Further $(- w_0^\BL w_0^\BG \rho,\beta) = (w_0^\BL \rho,\beta) = (\rho,w_0^\BL \beta)$ by~\eqref{worho}. 
Note that $\beta$ considered as a weight of $\BL$ is antidominant (its scalar products with the simple roots of $\BL$
are nonpositive), hence $w_0^\BL \beta$ is $\BL$-dominant. By Lemma~\ref{w-dom} we conclude that $\bar\beta := w_0^\BL\beta$ 
is the maximal root in the $\BW_\BL$-orbit of $\beta$. Finally, it is a positive root since
$(\xi,w_0^\BL\beta) = (w_0^\BL\xi,\beta) = (\xi,\beta) > 0$ since $\xi$ is $\BW_\BL$-invariant.
\end{proof}

\begin{remark}\label{barbeta}
By Lemma~\ref{beta-orbit-lem},
$\bar\beta$ is in fact the maximal root
of the same length as $\beta$ and with the coefficient of $\beta$ equal to $1$.
This gives a very easy way to find $\bar\beta$
just by looking at the table of roots.
\end{remark}

\begin{remark}\label{rem-ind}
The integer $r$ is called {\sf the index} of the Grassmannian $\BG/\BP$.
\end{remark}

The following consequence of the above formula is useful.

\begin{corollary}\label{rinc}
Let $\BP$ be a maximal parabolic subgroup in $\BG$ and $\beta$ the corresponding simple root.
Let $\BH \subset \BH' \subset \BG$ be a pair of semisimple subgroups corresponding to a pair of
Dynkin subdiagrams $D_\BH \subset D_{\BH'} \subset D_\BG$ such that $\beta \in D_\BH$ and there is
a simple root $\alpha \in D_{\BH'}\setminus D_\BH$ adjacent to the connected component
of $\beta$ in~$D_\BH$.
Let $r$ and $r'$ be the indices of the Grassmannians $\BH/(\BH\cap \BP)$ and $\BH'/(\BH' \cap \BP)$
respectively. Then $r' > r$.
\end{corollary}
\begin{proof}
Let $\BM = \BL \cap \BH$ and $\BM' = \BL \cap \BH'$. Let $\bar\beta$ be the maximal root in the $\BW_\BM$-orbit of $\beta$
and $\bar\beta'$ the maximal root in the $\BW_{\BM'}$-orbit of $\beta$. Let $C\subset D_\BH$ denote the connected component of $\beta$ in $D_\BH$, and let $\alpha$ be a simple root of $\BH'$
adjacent to $C$. Note that since $\bar\beta$ is maximal, the coefficient
of any simple root of $C$ in $\bar\beta$ is strictly positive.
In particular, the coefficients of simple roots in $C$ adjacent to $\alpha$ are positive, 
hence the scalar product $(\alpha,\bar\beta)$
is strictly negative. Therefore,
$$
s_\alpha(\bar\beta) = \bar\beta - 2\frac{(\alpha,\bar\beta)}{\alpha^2}\alpha
$$
has a strictly positive coefficient of $\alpha$.
Therefore, $(\rho,\bar\beta') \ge (\rho,s_\alpha(\bar\beta)) \ge (\rho,\bar\beta) + (\rho,\alpha) > (\rho,\bar\beta)$
since $(\rho,\alpha) = \alpha^2/2 > 0$. Now the assertion follows from Lemma \ref{index}.
\end{proof}

\section{Exceptional blocks}\label{ss-eb}

Let $\BG$ be a simple simply connected algebraic group and $\BP \subset \BG$ a maximal parabolic subgroup.
We take $X = \BG/\BP$ and denote by $\D(X)$ the bounded derived category of coherent sheaves on $X$
and $\D^\BG(X)$ --- the bounded derived category of $\BG$-equivariant coherent sheaves.
We denote by $\Forget:\D^\BG(X) \to \D(X)$ the forgetful functor.

We denote as usual $\Ext^i(F,F') = \Hom(F,F'[i])$, $\Ext$-groups in category $\D(X)$.
Similarly, $\Ext$-groups in the equivariant category $\D^\BG(X)$ are denoted by $\Ext_\BG^i(F,F')$.
Recall that $\Ext_\BG^i(F,F') = \Ext^i(F,F')^\BG$ by Proposition~\ref{extuu-g}(i).
Note that the forgetful functor induces a linear map
$$
\Forget:\Ext_\BG^i(F,F') \to \Ext^i(F,F').
$$

For each triple of $\BL$-dominant weights $\lambda,\mu,\nu \in P_\BL^+$ consider the map
$$
\Ext^\bullet_\BG(\CU^\lambda,\CU^\nu) \otimes \Hom(\CU^\nu,\CU^\mu)\to
\Ext^\bullet(\CU^\lambda,\CU^\mu),
$$
the composition of the action of the forgetful functor with the Yoneda multiplication.

Now we can introduce the main notion of this section.

\begin{definition}\label{def-eb}
A set of $\BL$-dominant weights $\SB \subset P_\BL^+$ is called an {\sf exceptional block}\/ if
for all $\lambda,\mu \in \SB$ the canonical map
\begin{equation}\label{he}
\bigoplus_{\nu \in \NB} \Ext^\bullet_\BG(\CU^\lambda,\CU^\nu) \otimes \Hom(\CU^\nu,\CU^\mu)\to
\Ext^\bullet(\CU^\lambda,\CU^\mu)
\end{equation}
is an isomorphism.
\end{definition}

The goal of this section is to show that for any exceptional block $\SB \subset P_\BL^+$
the category
$$
\D_\SB(X) = \langle \CU^\lambda \rangle_{\lambda \in \SB} \subset \D(X)
$$
generated in $\D(X)$ by the bundles $\CU^\lambda$ with $\lambda \in \SB$, has a full exceptional collection.

%

\subsection{The $\xi$-ordering}\label{ss-xi-ord}

Recall that $\beta$ is the simple root of $\BG$ corresponding to the maximal parabolic $\BP$
and $\xi$ is the corresponding fundamental weight. By Lemma~\ref{xi-inv} it is invariant
under the action of $\BW_\BL$.

Consider the partial ordering on the weight lattice $P_\BL$ defined by:
\begin{equation}\label{sp-ord}
\begin{array}{lll}
\lambda \prec \mu & \qquad & \text{if $(\xi,\lambda) < (\xi,\mu)$}\\
\lambda \preceq \mu & \qquad & \text{if either $\lambda \prec \mu$ or $\lambda = \mu$}
\end{array}
\end{equation}
We will call it {\sf the $\xi$-ordering}.

\begin{lemma}\label{uuord}
If $\Hom(\CU^\lambda,\CU^\mu) \ne 0$ then $\lambda \preceq \mu$.
\end{lemma}
\begin{proof}
By Corollary~\ref{extuu} if $\Hom(\CU^\lambda,\CU^\mu) \ne 0$ then there is a
non-trivial $\BL$-map $V_\BL^\kappa \subset (V_\BL^\lambda)^\vee \otimes V_\BL^\mu$
for some $\BG$-dominant weight $\kappa$. This means that there is a non-trivial $\BL$-map
$V_\BL^\kappa\otimes V_\BL^\lambda \to V_\BL^\mu$, hence
$\mu \in \Conv(\lambda + w\kappa)_{w \in \BW_\BL}$ by Lemma~\ref{tens-conv}.
But for any $w \in \BW_\BL$ we have
$$
(\xi,\lambda + w\kappa) - (\xi,\lambda) = (\xi,w\kappa) = (w^{-1}\xi,\kappa) = (\xi,\kappa) \ge 0,
$$
where the last inequality follows from Lemma~\ref{scal-dom} since both $\xi$ and $\kappa$ are $\BG$-dominant.
Moreover, since $\BG$ is simple, the inequality is strict
unless $\kappa = 0$. Thus, we see that $\lambda \prec \mu$ unless $\kappa = 0$.
But if $\kappa = 0$ then $V_\BL^\kappa\otimes V_\BL^\lambda = V_\BL^\lambda$, hence $\mu = \lambda$.
\end{proof}

Thus, we see that $\Hom$ groups between $\CU^\lambda$ in $\D(X)$ go in the direction of the $\xi$-ordering.
It turns out that $\Ext$ groups in the equivariant category go in the opposite direction!

\begin{lemma}\label{uugord}
If $\Ext^\bullet_\BG(\CU^\lambda,\CU^\mu) \ne 0$ then $\mu \preceq \lambda$.
More precisely, if $\Ext^i_\BG(\CU^\lambda,\CU^\mu)\neq 0$ then
$$(\xi,\lambda)-(\xi,\mu)\ge i(\xi,\beta),$$
and for $i=1$ this inequality becomes an equality.
Also, each bundle $\CU^\lambda$ is exceptional in $\D^\BG(X)$.
\end{lemma}
\begin{proof}
By Proposition~\ref{extuu-g} if $\Ext^i_\BG(\CU^\lambda,\CU^\mu) \ne 0$ then there is a non-trivial $\BL$-map
$V_\BL^{v\rho - \rho} \to (V_\BL^\lambda)^\vee \otimes V_\BL^\mu$
for some $v \in \SR_\BG^\BL$ with $\ell(v) = i$.
This means that there is a non-trivial $\BL$-map
$V_\BL^{v\rho - \rho}\otimes V_\BL^\lambda \to V_\BL^\mu$, hence
$\mu \in \Conv(\lambda + w(v\rho - \rho))_{w \in \BW_\BL}$ by Lemma~\ref{tens-conv}.
Now by Lemma~\ref{xil}, for any $w \in \BW_\BL$ we have
$$
(\xi,\lambda + w(v\rho - \rho)) - (\xi,\lambda) = (\xi,w(v\rho - \rho)) = (w^{-1}\xi,v\rho - \rho) = (\xi,v\rho - \rho) \le -i(\xi,\beta),
$$
where the last inequality becomes an equality for $i=1$. This implies that
$$(\xi,\mu)-(\xi,\lambda)\le -i (\xi,\beta)$$
with equality for $i=1$, as required.
Thus, we see that $\mu \prec \lambda$ unless $v = 1$.
But if $v = 1$ then $V_\BL^{v\rho - \rho}\otimes V_\BL^\lambda = V_\BL^\lambda$, hence $\mu = \lambda$.
Also, if $v=1$ then $i=\ell(v)=0$, so $\Ext^{>0}_\BG(\CU^\lambda,\CU^\lambda)=0$ and by Proposition~\ref{extuu-g}
we have $\Hom_\BG(\CU^\lambda,\CU^\lambda) = \Hom_\BL(V_\BL^\lambda,V_\BL^\lambda) = \kk$, hence $\CU^\lambda$ is exceptional in $\D^\BG(X)$.
\end{proof}

Lemma \ref{uugord} has the following important consequence.

\begin{theorem}\label{dgx-ec}
The bundles $\{\CU^\lambda\}_{\lambda \in P_\BL^+}$ ordered with respect to any total ordering extending
the opposite of the $\xi$-ordering constitute a full exceptional collection in the derived category of equivariant 
sheaves $\D^\BG(X)$.
\end{theorem}
\begin{proof}
The fact that we get an exceptional collection follows from Lemma~\ref{uugord}.
It remains to check that it is full.

Indeed, let us show that every object belongs to the triangulated subcategory generated by this collection. 
It suffices to check this
for pure objects, that is for $\BG$-equivariant coherent sheaves. As we know, the category of
$\BG$-equivariant coherent sheaves is equivalent to the category of $\BP$-representations. But each such representation has a filtration (an extension of the radical
filtration) with the quotients that are simple $\BL$-representations, i.e. correspond to bundles $\CU^\lambda$
with appropriate $\lambda \in P_\BL^+$. Thus, it is contained in the subcategory generated by the
$\CU^\lambda$.
\end{proof}

\begin{remark}\label{rem-opp}
The fact that the orderings of $\Hom$'s in $\D(X)$ and $\Ext$'s in $\D^\BG(X)$ are opposite is the reason
for the fact that an object $\CU^\lambda$ is typically not exceptional in $\D(X)$ --- one can construct a nontrivial element of
$\Ext^\bullet(\CU^\lambda,\CU^\lambda)$ by composing $\Hom$'s and equivariant $\Ext$'s. As we will see in section~\ref{ss-el} below,
the cure is, in a sense, to reverse one of the orderings.
\end{remark}


\subsection{The forgetful functor and its adjoint}\label{ss-fg-adj}

Let $\SB \subset P_\BL^+$ be an exceptional block.
Let
$$
\D_\SB^\BG(X) = \langle \CU^\lambda \rangle_{\lambda \in \SB}
$$
denote the subcategory of $\D^\BG(X)$ generated by $\CU^\lambda$ with $\lambda$ in $\SB$.
Since the collection $\{\CU^\lambda\}_{\lambda \in \SB}$ is exceptional,
the category $\D_\SB^\BG$ is saturated (see~\cite{BK89}), hence the forgetful functor $\Forget:\D_\SB^\BG(X) \to \D_\SB(X)$
has a right adjoint functor which we denote by $\Forget^!:\D_\SB(X) \to \D_\SB^\BG(X)$ 
(cf.\ \cite[Prop.\ 2.7]{BK89}).

The crucial observation is the following

\begin{proposition}\label{fshr}
If $\SB$ is an exceptional block then
$$
\Forget^!(\Forget(\CU^\mu)) = \bigoplus_{\nu \in \SB} \Hom(\CU^\nu,\CU^\mu)\otimes\CU^\nu,
$$
where $\Hom(\CU^\lambda,\CU^\mu)$ are considered just as vector spaces, not as representations of $\BG$.
\end{proposition}
\begin{proof}
Let
$$
\TCU^\mu := \bigoplus_{\nu \in \SB} \Hom(\CU^\nu,\CU^\mu)\otimes\CU^\nu \in \D_\SB^\BG(X).
$$
We have a canonical evaluation map $\ev:\Forget(\TCU^\mu) \to \Forget(\CU^\mu)$ in $\D(X)$.
By adjunction it gives a map $\TCU^\mu \to \Forget^!\Forget(\CU^\mu)$. Let us show it is an isomorphism.
For this let us check that the induced map
$$
f:\Ext^\bullet_\BG(\CU^\lambda,\TCU^\mu) \to \Ext^\bullet_\BG(\CU^\lambda,\Forget^!\Forget(\CU^\mu))
$$
is an isomorphism for all $\lambda \in \SB$.
Indeed, we have a commutative diagram
$$
\xymatrix{
\hbox{\raisebox{0pt}[0pt][0pt]{$\displaystyle\bigoplus_{\nu\in\SB}$}} \Ext^\bullet_\BG(\CU^\lambda,\CU^\nu) \otimes \Hom(\CU^\nu,\CU^\mu) \ar@{=}[r] \ar[d]^{\Forget\otimes 1} &
\Ext^\bullet_\BG(\CU^\lambda,\TCU^\mu) \ar[r]^-f \ar[d]^\Forget & 
\Ext^\bullet_\BG(\CU^\lambda,\Forget^!\Forget(\CU^\mu)) \ar@{=}[d] \\
\hbox{\raisebox{0pt}[0pt][0pt]{$\displaystyle\bigoplus_{\nu\in\SB}$}} \Ext^\bullet(\CU^\lambda,\CU^\nu) \otimes \Hom(\CU^\nu,\CU^\mu) \ar@{=}[r] &
\Ext^\bullet(\Forget(\CU^\lambda),\Forget(\TCU^\mu)) \ar[r]^\ev & 
\Ext^\bullet(\Forget(\CU^\lambda),\Forget(\CU^\mu))
}
$$
The composition of the left vertical map with the maps in the bottom row is the map of~\eqref{he}
which is an isomorphism since $\SB$ is an exceptional block. Hence, the map $f$ in the top row
is an isomorphism as well.

It follows that the cone of the map $\TCU^\mu \to \Forget^!\Forget(\CU^\mu)$ is orthogonal to all $\CU^\lambda$ in $\D_\SB^\BG(X)$.
But $\CU^\lambda$ generate this category, hence the cone is zero.
\end{proof}

\begin{question}\label{qu-fgadj}
It is interesting to find a general formula for $\Forget^!$ (or maybe for $\Forget^!\circ\Forget$).
\end{question}

\subsection{Exceptional bundles $\CE^\lambda$}\label{ss-el}
The crucial step is to replace the exceptional collection $\CU^\lambda$ in $\D_\SB^\BG(X)$
by its {\sf right dual exceptional collection} (see~\cite{B}).

Recall that if $(E,F)$ is an exceptional pair in a triangulated category $\CT$
then the {\sf right mutation}\/ $\BR_F(E)$ is defined as the (shifted) cone
$$
\BR_F(E) := \Cone(\xymatrix@1{E \ar[rr]^-{\text{{\sf coev}}} && \Hom^\bullet(E,F)^\vee\otimes F})[-1],
$$
It is well known that $(F,\BR_F(E))$ is also an exceptional pair which generates the same subcategory
in $\CT$ as the initial pair $(E,F)$.

Now assume that $E_1,\dots,E_n$ is an exceptional collection.
Its right dual collection is defined as the collection obtained by a sequence of right mutations:
$$
(E_n,\BR_{E_n}E_{n-1},\BR_{E_n}\BR_{E_{n-1}}E_{n-2},\dots,\BR_{E_n}\cdots\BR_{E_2}E_1).
$$
This collection is exceptional and generates the same subcategory as the initial collection.
Note that the composition of mutations $\BR_{E_n}\ldots \BR_{E_{n-i}}$ 
depends only on the subcategory generated by $E_n,\ldots,E_{n-i}$, so we denote it by
$\BR_{\langle E_n,\ldots,E_{n-i}\rangle}$. 

Now we apply this construction to the exceptional collection $(\CU^\lambda)_{\lambda \in \SB}$ (with 
respect to some total ordering extending the opposite of the $\xi$-ordering)
in the derived category of equivariant sheaves $\D^\BG(X)$ and denote by
\begin{equation}\label{def-el}
\CE_\SB^\lambda := \BR_{\langle \CU^\mu \rangle_{\{\mu \in \SB\ |\ \mu \prec \lambda\}}}\CU^\lambda,
\end{equation}
the objects of the right dual collection (as this formula indicates,
$\CE_\SB^\lambda$ does not depend on a choice of the total ordering). 
Further on we will frequently drop the index $\SB$ in the notation $\CE^\lambda_\SB$
if it is clear which block $\SB$ is considered.

By definition, the objects $\CE^\lambda$ are exceptional in the derived category of equivariant sheaves.
Our goal now is to show that the objects $\Forget(\CE^\lambda)$ in the usual derived category $\D(X)$
are also exceptional and moreover form a full exceptional collection in $\D_\SB(X)$.

First of all, recall that the standard property of the right dual exceptional collections gives
\begin{equation}\label{geu}
\Ext^\bullet_\BG(\CE^\lambda,\CU^\mu) = \begin{cases} \kk, & \text{for $\lambda = \mu$}\\0, & \text{otherwise}\end{cases}
\end{equation}
(see e.g.~\cite{B}).
Also, it follows from the construction of the dual collection that the subcategories
both in $\D^\BG(X)$ and $\D(X)$ generated by objects $\CE^\mu$ and $\CU^\mu$ coincide:
\begin{equation}\label{ut1}
\langle \CE^\mu \rangle_{\mu \preceq \lambda} = \langle \CU^\mu \rangle_{\mu \preceq \lambda},
\end{equation}
and moreover, for each $\lambda$ there is a morphism $\CE^\lambda \to \CU^\lambda$ such that
\begin{equation}\label{ut2}
\Cone(\CE^\lambda \to \CU^\lambda) \in \langle \CU^\mu \rangle_{\mu \prec \lambda}.
\end{equation}

\begin{corollary}\label{cor-eu}
For all $\lambda,\mu \in \SB$ we have
\begin{equation}\label{etou}
\Ext^\bullet(\Forget(\CE^\lambda),\CU^\mu) = \Hom(\CU^\lambda,\CU^\mu).
\end{equation}
\end{corollary}
\begin{proof}
Indeed, by Proposition~\ref{fshr} we have
$$
\Ext^\bullet(\Forget(\CE^\lambda),\CU^\mu) \cong
\Ext^\bullet_\BG(\CE^\lambda,\Forget^!(\CU^\mu)) \cong
\Ext^\bullet_\BG(\CE^\lambda,\bigoplus_{\nu \in \SB} \Hom(\CU^\nu,\CU^\mu)\otimes\CU^\nu).
$$
Now note that by~\eqref{geu} we have
$\Ext^\bullet_\BG(\CE^\lambda,\CU^\nu) = 0$  unless $\lambda = \nu$.
Thus, the RHS equals $\Hom(\CU^\lambda,\CU^\mu)$.
\end{proof}

\begin{proposition}\label{el-fec}
For an exceptional block $\SB$
the objects $\Forget(\CE^\lambda)$ form a full exceptional collection in $\D_\SB(X)$
with respect to any total ordering extending the $\xi$-ordering.
\end{proposition}
\begin{proof}
First, take $\mu \prec \lambda$. By~\eqref{etou} and Lemma~\ref{uuord} we have $\Ext^\bullet(\Forget(\CE^\lambda),\CU^\mu) = 0$.
Then~\eqref{ut1} implies $\Ext^\bullet(\Forget(\CE^\lambda),\Forget(\CE^\mu)) = 0$ as well.
On the other hand, using this semiorthogonality and~\eqref{ut2} we deduce that
$\Ext^\bullet(\Forget(\CE^\lambda),\Forget(\CE^\lambda)) \cong \Ext^\bullet(\Forget(\CE^\lambda),\CU^\lambda) = \Hom(\CU^\lambda,\CU^\lambda) = \kk$,
so each $\Forget(\CE^\lambda)$ is exceptional. Finally, the fullness of the collection $\{\Forget(\CE^\lambda)\}_{\lambda \in \SB}$
in $\D_\SB(X)$ follows from~\eqref{ut1}.
\end{proof}

From now on to unburden the notation we will denote $\Forget(\CE^\lambda)$ simply by $\CE^\lambda$.

\subsection{Properties of exceptional blocks}\label{ss-ebprops}

Let $\SB$ be any subset of $P_\BL^+$ and $\mu \in P_\BL^+$. Denote
$$
\SB + \mu = \{ \lambda + \mu\ |\ \lambda \in \SB \}.
$$

\begin{lemma}\label{sb+xi}
If $\SB$ is an exceptional block then for each $t \in \ZZ$ the block $\SB + t\xi$ is exceptional.
Moreover, $\CE^{\lambda + t\xi}_{\SB + t\xi} = \CE^\lambda_\SB(t)$.
\end{lemma}
\begin{proof}
Recall that $\CU^{t\xi} = \CO_X(t)$ and twisting by this bundle takes $\CU^\lambda$ to $\CU^{\lambda + t\xi}$.
Since such a twisting is an autoequivalence, it follows that it preserves the exceptionality of a block.
\end{proof}

Let us say that a subset $\SB' \subset \SB$ is {\sf closed with respect to decreasing in the $\xi$-ordering},
if for any $\lambda,\mu \in \SB$ if $\lambda \in \SB'$ and $\mu \preceq \lambda$ then $\mu \in \SB'$.

\begin{lemma}
Let $\SB$ be an exceptional block and $\SB' \subset \SB$ be a subset closed with respect to decreasing in the
$\xi$-ordering.
Then $\SB'$ is an exceptional block. Moreover, $\CE^\lambda_{\SB'} = \CE^\lambda_\SB$ for all $\lambda \in \SB'$.
\end{lemma}
\begin{proof}
Take $\lambda,\mu \in \SB'$ and consider the map~\eqref{he}. It is an isomorphism since $\SB$ is exceptional.
On the other hand, $\nu \in \SB$ contributes to the LHS only if $\Ext^\bullet_\BG(\CU^\lambda,\CU^\nu) \ne 0$
which by Lemma~\ref{uugord} implies that $\nu \prec \lambda$. But then $\nu \in \SB'$ since $\SB'$ is closed
with respect to decreasing in $\xi$-ordering. Thus, the LHS of~\eqref{he} coincides with the LHS of analogous map
written for the block $\SB'$, hence $\SB'$ is exceptional.

An isomorphism between $\CE^\lambda_{\SB'}$ and $\CE^\lambda_\SB$ follows immediately from the definition~\eqref{def-el}.
\end{proof}

\subsection{The output set and the criterion of exceptionality}

In this section we give a criterion for a block $\SB$ to be exceptional
in terms of the Weyl group action on weights and the representation theory of~$\BL$.
We start with some preparations.

\begin{lemma}
Let $\mu \in P_\BL^+ \cap (P_\BG^\reg - \rho)$. Then there exists a unique pair $(\kappa,v)$,
where $\kappa \in P_\BG^+$ and $v \in \BW_\BG$ such that
$$
\mu = v(\kappa + \rho) - \rho.
$$
Moreover, $v \in \SR_\BG^\BL$.
\end{lemma}
\begin{proof}
The existence and uniqueness of the pair $(\kappa,v)$ follow from the regularity of $\mu + \rho$.
And since $\mu \in P_\BL^+$ we conclude that $v \in \SR_\BG^\BL$.
\end{proof}

Using this simple observation we can rewrite the formula of Corollary~\ref{extuu} as follows:
$$
\Ext^\bullet(\CU^\lambda,\CU^{\lambda'}) = \bigoplus_{\kappa \in P_\BG^+,\ v \in \SR_\BG^\BL\ |\ v(\kappa+\rho) - \rho \in \Conv(\lambda' - w\lambda)_{w \in \BW_\BL}}
\Hom(V_\BL^{v(\kappa+\rho)-\rho},V_\BL^{\lambda'}\otimes V_\BL^{-w_0^\BL\lambda}) \otimes V_\BG^\kappa [-\ell(v)].
$$
It is clear from this formula that it is convenient to have a control
over the set of all pairs $(\kappa,v)$ which can appear in the RHS.
So, we define the {\sf output set} for the pair of weights $\lambda,\lambda'$ of $\BL$ as
$$
\OP(\lambda,\lambda') = \{ (\kappa,v) \in P_\BG^+ \times \SR_\BG^\BL\ |\ v(\kappa + \rho) - \rho \in \Conv(\lambda' - w\lambda)_{w\in \BW_\BL} \}.
$$
Consequently, we define the {\sf output set} of a block $\SB$ to be
$$
\OP(\SB) = \bigcup_{\lambda,\lambda' \in \SB} \OP(\lambda,\lambda') \subset P_\BG^+ \times \SR_\BG^\BL,
$$
and we denote by $\OP_1(\SB) \subset P_\BG^+$ and $\OP_2(\SB) \subset \SR_\BG^\BL$
the projections of $\OP(\SB)$ to $P_\BG^+$ and $\SR_\BG^\BL$ respectively, so that
$$
\OP(\SB) \subset \OP_1(\SB) \times \OP_2(\SB).
$$
Using these definitions we can rewrite the formula of Corollary~\ref{extuu} as follows:
\begin{equation}\label{extuu2}
\Ext^\bullet(\CU^\lambda,\CU^{\lambda'}) = \bigoplus_{(\kappa,v) \in \OP(\lambda,\lambda')}
\Hom(V_\BL^{v(\kappa+\rho)-\rho},V_\BL^{\lambda'}\otimes V_\BL^{-w_0^\BL\lambda}) \otimes V_\BG^\kappa [-\ell(v)].
\end{equation}
Note that we can extend the area of summation 
in the above formula. Indeed,
if for a pair $(\kappa,v)$ one has
$v(\kappa+\rho)-\rho\not\in \Conv(\lambda' - w\lambda)_{w \in \BW_\BL}$ then
$\Hom(V_\BL^{v(\kappa+\rho)-\rho},V_\BL^{\lambda'}\otimes V_\BL^{-w_0^\BL\lambda}) = 0$
by Lemma \ref{tens-conv},
and we have no contribution. So, we can replace $\OP(\lambda,\lambda')$ by $\OP(\SB)$,
or even by $\OP_1(\SB)\times\OP(\SB_2)$.


Also, for each set of $\BL$-dominant weights $S \subset P_\BL^+$ denote by $\Pi_S:\Rep\BL \to \Rep\BL$
the projector onto the subcategory formed by all $V_\BL^\nu$ with $\nu \in S$. In other words, $\Pi_S$
is a functor such that
$$
\Pi_S(V_\BL^\lambda) = \begin{cases} V_\BL^\lambda, & \text{if $\lambda \in S$}\\0, & \text{otherwise}\end{cases}
$$

\begin{proposition}\label{crit-eb}
Assume that a subset $\SB \subset P_\BL^+$ has the following two properties:
\begin{enumerate}\renewcommand{\theenumi}{\alph{enumi}}
\item for all $\kappa \in \OP_1(\SB)$, $v \in \OP_2(\SB)$ we have $v\kappa = \kappa$;

\item for all $\kappa \in \OP_1(\SB)$, $v \in \OP_2(\SB)$, $\lambda \in \SB$ the canonical map
\begin{equation}\label{pinb}
\Pi_{\SB}(V_\BL^{\kappa + v\rho -\rho}\otimes V_\BL^{\lambda})\to
\Pi_{\SB}(V_\BL^\kappa \otimes \Pi_{\SB}(V_\BL^{v\rho - \rho} \otimes V_\BL^{\lambda}))
\end{equation}
is an isomorphism.
\end{enumerate}
Then the block $\SB$ is exceptional.
\end{proposition}

In what follows we will refer to part (a) of this criterion as the {\sf invariance condition},
and to part (b) as the {\sf compatibility condition}.

\begin{proof}
Fix a pair of weights $\lambda, \lambda' \in \SB$.
We have to check that the map~\eqref{he} (with $\mu=\lambda'$) is an isomorphism.

We start by rewriting~\eqref{extuu2} in a more convenient form.
First of all, we extend the summation area to $\OP_1(\SB)\times\OP_2(\SB)$
(as was mentioned above, this does not spoil the equality).
Next, we use the isomorphism
$$\Hom(V_\BL^{v(\kappa+\rho)-\rho},V_\BL^{\lambda'}\otimes V_\BL^{-w_0^\BL\lambda})\simeq
\Hom(V_\BL^{\lambda'},V_\BL^{v(\kappa+\rho)-\rho}\otimes V_\BL^{\lambda})^\vee\simeq
\Hom(V_\BL^{\lambda'},\Pi_\SB(V_\BL^{v(\kappa+\rho)-\rho}\otimes V_\BL^{\lambda}))^\vee,$$
where the second isomorphism follows from the condition $\lambda' \in \SB$.
Finally, by the invariance condition we have
$v(\kappa+\rho) - \rho=\kappa + v\rho - \rho$.
Thus, we obtain
\begin{equation}\label{rhs}
\Ext^\bullet(\CU^{\lambda},\CU^{\lambda'}) =
\bigoplus_{\kappa \in \OP_1(\SB),\ v \in \OP_2(\SB)}
\Hom(V_\BL^{\lambda'},\Pi_\SB(V_\BL^{\kappa+v\rho-\rho}\otimes V_\BL^{\lambda}))^\vee \otimes V_\BG^\kappa[-\ell(v)],
\end{equation}

Now specializing~\eqref{rhs} we can obtain an expression for $\Ext_\BG$ and $\Hom$ in the LHS of~\eqref{he}.
To obtain an expression for $\Ext_\BG$ we should restrict to the case $\kappa = 0$.
Replacing also $\lambda'$ by $\nu \in \SB$ we obtain
\begin{equation}\label{lhs1}
\Ext_\BG^\bullet(\CU^{\lambda},\CU^\nu) = \bigoplus_{v \in \OP_2(\SB)}
\Hom(V_\BL^\nu,\Pi_{\SB}(V_\BL^{v\rho-\rho}\otimes V_\BL^{\lambda}))^\vee[-\ell(v)].
\end{equation}

On the other hand, to obtain an expression for $\Hom$ we should restrict to $v = 1$.
Replacing also $\lambda$ by $\nu$ we obtain
\begin{equation}\label{lhs2}
\Hom(\CU^\nu,\CU^{\lambda'}) = \bigoplus_{\kappa \in \OP_1(\SB)}
\Hom(V_\BL^{\lambda'},\Pi_{\SB}(V_\BL^\kappa\otimes V_\BL^\nu))^\vee \otimes V_\BG^\kappa.
\end{equation}

Combining~\eqref{lhs1} with~\eqref{lhs2} we rewrite the LHS of~\eqref{he} as
\begin{multline*}
\bigoplus_{\nu \in \SB}\Ext_\BG^\bullet(\CU^{\lambda},\CU^\nu) \otimes \Hom(\CU^\nu,\CU^{\lambda'}) = \\
\bigoplus_{\nu \in \SB,\ \kappa \in \OP_1(\SB),\ v \in \OP_2(\SB)}
\Hom(V_\BL^\nu,\Pi_{\SB}(V_\BL^{v\rho-\rho}\otimes V_\BL^{\lambda}))^\vee \otimes
\Hom(V_\BL^{\lambda'},\Pi_{\SB}(V_\BL^\kappa\otimes V_\BL^\nu))^\vee \otimes V_\BG^\kappa[-\ell(v)] = \\
\bigoplus_{\kappa \in \OP_1(\SB),\ v \in \OP_2(\SB)}
\Hom(V_\BL^{\lambda'},\Pi_{\SB}(V_\BL^\kappa\otimes \Pi_{\SB}(V_\BL^{v\rho-\rho}\otimes V_\BL^{\lambda})))^\vee \otimes V_\BG^\kappa[-\ell(v)],
\end{multline*}
where the second equality follows from the formula
$$
\Pi_{\SB}(V_\BL^{v\rho-\rho}\otimes V_\BL^{\lambda}) =
\bigoplus_{\nu \in \SB} \Hom(V_\BL^\nu,\Pi_{\SB}(V_\BL^{v\rho-\rho}\otimes V_\BL^{\lambda}))^\vee \otimes V_\BL^\nu.
$$

To conclude we compare the obtained expression for the LHS of~\eqref{he} with the expression~\eqref{rhs} for the RHS
and note that the map from the LHS of~\eqref{he} to the RHS is induced by the map~\eqref{pinb}.
Thus, if the compatibility property~(b) holds then the map is an isomorphism, hence the block $\SB$ is exceptional.
\end{proof}

\section{On strongness and purity}\label{s-snp}

Note that a priori the exceptional objects $\CE^\lambda$ constructed above are complexes.
However, we have the following

\begin{conjecture}\label{evb}
For any exceptional block $\SB \subset P_\BL^+$ and any $\lambda \in \SB$ the object $\CE^\lambda$ is a vector bundle.
\end{conjecture}

Note that the standard $t$-structure on $\D^\BG(X)$ restricts to a $t$-structure on
the category $\D_\SB^\BG(X)$ whose heart $\CC_\SB$ consists of $\BG$-equivariant coherent
sheaves that are obtained by successive extensions from $\CU^\lambda$ with $\lambda\in\SB$.
As was already mentioned before, the category of $\BG$-equivariant coherent sheaves on $X$ is equivalent to
the category of finite-dimensional representations of $\BP$, which in turn is equivalent to the category
of finite-dimensional representations of a certain infinite quiver with relations $(\CQ,\CI)$  (see \cite{Hille}).
Recall that the vertices of $\CQ$ are in bijection with the set $P^+_\BL$ of dominant weights of $\BL$,
and there is an arrow $\lambda\to\mu$
if and only if $V^{\mu}_\BL$ appears in  $V^{-\beta}_\BL\otimes V^\lambda_\BL$ (i.e., when there is
a nontrivial $\Ext^1_\BG(\CU^{\lambda},\CU^\mu)$).
Note that by Lemma~\ref{uugord}, this quiver is leveled by the function
\begin{equation}\label{quiver-level-function}
\w(\lambda)=-(\xi,\lambda)/(\xi,\beta),
\end{equation}
which means that for every arrow $\lambda\to\mu$ one has $s(\mu)=s(\lambda)+1$.
The subcategory $\CC_\SB$ corresponds to the subcategory of representations supported at the
vertices $\SB\subset P^+_\BL$. Hence, it is equivalent to the category of finite-dimensional representations
$\kk[\CQ_\SB]/\CI_\SB$, where $\kk[\CQ_\SB]$ is the path-algebra
of the full subquiver $\CQ_\SB\subset\CQ$ corresponding to the set of vertices $\SB$, $\CI_\SB$ is an
ideal of relations.

We refer to \cite{ARS} for an introduction to
quivers and representation theory of finite-dimensional algebras. In particular, we use the notion
of a projective cover $P$ of a simple object $S$ corresponding to an object (such $P$ is an indecomposable
projective object with a surjective map $P\to S$).

\begin{proposition}\label{E-vec-bun-prop}
The following conditions are equivalent:

\noindent
$(i)$ Each $\CE^\lambda$ for $\lambda\in\SB$ is a vector bundle.

\noindent
$(ii)$ For each $\lambda\in\SB$, $\CE^\lambda$ is a projective cover of $\CU^\lambda$ in the category $\CC_\SB$.

\noindent
$(iii)$ The natural map $\Ext^\bullet_{\CC_\SB}(\CU^\lambda,\CU^\mu)\to\Ext^\bullet_\BG(\CU^\lambda,\CU^\mu)$
is an isomorphism for any $\lambda,\mu\in\SB$.

\noindent
$(iv)$ The exceptional collection $(\CE^\lambda)_{\lambda\in\SB}$ in $\D(X)$ is strong.

\noindent
Furthermore, under these conditions the canonical map $\CE^\lambda\to \CU^\lambda$ induces an 
isomorphism
\begin{equation}\label{strong-exc-hom-eq}
\Hom(\CU^\lambda,\CE^\mu)\stackrel{\sim}{\to}\Ext^\bullet(\CE^\lambda,\CE^\mu),
\end{equation}
where $\lambda,\mu\in\SB$.
\end{proposition}
\begin{proof}
(i)$\Rightarrow$(ii). If the objects $\CE^\lambda$ are vector bundles then they belong to $\CC_\SB$. Furthermore,
since $\CC_\SB$ is a heart of a t-structure of a full subcategory $\D^\BG_\SB(X)$ of $\D^\BG(X)$ we have
$\Ext^1_{\CC_\SB}(\CE^\lambda,\CU^\mu)\simeq\Ext^1_\BG(\CE^\lambda,\CU^\mu)=0$ for $\lambda,\mu\in \SB$.
This implies that $\Ext^1_{\CC_\SB}(\CE^\lambda,\CF)=0$ for any $\CF$ in $\CC_\SB$,
i.e., $\CE^\lambda$ is projective.
\medskip

\noindent
(ii)$\Rightarrow$(i). If $\CE^\lambda$ is a projective cover of $\CU^\lambda$ in $\CC_\SB$ then $\CE^\lambda$
itself is an object of $\CC_\SB$, hence a successive extension of $\CU^\mu$ with $\mu \in \SB$. In particular, it is a vector bundle on $X$.
\medskip

\noindent
(ii)$\Rightarrow$(iii). Using (ii) we can construct for any object $\CF$ in $\CC_\SB$
a projective resolution consisting of direct sums of objects $\CE^\lambda$.
Computing $\Ext^\bullet_{\CC_\SB}(\CF,\CU^\mu)$ using such a resolution and using the isomorphisms
$$\Hom_{\CC_\SB}(\CE^\lambda,\CU^\mu)\simeq\Ext^\bullet_\BG(\CE^\lambda,\CU^\mu)$$
(coming from the assumption that $\CE^\lambda$ is a projective cover of $\CU^\lambda$ and from
the orthogonality relations \eqref{geu}),
we derive that the map $\Ext^\bullet_{\CC_\SB}(\CF,\CU^\mu)\to\Ext^\bullet_\BG(\CF,\CU^\mu)$ is
an isomorphism.
\medskip

\noindent
(iii)$\Rightarrow$(ii).
For $\lambda\in\SB$ let $\CP^\lambda\to\CU^\lambda$ be the projective cover of $\CU^\lambda$
in $\CC_{\SB}$. The condition (iii) implies that the natural map
$$
\Ext^\bullet_{\CC_\SB}(\CP^\lambda,\CU^\mu)\to\Ext^\bullet_\BG(\CP^\lambda,\CU^\mu) 
$$
is an isomorphism. It follows that $(\CP^\lambda)$ satisfies the orthogonality relations \eqref{geu},
characterizing the right dual exceptional sequence to $(\CU^\lambda)_{\lambda\in\SB}$, so we get
$\CP^\lambda\simeq \CE^\lambda$ for $\lambda\in\SB$.
\medskip

\noindent
(i)$\Rightarrow$(iv). The condition \eqref{he} implies that the natural map
$$\bigoplus_{\nu \in \NB} \Ext^\bullet_\BG(A,\CU^\nu) \otimes \Hom(\CU^\nu,B)\to
\Ext^\bullet(A,B)$$
is an isomorphism for any $A,B\in \CC_\SB$. Applying this to $A=\CE^\lambda$, $B=\CE^\mu$,
where $\lambda,\mu\in\SB$, and
using \eqref{geu}, we derive the isomorphism \eqref{strong-exc-hom-eq}, which implies
that the exceptional collection $(\CE^\lambda)$ is strong.
\medskip

\noindent
(iv)$\Rightarrow$(i). Choose any ordering of $\CU^\lambda$ compatible with the partial ordering $\prec$.
Let $\CU_p$ denote the $p$-th object for this ordering. Let $\CE_p$ be the objects
of the dual collection. Then for any $\CF \in \D_\SB^\BG(X)$ there is a spectral sequence
$\Ext^q_\BG(\CE_p[p],\CF)\otimes \CU_p \Rightarrow H^{q-p}\CF$. For $\CF = \CE^\lambda$
this spectral sequence implies (i).
\end{proof}

Now we are going to prove two criteria for the equivalent conditions of Proposition~\ref{E-vec-bun-prop} to hold.

\begin{proposition}\label{path-closed-prop}
Assume that the subquiver
$\CQ_\SB\subset\CQ$ contains entirely any path in $\CQ$ 
that starts and ends in $\CQ_\SB$. Then the equivalent
conditions of Proposition~$\ref{E-vec-bun-prop}$ hold.
\end{proposition}
\begin{proof}
Recall that the projective cover of a simple object of a vertex $\lambda$
is the representation of $(\CQ_\SB,\CI_\SB)$ associating with a vertex $\mu \in \SB$
the vector space generated by all paths in the quiver from the vertex $\lambda$ to~$\mu$
(modulo the relations). The condition of the Proposition ensures that this representation
is isomorphic to the restriction to $\CQ_\SB$ of the projective cover of the simple object
of the vertex $\mu$ in the category of representations of $\CQ$. It follows that $\Hom$'s
from projective objects to simple objects in $\CQ_\SB$ are the same as in $\CQ$,
and moreover, the restrictions to $\CQ_\SB$ of projective resolutions of simple
objects in $\CQ$ give projective resolutions in $\CQ_\SB$. Combining all this we deduce
that $\Ext$'s between simple objects in $\CC_\SB$ are isomorphic to those in $\CC = \Coh^\BG(X)$,
i.e.\ the condition (iii) of Proposition~\ref{E-vec-bun-prop} holds.
\end{proof}

Also, the properties of purity and strongness of the collection $\CE^\lambda$ are related to the Koszulity
 of a certain algebra. We refer to~\cite{BGS}, \cite{PP} for basic facts about Koszul algebras.

\begin{proposition}\label{Koszul-alg-prop}
$(i)$ Assume that the graded algebra
$$A_\SB=\bigoplus_{\lambda,\mu\in\SB}\Ext_\BG^\bullet(\CU^\lambda,\CU^\mu)$$
is Koszul (with respect to the cohomological grading).
Then the equivalent conditions of Proposition~$\ref{E-vec-bun-prop}$ hold.

\noindent
$(ii)$ If the algebra $A_\SB$ is one-generated then
Koszulity of $A_\SB$ is equivalent to the conditions of Proposition~$\ref{E-vec-bun-prop}$.
\end{proposition}
\begin{proof} $(i)$ This follows from the main result of \cite[Cor.\ 8]{Posic}
(see also the proofs of Theorems 4.1 and 4.2
in \cite{Pol-perv}).

\medskip

\noindent
$(ii)$ If the condition (iii) of Proposition \ref{E-vec-bun-prop} is satisfied then
$A_\SB$ is isomorphic (as a graded algebra) to the $\Ext$-algebra
between simple objects in the abelian category $\CC_\SB$. Thus, the assumption
that $A_\SB$ is one-generated implies that the function \eqref{quiver-level-function} is a Koszul weight
function on the set of simple objects of $\CC_\SB$, i.e., $\Ext^i_{\CC_\SB}(\CU^{\lambda},\CU^\mu)\neq 0$
only for $\w(\mu)-\w(\lambda)=i$. Thus, by \cite[Prop.\ 2.1.3]{BGS}, the algebra 
$\kk[\CQ_\SB]/\CI_\SB$ is Koszul, hence, its Yoneda algebra $A_\SB$ is also Koszul.
\end{proof}

\begin{remark} In the case when the unipotent radical of $\BP$ is abelian (in this case the Grassmannian
$X = \BG/\BP$ is called {\sf cominuscule}) and the subquiver $\CQ_\SB\subset\CQ$ contains entirely
any path that starts and ends in $\CQ_\SB$, the algebra $A_\SB$ is Koszul as follows from
the main result of \cite{Hille} and from Proposition \ref{path-closed-prop}.
\end{remark}

\begin{remark}
In \ref{lag-purity-sec} we will give an example (Example~\ref{eee}) of an exceptional block
for which Proposition~\ref{path-closed-prop} does not apply, and at the same time the inequality
of Lemma~\ref{uugord} becomes strict in some cases (and so the algebra $A_\SB$ is not one-generated)
and so Proposition~\ref{Koszul-alg-prop} does not apply as well,
but the equivalent conditions of Proposition \ref{E-vec-bun-prop} still hold.
See Conjecture~\ref{ainfty} for a possible explanation of this.
\end{remark}





\section{Constructing exceptional blocks}\label{ss-ckb}

In this section we suggest a construction of an exceptional block
which depends on a choice of a semisimple subgroup $\BH \subset \BG$.
We start with some preparation.

\subsection{Cores}\label{ss-cores}

Let $\BH$ be a semisimple group. 
Let $\delta \in P_\BH^+$ be a strictly dominant weight
(see Definition~\ref{def-sdom}).

\begin{definition}\label{def-core}
The polyhedron
\begin{equation}\label{core}
\BR_\delta = \{ \lambda \in P_\BH \otimes \RR\ |\ \forall w\in \BW_\BH\ (w\delta,\lambda) \le (\delta,\rho_\BH) \}
\end{equation}
is called {\sf the core} of shape $\delta$.
\end{definition}

We will denote by
\begin{equation}\label{intcore}
\BR_\delta^* := \{ \lambda \in P_\BH \otimes \RR \ |\ \forall w\in \BW_\BH\ (w\delta,\lambda) < (\delta,\rho_\BH) \}.
\end{equation}
the interior of the core $\BR_\delta$.
Note that both $\BR_\delta$ and $\BR_\delta^*$ are $\BW_\BH$-invariant and convex.


\begin{lemma}\label{dom-plan}
The intersection of a core with the set of dominant weights is given by
$$
\BR_\delta \cap P_\BH^+ = \{ \lambda\in P_\BH^+\ |\ (\delta,\lambda) \le (\delta,\rho_\BH) \}.
$$
Similarly,
$$
\BR_\delta^* \cap P_\BH^+ = \{ \lambda\in P_\BH^+\ |\ (\delta,\lambda) < (\delta,\rho_\BH) \}.
$$
\end{lemma}
\begin{proof} Let us check the first equality (the second is proved analogously).
By definition, the LHS is contained in the RHS. On the other hand, since both $\lambda$ and $\delta$ are $\BH$-dominant,
by Corollary~\ref{max-scalar-prod-cor}, we have $(w\delta,\lambda) \le (\delta,\lambda)$ for all $w \in \BW_\BH$,
hence the RHS is contained in the LHS.
\end{proof}

We will say that a point of $P_\BH\otimes\RR$ is {\sf integral} if it lies in the weight lattice $P_\BH \subset P_\BH\otimes\RR$.

\begin{lemma}\label{core-sing}
All integral points of $\BR_\delta^*$ are $\BH$-singular.
All $\BH$-regular integral points of the core $\BR_\delta$ are contained in the $\BW_\BH$-orbit of $\rho_\BH$.
\end{lemma}
\begin{proof}
Assume that $\lambda \in P_\BH \cap \BR_\delta$ is $\BH$-regular.
Take $w \in \BW_\BH$ such that $w\lambda$ is $\BH$-dominant. Then
$w\lambda \in \BR_\delta \cap P_\BH^+$ and since $w\lambda$ is $\BH$-regular
we can write $w\lambda = \rho_\BH + \mu$, $\mu \in P_\BH^+$. Therefore,
$$
(\delta,\rho_\BH+\mu) = (\delta,w\lambda) = (w^{-1}\delta,\lambda) \le (\delta,\rho_\BH),
$$
hence $(\delta,\mu) \le 0$. Since $\delta$ is strictly dominant, this implies $\mu = 0$ by Lemma~\ref{scal-dom},
hence $\lambda = w^{-1}\rho_\BH$.
\end{proof}


\subsection{The setup}\label{s-oi}

Consider the complement $D_\BG \setminus \beta$ of the vertex $\beta$ of the Dynkin diagram $D_\BG$ of $\BG$.
In general it consists of several (up to 3) connected components of different types.
We choose one component of type $A$ (possibly empty) to be called {\sf the outer component} and denote it by $\Do$.
The union of the other components will be called {\sf the inner component} and denoted by $\Di$.
We denote the corresponding connected semisimple groups by $\BLo$ and $\BLi$ and by
$$
\jo:\BLo \to \BL,
\qquad
\ji:\BLi \to \BL
$$
the canonical embeddings.
Abusing the notation we will also denote by $\jo$ (resp., $\ji$) the embedding of $\BLo$
(resp., $\BLi$) into $\BG$.
Note that the groups $\BLo$ and $\BLi$ are simply connected (this follows from the fact that
an embedding of Dynkin diagrams induces a surjection of the weight lattices).
In particular, we have
\begin{equation}\label{a-ass}
\BLo \cong \SL_k
\end{equation}
for some $k \ge 1$. We fix a numbering of the vertices of $D = D_\BG$ as follows.
First, we number the vertices of the outer part $\Do = A_{k-1}$ by integers from $1$ to $k-1$ in a standard way
(if $k \ge 3$ there are two ways to number the vertices of $D_\out$, see section~\ref{ss-cnr} for an illustration).
Then we number the vertex $\beta$ by $k$ and the remaining vertices in an arbitrary way.
We denote by $b$ the number of the vertex in $\Do$ which is adjacent to $\beta$ (note that
such vertex is unique).

Note that we have the following decomposition of the Weyl group of $\BL$:
$$
\BW_\BL = \BW_\BLo \times \BW_\BLi.
$$
(since $\Do$ and $\Di$ are not adjacent, the corresponding simple reflections commute).

Now consider the chain of subdiagrams
$$
D_b \subset D_{b-1} \subset \dots \subset D_1 \subset D_0 = D_\BG,
\qquad
D_a = D_\BG \setminus \{1,\dots,a\}.
$$
Let
$$
\BH_b \subset \BH_{b-1} \subset \dots \subset \BH_1 \subset \BH_0 = \BG
$$
be the corresponding chain of semisimple subgroups of $\BG$. For $a=0,\ldots,b$ we denote by
$$h_a:\BH_a \to \BG
$$
the embedding. Note that any $\BH_a$ contains $\BLi$.
Abusing the notation we will denote the corresponding embedding by $\ji:\BLi\to \BH_a$.

For each $a=0,\ldots,b$ we choose a strictly dominant weight $\delta_a \in P_{\BH_a}^+$
(in the sense of Definition~\ref{def-sdom} --- note that the Dynkin diagram $D_a$ can be disconnected,
so the group $\BH_a$ can be nonsimple)
and consider the corresponding core $\BR_{\delta_a} \subset P_{\BH_a}\otimes \RR$.
To unburden the notation we denote this core by $\BR_a$.
The interior of this core will be denoted by $\BR^*_a$.
%

Let $r$ be the index of $\BG/\BP$ and let $r_a$ be the index of $\BH_a/(\BH_a \cap \BP)$.
Note that by Corollary~\ref{rinc} we have
$$
0 < r_b < r_{b-1} < \dots < r_1 < r_0 = r
$$




\subsection{The indexing set}

Let us denote by $\theta$ an element of $P_\BL\otimes\QQ$ such that
\begin{equation}\label{theta}
\theta \in \langle \omega_1,\dots,\omega_{k-1}\rangle ^\perp \cap \Ker i^*,
\qquad\text{and}\qquad
(\theta,\xi) = 1.
\end{equation}
Since $\omega_1,\dots,\omega_{k-1},\xi$ form a basis of $\Ker i^* \subset P_\BL\otimes\QQ$,
such $\theta$ exists and unique.
Note that the set $(\theta,P_\BL)$ of all scalar products of $\theta$
with weights of $\BL$ is a cyclic subgroup of $\QQ$ containing $\ZZ$.
We consider the intersection of this subgroup with the half-closed interval $[0,r) \subset \QQ$:
$$
\SJ = \{ \sj \in (\theta,P_\BL)\ |\ 0 \le \sj < r    \}.
$$
This set will number the blocks in the collection.
Note that it is naturally linearly ordered.
The blocks will be shown to be semiorthogonal with respect to this order.

For each $\sj \in \SJ$ there is a unique integer $a(\sj)$ in the interval $0 \le a(\sj) \le b$ such that
\begin{equation}\label{aj}
r-r_{a(\sj)} \le \sj < r-r_{a(\sj) + 1},
\end{equation}
where we set $r_{b+1}=0$.
To unburden the notation we will write $\BH_\sj = \BH_{a(\sj)}$ , $h_\sj = h_{a(\sj)}$ and
$\BR_\sj=\BR_{\delta_{a(\sj)}}$.

Below we will need the following simple observation.

\begin{lemma}\label{existsp}
For any $\nu \in P_\BLi$ there is a rational number $p \in (\theta,P_\BL)$ such that $p\xi + \ji_*\nu \in P_\BL$.
\end{lemma}
\begin{proof}
Since any $\nu$ is a linear combination of fundamental weights, it suffices to consider the case
of $\nu = \ji^*\omega_t$ for some $t \in \Di$. Then it is clear that $\ji_*\nu = \ji_*\ji^*\omega_t$
is just the orthogonal projection of $\omega_t$ onto the subspace $\ji_*(P_\BLi\otimes\QQ) \subset P_\BL\otimes\QQ$.
Its orthogonal complement is generated by the lattice $Q_\BLo$
and by the weight $\xi$. 
Moreover, $\omega_t$ is orthogonal to
the lattice $Q_\BLo$ since $t \in \Di$. Hence,
$$
\ji_*\ji^*\omega_t = \omega_t - \frac{(\omega_t,\xi)}{\xi^2}\xi.
$$
It remains to check that $(\omega_t,\xi)/\xi^2 \in (\theta,P_\BL)$. For this we apply the linear function $(\theta,-)$
to the above equality. Since $\theta$ is orthogonal to the image of $i_*$, we conclude that
$(\omega_t,\xi)/\xi^2 = (\theta,\omega_t) \in (\theta,P_\BL)$.
\end{proof}

%

\subsection{The first approximation}\label{sbpab}

For each element of the indexing set $\sj \in \SJ$ we will define a subset $\HSB_\sj \subset P_\BL^+$.
We will show that this is an exceptional block if $\BG$ is of type $BCD$.
In other cases we will have to replace $\HSB_\sj$ by an appropriate
smaller subset $\SB_\sj$.

First, we define the inner part as
\begin{equation}\label{sbij}
\HSBi_\sj = \left\{ \nu \in P_\BLi^+\ \left|\
\begin{array}{ll}
(1) & \rho_{\BH_\sj} \pm 2 i_* (w \nu) \in \BR_{\sj}\ \text{for all $w \in \BW_{\BLi}$}\\
(2) & \sj\xi + i_*\nu \in P_\BL
\end{array}
\right\}\right..
\end{equation}
After that we define the outer part as
\begin{equation}\label{sboj}
\HSBo_\sj = \left\{ \mu \in \Ker h_\sj^* \cap P_\BG^+\ \left|\
\begin{array}{l}
\rho_{\BH_\sj} - h_\sj^*(w_\BLo\mu) - i_*(w_\BLi\nu) + i_*(w'_\BLi\nu') \in \BR_{\sj}\\
\ \text{for all $\nu,\nu' \in \HSBi_\sj$, $w_\BLo \in \BW_\BLo$, and $w_\BLi,w'_\BLi \in \BW_\BLi$}
\end{array}
\right\}\right..
\end{equation}
And finally, we consider the set
\begin{equation}\label{bbplus}
\HSB_\sj = \HSBo_\sj + \sj\xi + \ji_*(\HSBi_\sj).
\end{equation}

\begin{remark}\label{w-nu-rem}
In both definitions \eqref{sbij} and \eqref{sboj} we can replace all terms
of the form $i_*(w\nu)$ with $w\in \BW_\BLi$ by $wi_*(\nu)$ and allow $w$ to run through the entire
group $\BW_\BL$. Indeed, this follows from the decomposition $\BW_\BL=\BW_\BLo\times\BW_\BLi$
together with the fact that $\BW_\BLo$ acts trivially on the image of $i_*$.
\end{remark}

\subsection{Very special representatives}\label{sbap}

In this section we will define a certain class of elements of the set $\SR_\BH^\BM$
and using them define a subblock $\SB_\sj \subset \HSB_\sj$. In fact, $\SB_\sj = \HSB_\sj$
if $\BG$ is of type $BCD$.

Recall that $\BLo = \SL_k$, see~\eqref{a-ass}.
We will use the following representation of the weight lattice of $\SL_k$:
$$
P_{\SL_k} = \{ (\lambda_1,\dots,\lambda_k) \in \QQ^k\ |\ \text{$\lambda_i - \lambda_{i+1} \in \ZZ$ for all $1\le i \le k-1$ and $\sum_{i=1}^k \lambda_i = 0$} \},
$$
where the simple roots and the fundamental weights are given by
$$
\alpha_t = \Big(\underbrace{0,\dots,0}_{\text{$t-1$}},1,-1,\underbrace{0,\dots,0}_{\text{$k-t-1$}}\Big),
\qquad
\omega_t = \Big(\underbrace{\frac{k-t}k,\dots,\frac{k-t}k}_{\text{$t$}},\underbrace{-\frac{t}k,\dots,-\frac{t}k}_{\text{$k-t$}}\Big).
$$

\begin{remark}\label{scaling}
Note that this representation fixes the scaling of the scalar product as $\alpha_t^2 = 2$
for all $1 \le t \le k-1$. From now on we fix this scaling.
\end{remark}

Let $\BH=\BH_a$ for some $a$, $1\le a\le b$. For each $v \in \SR_\BH^\BM$ define a rational number
\begin{equation}\label{phi-v}
\phi(v) :=
\frac{(\xi,\rho - v\rho)}{k(\xi,\omega_1)} \left( 1 - k \frac{(\xi,\omega_1)^2}{\xi^2} \right).
%
\end{equation}

\begin{definition}\label{very-s}
An element $v \in \SR_\BH^\BM$ is {\sf very special} if $\phi(v)$ is a positive integer.
\end{definition}

\begin{lemma}\label{vse}
If $\BG$ is a group of type $B$, $C$ or $D$, then there are no very special elements.
\end{lemma}
\begin{proof}
Consider the standard numbering of vertices. Let $\beta = \alpha_k$.
Note that if we take $\Do$ to be empty then we have nothing to check (since we assumed $a \ge 1$).
This means that we only have to consider the case when $\Do$ consists of vertices from $1$ to $k-1$.

First, assume that either $k \le n-1$ for type $B$ and $k \le n-2$ for type $D$ or any $k$ for type $C$.
Then $(\xi,\omega_1) = 1$, $\xi^2 = k$ and we see that the second factor in~\eqref{phi-v} vanishes, hence $\phi(v) = 0$.
In the remaining cases ($k = n$ for type $B$ and $k = n$ for type $D$) we have
$(\xi,\omega_1) = 1/2$, $\xi^2 = n/4$, and $k = n$, so the second factor vanishes as well.
\end{proof}

\begin{remark}
It seems plausible that for types $E$, $F$ and $G$ there are no very special elements as well,
although we have not checked this.
On the contrary, for type $A$
$$
\phi(v) = (\xi,\rho - v\rho)/(n + 1 - k),
$$
so very special elements correspond to permutations
$v \in \SSS_{n+1}$ such that $v(n+1) = k$.
\end{remark}

Now we are ready to define the block --- we just set
\begin{equation}\label{bap}
\begin{array}{lll}
\SBo_\sj &=& \{ \lambda \in \HSBo_\sj\ |\ (\lambda + \rho - v\rho,\alpha_1 + \dots + \alpha_{k-1}) < \phi(v)\ \text{for all very special $v$} \},\\
\SBi_\sj &=& \HSBi_\sj,\\
\SB_\sj &=& \SBo_\sj + \sj\xi + \ji_*(\SBi_\sj),
\end{array}
\end{equation}



Further we will show that the block $\SB_\sj$ defined by~\eqref{bap} is exceptional
if its outer part $\SBo_\sj$, viewed as a set of Young diagrams, is closed under
passing to a subdiagram. In fact, we will prove part (a) of the criterion~\ref{crit-eb}
for the block $\SB_\sj$ in Section~\ref{s-np} (without additional conditions).
Part (b) of this criterion will be proved in section~\ref{ss-aw} assuming that $\SBo_\sj$
is closed under passing to a subdiagram.
Finally, we will verify the latter condition for the groups of type $BCD$
by a direct computation in section~\ref{ss-ex}.

\subsection{Exceptional collections}\label{ss-ec}

Before we proceed to the proof that the constructed blocks are exceptional
we will explain how one can make these blocks smaller in order to achieve
semiorthogonality of the subcategories of $\D^b(X)$ generated by the corresponding
equivariant bundles.

First, we define subsets $\SBSi_\sj \subset \SBi_\sj$
by the formula
\begin{equation}\label{sbt-in}
\SBSi_\sj = \left\{ \nu \in \SBi_\sj\ \left|\
\begin{array}{l}
\text{for all $\sj' < \sj$, $\nu' \in \SBSi_{\sj'}$, and $w_\BLi,w'_\BLi \in \BW_\BLi$}\\[2pt]
\text{one has $\rho_{\BH_{\sj'}} - (\sj - \sj')\xi - w_\BLi\ji_*\nu + w'_\BLi\ji_*\nu' \in \BR^*_{\sj'}$}
\end{array}
\right\}\right..
\end{equation}
Note that the above formula is recursive --- it describes $\SBSi_\sj$ in terms of all $\SBSi_{\sj'}$ with $\sj' < \sj$.
We also set
\begin{equation}\label{sbt}
\SBSo_\sj = \left\{ \lambda_0 \in \SBo_\sj\ \left|\
\begin{array}{l}
\text{for all $\sj' < \sj$, $\nu \in \SBSi_{\sj}$, $\nu' \in \SBSi_{\sj'}$, $w_\BLi,w'_\BLi \in \BW_\BLi$, and $w_\BL \in \BW_\BL$}\\[2pt]
\text{one has $\rho_{\BH_{\sj'}} - h_{\sj'}^*(w_\BL\lambda_0 + (\sj - \sj')\xi) - w_\BLi\ji_*\nu + w'_\BLi\ji_*\nu' \in \BR^*_{\sj'}$}
\end{array}
\right\}\right.
\end{equation}
Note that by Remark \ref{w-nu-rem}, we can let the elements $w_\BLi$ and $w'_\BLi$
run through the entire group $\BW_\BL$ in the definitions \eqref{sbt-in} and \eqref{sbt}.
Finally, we set
\begin{equation}\label{sbsp}
\SBS_\sj = \SBSo_\sj + \sj\xi + \ji_*\SBSi_\sj,
\end{equation}
and define the subcategory
$$
\CA_\sj := \langle \CU^{\lambda} \rangle_{\lambda \in \SBS_\sj}.
$$

\begin{theorem}\label{aap-ec}
The collection of subcategories $\{ \CA_\sj \}_{\sj \in \SJ}$
ordered by increasing of $\sj$ is semiorthogonal.
\end{theorem}
\begin{proof}
Assume that $\sj' < \sj$.
Let $\lambda_0 \in \SBSo_\sj$, $\lambda'_0 \in \SBSo_{\sj'}$, $\nu \in \SBSi_{\sj}$, $\nu' \in \SBSi_{\sj'}$.
We have to check that
$$
\Ext^\bullet(\CU^{\lambda_0+\sj\xi+\ji_*\nu},\CU^{\lambda'_0+\sj'\xi+\ji_*\nu'}) = 0.
$$
By Corollary~\ref{extuu} we have to check that for any $\BL$-dominant weight
$$
\mu \in \Conv(\lambda'_0-w_\BL\lambda_0+(\sj'-\sj)\xi+\ji_*\nu' - w_\BL\ji_*\nu)_{w_\BL \in \BW_\BL}
$$
the sum $\mu + \rho_\BG$ is $\BG$-singular.
Note that $h_{\sj'}^*(\lambda'_0) = 0$ since $\lambda'_0 \in \SBSo_{\sj'} \subset \Ker h_{\sj'}^*$, hence
$$
h_{\sj'}^*(\rho + \lambda'_0-w_\BL\lambda_0+(\sj'-\sj)\xi+\ji_*\nu' - w_\BL\ji_*\nu) =
\rho_{\BH_{\sj'}} - h_{\sj'}^*(w_\BL\lambda_0 - (\sj-\sj')\xi) + \ji_*\nu' - w_\BL\ji_*\nu.
$$
By definition of $\SBSo_\sj$, all these weights for $w_\BL\in \BW_\BL$
lie in the interior of the core $\BR_{\sj'}$,
hence we have $h_{\sj'}^*(\mu+\rho)\in\BR^*_{\sj'}$, and so
by Lemma~\ref{core-sing} $h_{\sj'}^*(\mu+\rho)$ is $\BH_{\sj'}$-singular.
But the map $h_{\sj'}^*$ preserves regularity,
hence $\mu + \rho$ is $\BG$-singular as well.
%
\end{proof}

\section{Verification of the invariance condition}\label{s-np}

In this section we prove that the blocks $\SB_\sj$ and $\SBS_\sj$ constructed
in section \ref{ss-ckb} satisfy the invariance condition (part~(a) of the criterion~\ref{crit-eb}).


First, we will need the following simple fact. 
Assume that $\BH \subset \BH'$ is an embedding of semisimple groups corresponding
to the embedding of the Dynkin diagrams $D_\BH \subset D_{\BH'}$ such that 
%
%
$D_{\BH'} \setminus D_\BH$ consists only of one vertex. Let $\alpha$ be
the corresponding simple root and $\eta$ the corresponding fundamental weight of $\BH'$.

\begin{lemma}\label{ccl}
There is a positive integer $k = k_{{\BH'},\BH}$ such that
$$
\rho_{\BH'} - k \eta = w_0^\BH w_0^{\BH'} \rho_{\BH'}.
$$
Moreover, for all $0 < c < k$ the weight $\rho_{\BH'} - c\eta$ is $\BH'$-singular.
\end{lemma}
\begin{proof}
Let us denote the embedding $\BH \to \BH'$ by $h$. Then as we know
$h^*\rho_{\BH'} = \rho_\BH$ and $\Ker h^* = \ZZ\eta$.
Since
$$
h^*w_0^\BH w_0^{\BH'}\rho_{\BH'} = - h^*w_0^\BH\rho_{\BH'} = - w_0^\BH\rho_\BH = \rho_\BH,
$$
we get
$$
\rho_{\BH'} - w_0^\BH w_0^{\BH'}\rho_{\BH'}  = k \eta
$$
for some $k\in\ZZ$.
Moreover, the LHS is a sum of positive roots by Lemma~\ref{w-dom}, hence $(k\eta,\eta) > 0$, hence $k$ is positive.
This proves the first statement.

For the second, by Lemma~\ref{core-sing} it is enough to show that $\rho_{\BH'} - c\eta$ with $0 < c < k$ is in the interior of a core $\BR^*_\delta$
for some strictly dominant $\delta$. In fact, we will show that one can take any strictly dominant $\delta$. 
Indeed, since $\BR^*_\delta$ is convex and $\rho_{\BH'} - c\eta$
lies in the convex hull of $\rho_{\BH'} - \eta$ and $\rho_{\BH'} - (k-1)\eta$ it is enough
to check that the latter two weights are in $\BR^*_\delta$. Fix some strictly dominant $\delta$.


First, we have $(\delta,\rho_{\BH'}-\eta) = (\delta,\rho_{\BH'}) - (\delta,\eta) < (\delta,\rho_{\BH'})$,
so since $\rho_{\BH'}-\eta$ is dominant we have $\rho_{\BH'} - \eta\in\BR_\delta^*$ by Lemma~\ref{dom-plan}. On the other hand,
$$
\rho_{\BH'} - (k-1)\eta =
w_0^\BH w_0^{\BH'}\rho_{\BH'} + \eta =
w_0^\BH w_0^{\BH'}(\rho_{\BH'} +w_0^{\BH'}w_0^\BH\eta)=
w_0^\BH w_0^{\BH'}(\rho_{\BH'}+w_0^{\BH'}\eta).
$$
Since $-w_0^{\BH'}\eta$ is a fundamental weight of $\BH'$,
the same argument as above shows that
$$\rho_{\BH'}+w_0^{\BH'}\eta=\rho_{\BH'}-(-w_0^{\BH'}\eta)\in \BR_\delta^*.$$
Hence, we obtain that $\rho_{\BH'}-(k-1)\eta$ is also in $\BR_\delta^*$.
This finishes the proof.
\end{proof}


\begin{remark}
One can also deduce the claim geometrically. Consider the Grassmannian of ${\BH'}$
corresponding to the root $\alpha$. Then its Picard group is $\ZZ$ and its generator
is the line bundle corresponding to the weight $\eta$.
By Lemma~\ref{canclass} the canonical class of the Grassmannian is given by the weight
$w^0_\BH w^0_{\BH'} \rho - \rho$. On the other hand, it is equal to the line bundle
corresponding to the weight $-k \eta$ for some $k \in \ZZ$. This gives the equality.
Having all this, the singularity of weights $\rho - c\eta$ with $0 < c < k$ is clear.
Indeed by the Borel--Bott--Weil the singularity of $\rho - c\eta$ is equivalent to
the vanishing of the cohomology of the line bundle corresponding to the weight $-c\eta$, 
which indeed vanishes by Kodaira vanishing theorem.
\end{remark}

Now we can verify the invariance condition.

\begin{proposition}\label{inv-cond-prop}
Let $\kappa \in \OP_1(\SB_\sj)$, $v \in \OP_2(\SB_\sj)$.
Then $\kappa \in \Ker h_\sj^*$ and $v \in \BW_{\BH_\sj}$.
In particular, $v\kappa = \kappa$.
\end{proposition}
\begin{proof}
Take arbitrary $\lambda,\lambda' \in \SB_\sj$.
Then $\lambda = \lambda_0 + p\xi + \ji_*\nu$, $\lambda' = \lambda'_0 + p\xi + \ji_*\nu'$, with $\lambda_0,\lambda'_0 \in \SBo_\sj$
and $\nu,\nu' \in \SBi_\sj$.
Note that for any $w_\BL \in \BW_\BL$ we have
\begin{equation}\label{temp-w}
h_\sj^*(\rho + \lambda' - w_\BL\lambda) =
h_\sj^*(\rho + \lambda'_0 + \ji_*\nu' - w_\BL\lambda_0 - w_\BL\ji_*\nu) =
h_\sj^*(\rho - w_\BL\lambda_0) + \ji_*\nu' - w_\BL\ji_*\nu
\end{equation}
since $\lambda'_0 \in \Ker i^*$ and $h_\sj \circ \ji = \ji$.
So, by definition of $\SB_\sj$ (using Remark \ref{w-nu-rem})
we conclude that the weight~\eqref{temp-w} is in $\BR_\delta$.

Let $(\kappa,v) \in \OP(\SB_\sj)$, that is $(\kappa,v) \in \OP(\lambda,\lambda')$ for some $\lambda,\lambda' \in \SB_\sj$.
By definition of the output set the weight
$$
\mu := v(\kappa + \rho) - \rho \in \Conv(\lambda' - w_\BL\lambda)_{w_\BL \in \BW_\BL}
$$
is $\BL$-dominant and $\mu + \rho$ is $\BG$-regular.
Moreover, $h_\sj^*(\mu+\rho)$ is in the convex hull of the weights~\eqref{temp-w} (where $w_\BL$ runs through $\BW_\BL$)
hence is in the core $\BR_\delta$. So, Proposition~\ref{key2} below applies and we conclude that
$\kappa \in \Ker h_\sj^*$, $v \in \BW_{\BH_\sj}$.
%
%
%
\end{proof}

%

\begin{proposition}\label{key2}
Assume that a weight $\mu \in P_\BL$ satisfies
\begin{equation}\label{lambda-ass}
\mu \in P_\BL^+,
\qquad
\mu+\rho \in P_\BG^\reg,
\qquad
h_a^*(\mu+\rho) \in \BR_\delta,
\end{equation}
for some $a$, $0 \le a \le b$. Let also $\mu = v(\kappa + \rho) - \rho$ be the unique presentation of $\mu$
with $\kappa\in P_\BG^+$ and $v\in\SR_\BG^\BL$.
Then
$$
v \in \SR_\BG^\BL \cap \BW_{\BH_a}
\qquad\text{and}\qquad
\kappa \in P_\BG^+ \cap \Ker h_a^*.
$$
In particular, $v\kappa = \kappa$.
%
%
\end{proposition}
\begin{proof}
To simplify the notation we write $\BH$ instead of $\BH_a$ and $h$ instead of $h_a$.
Set $\BM = \BL \cap \BH$. Note that $h^*$ takes $\BG$-regular $\BL$-dominant weights of $P_\BG$ to $\BH$-regular $\BM$-dominant weights of $P_{\BH}$,
hence $h^*(\mu + \rho)$ is $\BH$-regular and $\BM$-dominant. On the other hand, $h^*(\mu + \rho) \in \BR_\delta$,
so Lemma~\ref{core-sing} implies that $h^*(\mu + \rho) = v\rho_{\BH}$ with $v \in \BW_\BH$.
Thus, $v\rho_\BH$ is $\BM$-dominant, so we have $v \in \SR_{\BH}^{\BM}$.
Further, $v\rho_\BH = h^*(v\rho)$, hence
$h^*(\mu + \rho - v\rho) = 0$. Denoting
$$
\kappa = \mu + \rho - v\rho
$$
we see that $\kappa \in \Ker h^*$ and $\mu = v\rho - \rho + \kappa$.
Since $\kappa \in \Ker h^*$ and $v \in \BW_\BH$, we have $v\kappa = \kappa$,
so $\mu$ can be written as $v(\kappa+\rho) - \rho$.
So it remains to check that $\kappa$ is $\BG$-dominant.

To check the dominance of a weight we should check that its inner products with all simple roots are nonnegative.
We divide the simple roots into three groups and consider them one by one.

{\bf Case 1:}\/ the simple roots of $\BH$. If $\alpha \in D_\BH$ then $(\kappa,\alpha) = 0$
since $\kappa \in \Ker h^*$.

{\bf Case 2:}\/ the simple roots of $\BG$ not adjacent to $D_{\BH}$.
If $\alpha$ is such a root then $v^{-1}\alpha = \alpha$ since $v \in \BW_\BH$,
hence $(v\rho,\alpha) = (\rho,v^{-1}\alpha) = (\rho,\alpha)$, therefore
$(\kappa,\alpha) = (\mu,\alpha) \ge 0$. Here the last inequality follows from $\BL$-dominance of $\mu$
since simple roots not adjacent to $D_\BH$ are roots of $\BL$.

{\bf Case 3:}\/ the simple root adjacent to $D_\BH$.
Let $\alpha$ be such a root and let $\BH'$ be the reductive subgroup of $\BG$
such that $D_{\BH'} = D_\BH \cup \{\alpha\}$. Let $\eta \in P_{\BH'}$
be the fundamental weight of $\BH'$ corresponding to the root $\alpha$.
Let $h':\BH' \to \BG$ be the embedding, and let $h$ denote the embeddings $\BH \to \BH'$
and $\BH \to \BG$. Note that $\Ker (h^*:P_{\BH'} \to P_\BH) = \ZZ\eta$.

Note that $h^*(h')^*(\mu + \rho) = h^*(\mu + \rho) = v\rho_\BH = h^*(v\rho_{\BH'})$, hence
$(h')^*(\mu + \rho) = v\rho_{\BH'} + c\eta = v(\rho_{\BH'} + c\eta)$. It is enough to show that $c \ge 0$.
Indeed, since $\alpha$ is a root of $\BH'$ we have $\alpha = h'_*\alpha$, so
\begin{equation*}\label{mubeta}
(\kappa,\alpha) =
(\kappa,h'_*\alpha) =
((h')^*\kappa,\alpha) =
((h')^*(\mu + \rho - v\rho),\alpha) =
(v\rho_{\BH'} + c\eta - v\rho_{\BH'},\alpha) = c(\eta,\alpha) = c\alpha^2/2 \ge 0
\end{equation*}
and we are done. So, assume that $c < 0$.
Since $v^{-1}(h')^*(\mu+\rho)$ is $\BH'$-regular, Lemma~\ref{ccl} implies that
$v^{-1}(h')^*(\mu + \rho) =\rho_{\BH'}+c\eta= -w_0^\BH\rho_{\BH'} - c'\eta$ with $c' \ge 0$. Then
$$
(h')^*\mu = v( - w_0^\BH\rho_{\BH'} - c'\eta) - (h')^*\rho = -vw_0^\BH\rho_{\BH'} - \rho_{\BH'} - c'\eta.
$$
Let us check that the scalar product of this weight with $\alpha$ is always negative.
Indeed, $(\rho_{\BH'},\alpha) > 0$ since $\rho_{\BH'}$ is a strictly dominant weight of $\BH'$.
Further, the root
$w_0^\BH v^{-1}\alpha$ is positive since
$(\eta,w_0^\BH v^{-1}\alpha) = (v w_0^\BH\eta,\alpha) = (\eta,\alpha) > 0$.
Therefore,
$(v w_0^\BH\rho_{\BH'},\alpha) = (\rho_{\BH'},w_0^\BH v^{-1}\alpha) > 0$.
Finally, $(c'\eta,\alpha) \ge 0$ since $c' \ge 0$. Thus, we see that
$$
((h')^*\mu,\alpha) < 0.
$$
But this is equal to $(\mu,\alpha)$ which is nonnegative since $\mu$ is $\BL$-dominant.
This contradiction shows that we actually have $c \ge 0$ which completes the proof.
\end{proof}


\section{Adapted weights and compatibility condition}\label{ss-aw}

Let $L$ be a reductive algebraic group. For any subset $S\subset P_L^+$ of the set of dominant weights of $L$
we denote by $\Rep_S(L)$ the subcategory of $\Rep(L)$ consisting of direct sums of irreducible
representations with highest weights in $S$. We also denote by $\Pi_S:\Rep(L) \to \Rep(L)$
the corresponding projector (that leaves only representations in $\Rep_S(L)$).

A morphism $f:V_1 \to V_2$ in $\Rep(L)$ is called an {\sf $S$-isomorphism}\/
if $\Pi_S(f):\Pi_S(V_1) \to \Pi_S(V_2)$ is an isomorphism. In other words,
$f$ is an $S$-isomorphism if it induces an isomorphism on $\lambda$-isotypical components
for any $\lambda \in S$.

We say that a pair of $L$-dominant weights $(\kappa,\lambda)$ is
{\sf adapted to $S$}\/ (or {\sf $S$-adapted})
if the natural map
\begin{equation}\label{sadapted}
V_L^{\kappa+\lambda}\otimes V_L^\mu \to
V_L^{\kappa}\otimes V_L^{\lambda}\otimes V_L^\mu \to
V_L^{\kappa}\otimes \Pi_S(V_L^{\lambda}\otimes V_L^\mu)
\end{equation}
is an $S$-isomorphism for any $\mu\in S$.

The goal of this section is to show that for all $(\kappa,v) \in \OP_1(\SB)\times\OP_2(\SB)$
the pair $(\kappa,v\rho - \rho)$ (considered as a pair of weights of the Levi subgroup $\BL$)
is $\SB$-adapted for either $\SB = \SB_\sj$ or $\SB = \SBS_\sj$,
which will give the compatibility condition of Proposition~\ref{crit-eb}. 
In fact, we will prove the following more general statement.

Now let us return to the setup of section~\ref{s-oi}, i.e., fix a choice of
the outer component $D_\out$ of $D_\BG\setminus\beta$
of type $A_{k-1}$, a standard numbering of its vertices, and a subdiagram $D_a = D_\BG\setminus \{1,\dots,a\}$.
We will write $\BH$ for the corresponding semisimple subgroup $\BH_a \subset \BG$ and $h$ for its embedding into $\BG$
and put $\BM = \BL \cap \BH$.
Recall also that the subgroups $\BLo \subset \BL$ and $\BLi \subset \BL$ correspond to the outer and the inner parts of $D_\BG$.

Assume that any subsets $\SBi \subset P_\BLi^+$, $\SBo \subset P_\BG^+ \cap \Ker h^*$ and a rational number $\sj \in \QQ$
are given such that
$\sj\xi + \ji_*\SBi \subset P_\BL$.
Set
$$
\SB = \SBo + \sj\xi + \ji_*\SBi.
$$
Note that elements of $\SBo$, being
linear combinations of fundamental weights $\omega_1,\dots,\omega_a$
with nonnegative coefficients can be viewed as Young diagrams:
a weight
$x_1\omega_1 + \dots + x_a\omega_a$ corresponds to the Young diagram with $x_i$ columns of length $i$.
Recall the notion of a very special element of the set $\SR_\BH^\BM$ (Definition~\ref{very-s})
and of the the function $\phi(v)$ given by~\eqref{phi-v}.

\begin{theorem}\label{th-aw}
Assume that the set $\SBo$ has the following two properties:
\begin{enumerate}
\item for all $\lambda \in \SBo$  and all very special $v \in \SR_\BH^\BM$ we have
$(\lambda + \rho - v\rho,\alpha_1 + \dots + \alpha_{k-1}) < \phi(v)$;
\item the set $\SBo$ is closed under passing to Young subdiagrams.
\end{enumerate}
Then for any $\kappa \in P_\BG^+\cap  \Ker h^* $ and any $v \in \SR_\BH^\BM$
the pair $(\kappa,v\rho-\rho)$ is $\SB$-adapted.
\end{theorem}

This result applies to the blocks $\SB_\sj$ and $\SBS_\sj$ defined by~\eqref{bap} and~\eqref{sbsp}.

\begin{corollary}\label{cor-aw}
Assume for some $\sj\in\SJ$ the set $\SBo_\sj$ (resp., $\SBSo_{\sj}$) is closed under
passing to Young subdiagrams. Then the block $\SB_\sj$ (resp., $\SBS_\sj$)
is exceptional.
\end{corollary}
\begin{proof}
Set $\SB=\SB_\sj$ (resp., $\SBS_\sj$).
It is enough to check the two conditions of Proposition~\ref{crit-eb} for $\SB$.
The invariance condition holds for this block by Proposition~\ref{inv-cond-prop}.
To check the compatibility condition we can apply Theorem \ref{th-aw}.
The first condition of this theorem holds by the
definition~\eqref{bap} of the block $\SB_\sj$, while the second holds by assumption.
It remains to observe that for any pair $\kappa \in \OP_1(\SB)$, $v \in \OP_2(\SB)$
we have $\kappa \in P_\BG^+\cap  \Ker h^*$ and $v \in \SR_\BH^\BM$.
Hence, Theorem \ref{th-aw}, applied to $\SB$ and a pair $\kappa \in \OP_1(\SB)$, $v \in \OP_2(\SB)$,
implies that the compatibility condition is satisfied for $\SB$.
\end{proof}

Unfortunately, we were not able to find an abstract way of checking that
$\SB_\sj$ or $\SBSo_\sj$ is closed
under passing to Young subdiagrams.
So, we will check it for classical groups in section~\ref{ss-ex}
as a result of an explicit description of the blocks.

\subsection{Preparations}

We start with a description of the connected component of the center of $\BL$.

\begin{lemma}\label{integer-multiple-lem}
Let $\BZ\subset\BL$ be the connected component of the center of $\BL$.
Then $\BZ \cong \Gm$ and the map $P_\BL\to P_\BZ=\ZZ$, induced by the embedding $\BZ\to\BL$,
is given by the scalar product with the minimal rational multiple $c\xi$ of $\xi$,
such that $(c\xi,-)$ is an integral valued function on $P_\BL$.
\end{lemma}
\begin{proof}
First, note that $\BZ \cong \Gm$ since it is a 1-dimensional (since $\BP$ is maximal)
connected commutative reductive group. As a consequence, $P_\BZ \cong \ZZ$.
Since the map $P_\BL \to P_\BZ$ is dual to the embedding of $\BZ$ into a maximal
torus of $\BL$, it is surjective. Note also that the adjoint representation
of the semisimple part of $\BL$ is a trivial representation of $\BZ$, hence all simple roots of $\BL$ are mapped to zero.
This implies that the map is given by the scalar product with a multiple $c\xi$ of $\xi$. Moreover, since the scalar
product should be a map to $\ZZ$, it follows that $(c\xi,-)$ should be an integral function on $P_\BL$
(and in particular, $c$ should be rational since the scalar product has rational values on the weight lattice),
and the surjectivity of the map implies that $c$ is minimal with this property.
\end{proof}

Consider the diagram of groups
$$
\xymatrix{
& \BLo \times \Gm \times \BLi \ar[dl]_{\varpi} \ar[dr]^{\pi} \\
\GL_k\times\Gm\times\BLi && \BL
}
$$
where $\pi$ and $\varpi$ are defined as follows.
The morphism $\pi$ is induced by the embeddings $\jo:\BLo \to \BL$, $\ji:\BLi \to \BL$ and
by the isomorphism $\Gm \cong \BZ$.
The restriction of $\varpi$
to $\BLo = \SL_k$ (resp., $\BLi$) is given by
the natural embedding  $\SL_k \subset \GL_k$ (resp., the
identity map to $\BLi$). Finally, the restriction of $\varpi$ to $\Gm$ is given by
$z \mapsto (z^{(c\xi,\omega_1)}\times 1,z,1)$.
Note that the map $\pi$ is an isogeny and the map $\varpi$ is an embedding.

Now take any $\kappa \in \Ker h^* \cap P_\BG^+$, $v \in \SR_\BH^\BM$ and $\mu \in \SB$,
and consider the morphisms
\begin{equation}\label{map}
V_\BL^{\kappa + v\rho - \rho} \otimes V_\BL^\mu \to
V_\BL^{\kappa} \otimes V_\BL^{v\rho - \rho} \otimes V_\BL^\mu \to
V_\BL^{\kappa} \otimes \Pi_\SB(V_\BL^{v\rho - \rho} \otimes V_\BL^\mu).
\end{equation}
Our goal is to show that after application of $\Pi_\SB$ this map becomes an isomorphism.
For this we pullback the map via $\pi$ to a map of representations of the group $\BLo\times\Gm\times\BLi$
and check that the same map can be realized as a pullback via $\varpi$ of a map of representations
of $\GL_k\times\Gm\times\BLi$. We also express the action of the projector $\Pi_\SB$ in terms
of group $\GL_k\times\Gm\times\BLi$ and thus reduce the verification to the latter group.
It turns out that the components $\Gm$ and $\BLi$ play no role, and the statement
essentially reduces to a similar statement for representations of the group $\GL_k$. 
The latter statement is proved in the Appendix.

%

Recall that irreducible representations of $\GL_k$ are numbered by nonincreasing sequences
of integers of length $k$. For a sequence $\kappa_\bullet = (\kappa_1 \ge \kappa_2 \ge \dots \ge \kappa_k)$
we will denote by $V_{\GL_k}^{\kappa_\bullet}$ the corresponding $\GL_k$-representation.

\begin{lemma}\label{pivarpi}
For any $\lambda \in P_\BL^+$ we have
$$
\pi^*V_\BL^\lambda = V_\BLo^{\jo^*\lambda} \otimes V_\Gm^{(c\xi,\lambda)} \otimes V_\BLi^{\ji^*\lambda}.
$$
On the other hand, for any nonincreasing sequence $\kappa_\bullet = (\kappa_1 \ge \dots \ge \kappa_k)$ of integers, 
any $z\in \ZZ$ and any $\nu \in P_\BLi^+$ we have
$$
\varpi^*(V_{\GL_k}^{\kappa_\bullet} \otimes V_\Gm^z \otimes V_\BLi^\nu) = V_\BLo^\kappa \otimes V_\Gm^{(c\xi,\omega_1)\sum_{i=1}^k \kappa_i + z} \otimes V_\BLi^\nu,
$$
where $\kappa = \sum_{i=1}^{k-1}(\kappa_i - \kappa_{i+1})\omega_i$ is the weight of $\BLo$ corresponding to $\kappa_\bullet$.
\end{lemma}
\begin{proof}
This is straightforward (to compute the $\Gm$-component of $\pi^*V_\BL^\lambda$ use Lemma~\ref{integer-multiple-lem}).
\end{proof}

Now we give a description of the pullbacks via $\pi$ of representations $V_\BL^\kappa$, $V_\BL^{v\rho-\rho}$ and $V_\BL^\mu$
entering into~\eqref{map} as the pullbacks via $\varpi$ of appropriate representations of $\GL_k\times\Gm\times\BLi$.
In fact, such description is not unique (which is clear from Lemma~\ref{pivarpi}), so we choose a description which is most convenient for our purposes.

\begin{lemma}\label{kapparep}
Let $\kappa \in \Ker h^* \cap P_\BG^+$. Then there exists a unique nonincreasing sequence of integers
$\kappa_\bullet = (\kappa_1 \ge \kappa_2 \ge \dots \ge \kappa_a \ge \kappa_{a+1} = \dots = \kappa_k = 0)$
such that
%
$$
\pi^*V_\BL^\kappa = \varpi^*V_{\GL_k}^{\kappa_\bullet}.
$$
%
%
\end{lemma}
\begin{proof}
By definition, $\kappa$ is a nonnegative linear combination of $\omega_1,\dots,\omega_a$.
Let $\kappa_1-\kappa_2$, $\kappa_2 - \kappa_3$, \dots, $\kappa_{a-1} - \kappa_a$ and $\kappa_a$ be the coefficients.
Then $\kappa_1 \ge \dots \ge \kappa_a \ge 0$. Extending this sequence by $\kappa_{a+1} = \dots = \kappa_k = 0$ we obtain a sequence $\kappa_\bullet$.
To prove the required isomorphism we use Lemma~\ref{pivarpi}. By this Lemma, we only have to check that $(c\xi,\kappa) = (c\xi,\omega_1)\sum_{i=1}^k \kappa_i$.
For this we note that for $i < b$ we have $\alpha_i = 2\omega_i - \omega_{i-1} - \omega_{i+1}$, hence $(c\xi,\omega_i) = i(c\xi,\omega_1)$, so
$$(c\xi,\kappa) = (c\xi,\omega_1)\sum_{i=1}^a i(\kappa_i - \kappa_{i-1}) = (c\xi,\omega_1)\sum_{i=1}^a \kappa_i = (c\xi,\omega_1)\sum_{i=1}^k \kappa_i,$$
as required.
\end{proof}

\begin{lemma}\label{vrep}
Let $v \in \SR_\BH^\BM$. Set $\nu_v = \ji^*(v\rho - \rho)$.
Then there exists a unique sequence of integers $\tau_\bullet = (0 = \tau_1 = \dots = \tau_a \ge \tau_{a+1} \ge \dots \ge \tau_k)$
such that
$$
\pi^*V_\BL^{v\rho - \rho} = \varpi^*(V_{\GL_k}^{\tau_\bullet} \otimes V_\Gm^{z(v)} \otimes V_\BLi^{\nu_v}),
$$
where
\begin{equation}\label{zv}
z(v) = (v\rho - \rho,c\xi)(1 - k(\omega_1,\xi)^2/\xi^2).
\end{equation}
\end{lemma}
\begin{proof}
Consider the restriction $\jo^*(v\rho - \rho)$. It is a weight of $\SL_k$. A weight of $\SL_k$ can be thought
of as a weight of $\GL_k$ up to adding a central character. In other words, it is given by a nonincreasing sequence
of integers up to a simultaneous translation. Consider the sequence $\tau_1 \ge \dots \ge \tau_k$ representing
$\jo^*(v\rho - \rho)$ such that $\tau_1 = 0$. Note that $v\rho - \rho$ is orthogonal to $\alpha_1,\dots,\alpha_{a-1}$
(because these roots are orthogonal to the roots of $\BH$ and hence are $v$-invariant), hence $\tau_1 = \tau_2 = \dots = \tau_a$.

Further, we denote by $\nu_v$ the weight $\ji^*(v\rho - \rho)$. Then the representations $\pi^*V_\BL^{v\rho -\rho}$
and $\varpi^*(V_{\GL_k}^{\tau_\bullet} \otimes V_\BLi^{\nu_v})$ have the same restrictions to $\BLo$ and $\BLi$, so it
remains to compare the central characters. First, the central character of $V_\BL^{v\rho - \rho}$ is $(c\xi,v\rho - \rho)$.
Further, the central character of $V_{\GL_k}^{\tau_\bullet}$ is $(c\xi,\omega_1)\sum \tau_i$, while the central character of $V_\BLi^{\nu_v}$ is $0$.
Note that since $\tau_1 = 0$ and $k\omega_1=(k-1,-1,\ldots,-1)$, we have
$$
\sum\tau_i = -(k\omega_1,\jo^*(v\rho - \rho)) =
(v\rho - \rho,-k\jo_*\omega_1) =
(v\rho - \rho,-k\omega_1 + k((\omega_1,\xi)/\xi^2)\xi) =
k(v\rho - \rho,\xi)(\omega_1,\xi)/\xi^2.
$$
In the third equality above we use the formula $\jo_*\omega_1 = \omega_1 - \frac{(\omega_1,\xi)}{\xi^2}\xi$
analogous to the formula in the proof of Lemma~\ref{existsp}.
So, we see that the difference of the characters is
$$
(v\rho - \rho,c\xi) - (c\xi,\omega_1)k(v\rho - \rho,\xi)(\omega_1,\xi)/\xi^2 =
(v\rho - \rho,c\xi)(1 - k(\omega_1,\xi)^2/\xi^2)=z(v).
$$
Thus, twisting $V_{\GL_k}^{\tau_\bullet} \otimes V_\BLi^{\nu_v}$ by $V_\Gm^{z(v)}$ we obtain an isomorphism.
\end{proof}

\begin{lemma}\label{murep}
Let $\mu = \mu_\out + \sj\xi + \ji_*\mu_\inn \in \SB$. Then there exists a unique nonincreasing sequence of integers
$\mu_\bullet = (\mu_1 \ge \mu_2 \ge \dots \ge \mu_a \ge \mu_{a+1} = \dots = \mu_k = 0)$
such that
$$
\pi^*V_\BL^\mu = \varpi^*(V_{\GL_k}^{\mu_\bullet} \otimes V_\Gm^{c\sj\xi^2} \otimes V_\BLi^{\mu_\inn}).
$$
\end{lemma}
\begin{proof}
Note that $V_\BL^\mu = V_\BL^{\mu_\out} \otimes V_\BL^{\sj\xi + \ji_*\mu_\inn}$.
Since $\mu_\out \in \Ker h^* \cap P_\BG^+$, we already know from Lemma~\ref{kapparep} that
$\pi^*V_\BL^{\mu_\out} = \varpi^*V_{\GL_k}^{\mu_\bullet}$ for a uniquely determined
sequence $\mu_\bullet = (\mu_1 \ge \mu_2 \ge \dots \ge \mu_a \ge \mu_{a+1} = \dots = \mu_k = 0)$.
So, it remains to express $\pi^*V_\BL^{\sj\xi + \ji_*\mu_\inn}$ as a product of representations of $\Gm$ and $\BLi$.
Since $\ji^*(\sj\xi + \ji_*\mu_\inn) = \mu_\inn$, the $\BLi$-component is $V_\BLi^{\mu_\inn}$. On the other hand, the
$\Gm$-component has weight $(c\xi,\sj\xi + \ji_*\mu_\inn) = c\sj\xi^2$.
%
%
\end{proof}

\begin{proposition}\label{linb}
A representation $\varpi^*(V_{\GL_k}^{\lambda_\bullet} \otimes V_\Gm^z \otimes V_\BLi^\nu)$ is isomorphic to the pullback via $\pi$
of a representation in $\SB$ if and only if $\nu \in \SBi$,
\begin{equation}\label{l1a}
\sum_{i=1}^a (\lambda_i - \lambda_{i+1})\omega_i \in \SBo,
\end{equation}
and
\begin{equation}\label{lk}
\lambda_{a+1} = \dots = \lambda_k = \frac{c\sj\xi^2 - z}{k(c\xi,\omega_1)}.
\end{equation}
\end{proposition}
\begin{proof}
Note that by Lemma~\ref{pivarpi} for any $s \in \ZZ$ we have
$$
\varpi^*(V_{\GL_k}^{(\lambda_1,\dots,\lambda_k)} \otimes V_\Gm^z \otimes V_\BLi^\nu) \cong
\varpi^*(V_{\GL_k}^{(\lambda_1-s,\dots,\lambda_k-s)} \otimes V_\Gm^{z + sk(c\xi,\omega_1)}\otimes V_\BLi^\nu).
$$
So, taking $s = \lambda_k$ and using Lemma~\ref{murep} we deduce $z + k\lambda_k(c\xi,\omega_1) = c\sj\xi^2$.
The Proposition follows.
\end{proof}

Denote by $\pi^*\SB$ the set of all representations of $\BLo\times\Gm\times\BLi$ which are pullbacks via $\pi$
of representations of $\BL$ from the block $\SB$. Now we can rewrite the action of the projector $\Pi_{\pi^*\SB}$
on the subcategory of representations of $\BLo\times\Gm\times\BLi$ with a given $\Gm$-component.

\begin{corollary}
We have
\begin{equation*}
\Pi_{\pi^*\SB}(\varpi^*(V_{\GL_k}^{\lambda_\bullet} \otimes V_\Gm^{z(v)+c\sj\xi^2} \otimes V_\BLi^\nu)) =
\varpi^*(\Pi_{S^\out}(V_{\GL_k}^{\lambda_\bullet}) \otimes V_\Gm^{z(v)+c\sj\xi^2} \otimes \Pi_{\SB^\inn}(V_\BLi^\nu)),
\end{equation*}
where $S^\out$ is the set of all $\lambda_\bullet$ such that~\eqref{l1a} holds and
\begin{equation}\label{lkn}
\lambda_{a+1} = \dots = \lambda_k =  \phi(v).
\end{equation}
\end{corollary}
\begin{proof}
Substituting $z = z(v) + c\sj\xi^2$ into~\eqref{lk} and comparing~\eqref{zv} with~\eqref{phi-v}
we deduce the condition~\eqref{lkn}.
\end{proof}

\subsection{Proof of the compatibility}\label{ss-proof-cc}

The goal of this section is to prove Theorem~\ref{th-aw}. So we take arbitrary $\mu = \mu_\out + \sj\xi + \ji_*\mu_\inn \in \SB$ and
consider the tensor product $V_\BL^\kappa \otimes V_\BL^{v\rho - \rho} \otimes V_\BL^\mu$. We have
%
%
\begin{align*}
\pi^*(V_\BL^{\kappa + v\rho - \rho} \otimes V_\BL^\mu) & = 
\varpi^*\left((V_{\GL_k}^{\kappa_\bullet + \tau_\bullet}\otimes V_{\GL_k}^{\mu_\bullet}) \bigotimes V_\Gm^{z(v) + c\sj\xi^2} \bigotimes
(V_\BLi^{\nu_v} \otimes V_\BLi^{\mu_\inn})\right),\\
\pi^*(V_\BL^\kappa \otimes V_\BL^{v\rho - \rho} \otimes V_\BL^\mu) & = 
\varpi^*\left((V_{\GL_k}^{\kappa_\bullet}\otimes V_{\GL_k}^{\tau_\bullet}\otimes V_{\GL_k}^{\mu_\bullet}) \bigotimes V_\Gm^{z(v) + c\sj\xi^2} \bigotimes
(V_\BLi^{\nu_v} \otimes V_\BLi^{\mu_\inn})\right),\\
\pi^*(V_\BL^\kappa \otimes \Pi_\SB(V_\BL^{v\rho - \rho} \otimes V_\BL^\mu)) & = 
\varpi^*\left((V_{\GL_k}^{\kappa_\bullet}\otimes \Pi_{S^\out}(V_{\GL_k}^{\tau_\bullet}\otimes V_{\GL_k}^{\mu_\bullet})) \bigotimes V_\Gm^{z(v) + c\sj\xi^2} \bigotimes
\Pi_{\SB^\inn}(V_\BLi^{\nu_v} \otimes V_\BLi^{\mu_\inn})\right),
\end{align*}
So, the $\pi$-pullback of the map~\eqref{map} is equal to the $\varpi$-pullback of the tensor product of the map
\begin{equation}\label{mapgl}
V_{\GL_k}^{\kappa_\bullet + \tau_\bullet} \otimes V_{\GL_k}^{\mu_\bullet} \to 
V_{\GL_k}^{\kappa_\bullet} \otimes V_{\GL_k}^{\tau_\bullet} \otimes V_{\GL_k}^{\mu_\bullet} \to 
V_{\GL_k}^{\kappa_\bullet}\otimes \Pi_{S^\out}(V_{\GL_k}^{\tau_\bullet}\otimes V_{\GL_k}^{\mu_\bullet})
\end{equation}
with the maps
\begin{equation*}
V_\Gm^{z(v) + c\sj\xi^2} \xrightarrow{\ \id\ } V_\Gm^{z(v) + c\sj\xi^2}
\qquad\text{and}\qquad
V_\BLi^{\nu_v} \otimes V_\BLi^{\mu_\inn} \to \Pi_{\SB^\inn}(V_\BLi^{\nu_v} \otimes V_\BLi^{\mu_\inn}).
\end{equation*}
Since the last map is a $\SB^\inn$-isomorphism, we only
have to check that the map \eqref{mapgl} is an $S^\out$-isomorphism.

Let $\TS^\out$ be the set of all $\lambda_\bullet$ satisfying only~\eqref{lkn}.
We claim
that if we replace in~\eqref{mapgl}
the projector $\Pi_{S^\out}$ by $\Pi_{\TS^\out}$, then the obtained map
\begin{equation}\label{mapgl-bis}
V_{\GL_k}^{\kappa_\bullet + \tau_\bullet}\otimes V_{\GL_k}^{\mu_\bullet} \to
V_{\GL_k}^{\kappa_\bullet} \otimes V_{\GL_k}^{\tau_\bullet}\otimes V_{\GL_k}^{\mu_\bullet} \to
V_{\GL_k}^{\kappa_\bullet} \otimes \Pi_{\TS^\out}(V_{\GL_k}^{\tau_\bullet}\otimes V_{\GL_k}^{\mu_\bullet})
\end{equation}
will be an $\TS^\out$-isomorphism.
Indeed, if $\phi(v)$ is a nonpositive integer then this is Corollary~\ref{GL-lem-cor} from Appendix.
If $\phi(v)$ is not an integer, then $\TS^\out = \emptyset$, so any map is an $\TS^\out$-isomorphism.
Finally, if $\phi(v)$ is a positive integer then $v$ is very special, hence we have $(\mu + \rho - v\rho,\alpha_1 + \dots + \alpha_{k-1}) < \phi(v)$.
This means that $\mu_1 + \tau_k < \phi(v)$, so by Littlewood--Richardson rule the tensor products 
$V_{\GL_k}^{\kappa_\bullet + \tau_\bullet} \otimes V_{\GL_k}^{\mu_\bullet}$ and
$V_{\GL_k}^{\tau_\bullet} \otimes V_{\GL_k}^{\mu_\bullet}$ contain no terms $V_{\GL_k}^{\lambda_\bullet}$ with $\lambda_k = \phi(v)$,
and a fortiori no terms in $\TS^\out$. Thus, both the source and the target of the map~\eqref{mapgl} become zero after applying $\Pi_{\TS^\out}$,
hence the map becomes an isomorphism. This finishes the proof that
\eqref{mapgl-bis} is an $\TS^\out$-isomorphism.

Since $S^\out \subset \TS^\out$ it remains to check that 
\begin{equation}\label{sotso}
\Pi_{S^\out}(V_{\GL_k}^{\kappa_\bullet} \otimes \Pi_{\TS^\out}(V_{\GL_k}^{\tau_\bullet}\otimes V_{\GL_k}^{\mu_\bullet})) =
\Pi_{S^\out}(V_{\GL_k}^{\kappa_\bullet} \otimes \Pi_{S^\out}(V_{\GL_k}^{\tau_\bullet}\otimes V_{\GL_k}^{\mu_\bullet}).
\end{equation}
Indeed, if~\eqref{sotso} is true then the result of applying $\Pi_{S^\out}$ to~\eqref{mapgl} coincides with 
that of applying $\Pi_{S^\out}$ to~\eqref{mapgl-bis}. Since the latter map becomes an isomorphism already after applying $\Pi_{\TS^\out}$, the assertion would follow.

Now to verify~\eqref{sotso} we have to check that for any $\lambda_\bullet \in \TS^\out$ such that
$V_{\GL_k}^{\lambda_\bullet}$ appears as a summand in
$V_{\GL_k}^{\tau_\bullet}\otimes V_{\GL_k}^{\mu_\bullet}$ and
$\Pi_{S^\out}(V_{\GL_k}^{\kappa_\bullet} \otimes V_{\GL_k}^{\lambda_\bullet}) \ne 0$, 
one has $\lambda_\bullet \in S^\out$.
Let $\lambda'_\bullet\in S^\out$ be such that
$V_{\GL_k}^{\lambda'_\bullet}$ is a summand in 
$V_{\GL_k}^{\kappa_\bullet} \otimes V_{\GL_k}^{\lambda_\bullet}$.
Note that both $\lambda_\bullet$ and $\lambda'_\bullet$ satisfy~\eqref{lkn}. 
Since $\kappa_\bullet$ is nonnegative and $V_{\GL_k}^{\lambda'_\bullet}$ is a summand in
$V_{\GL_k}^{\kappa_\bullet} \otimes V_{\GL_k}^{\lambda_\bullet}$, it follows that the Young diagram corresponding
to the weight $\sum_{i=1}^a (\lambda_i-\lambda_{i+1})\omega_i$ is a subdiagram in the Young diagram corresponding
to the weight $\sum_{i=1}^a (\lambda'_i-\lambda'_{i+1})\omega_i$. The latter is in $\SBo$ since 
$\lambda'_\bullet\in S^\out$. 
Hence, the former is also in $\SBo$, since $\SBo$ is closed with respect to passing to a Young subdiagram. 
Thus, $\lambda_\bullet$ is in $S^\out$ and we are done.

\section{Explicit description of the exceptional blocks}\label{ss-ed}

In this section we will pass from the abstract description of the blocks $\SB_\sj$
given in subsections~\ref{sbpab} and~\ref{sbap} to a more explicit description
which will be used later to deal with concrete examples. We show in fact that both the inner
and the outer parts of the blocks are described by several simple inequalities, numbered
by $\BW_{\BM_\sj}$-orbits in the $\BW_{\BH_\sj}$-orbit of the weight $\delta_{a(\sj)}$ 
(the shape of the core $\BR_\sj$).

Let us fix $\sj\in\SJ$.
It will be convenient to write the shape $\delta=\delta_{a(\sj)}\in P_{\BH_\sj}^+$ of the core
$\BR_\sj=\BR_\delta$ in the form
\begin{equation}\label{dega}
\delta = - h_\sj^*\gamma,
\end{equation}
where $\gamma \in P_\BG$.

\begin{remark}
Since the action of the Weyl group on roots is much better understood than on arbitrary weights
(for example, one can use tables of roots), the most convenient choice of $\gamma$ is
the simple root of the vertex of $D_\BG$ adjacent to $D_{\BH_\sj}$.
In this case the $\BW_{\BH_\sj}$-orbit of $\gamma$ is described in Lemma~\ref{beta-orbit-lem}.
\end{remark}




\subsection{The big blocks}

First, we give a description of the block $\SB_\sj$.

Assume that $\gamma \in P_\BG$ and that $\delta$ defined by~\eqref{dega} is $\BH_\sj$-dominant.
To unburden the notation we will write $\BH$ for $\BH_\sj$, $h$ for $h_\sj$, and $\BM$ for $\BM_\sj = \BL \cap \BH_\sj$.
Since $\BW_{\BM} \subset \BW_{\BH}$, the $\BW_{\BH}$-orbit of $\gamma$ splits into several $\BW_{\BM}$-orbits.
We number the orbits by integers $0,\dots,m$ in such a way that the $0$-th orbit
is the $\BW_{\BM}$-orbit of $\gamma$ itself.

In each $\BW_\BM$-orbit we have two special elements:
the unique $\BM$-dominant representative $\gamma_{t+}$ and the unique $\BM$-antidominant representative $\gamma_{t-}$ (where $0 \le t \le m$).
Note that $\gamma_{0-} = \gamma$, since we assumed that $h^*\gamma = - \delta$ is $\BH$-antidominant.
Using these data we can describe the block $\SB_\sj$ more explicitly.
We start with the inner part of the block.

\begin{proposition}\label{bia-exp}
We have
\begin{equation}\label{bbi}
\SBi_\sj = \left\{\nu \in P^+_\BLi \ \left|\
\begin{array}{l}
\max \{ (\ji^*\gamma_{t+},\nu), -(\ji^*\gamma_{t-},\nu) \} \le \frac12(h^*(\gamma_{t-} - \gamma),\rho_{\BH})\ \text{for all $0 \le t \le m$}\\
\qquad\text{and $\sj\xi + i_*\nu \in P_\BL$}
\end{array} 
\right\}\right..
\end{equation}
\end{proposition}
\begin{proof}
By definition, $\SBi_\sj$ is the set of all $\nu$ such that $\sj\xi + i_*\nu \in P_\BL$
and $\rho_{\BH} \pm 2w_\BLi\ji_*\nu \in \BR_\delta$ for all $w_\BLi \in \BW_{\BLi}$.
We only need to rework the second condition. Substituting the Definition~\ref{def-core} 
of the core~$\BR_\delta$, it can be rewritten as
$$
-(w_{\BH}h^*\gamma,\rho_{\BH} \pm 2w_\BLi\ji_*\nu) \le -(h^*\gamma,\rho_{\BH}). 
$$
Since $h^*$ is $\BW_{\BH}$-equivariant, this inequality can be rewritten as
$$
\pm(h^*w_{\BH}\gamma,2w_\BLi\ji_*\nu) \le (h^*(w_{\BH}\gamma - \gamma),\rho_{\BH}).
$$
Note that $w_{\BH}\gamma = w_{\BM}\gamma_{t+}$ for appropriate $w_{\BM} \in \BW_{\BM}$ and $t \in \{0,1,\dots,m\}$.
After such a substitution the inequality takes the form
$$
\pm(h^*w_{\BM}\gamma_{t+},2w_\BLi\ji_*\nu) \le (h^*(w_{\BM}\gamma_{t+} - \gamma),\rho_{\BH}).
$$
Let $\BMo = \BLo \cap \BH$, $\BMi = \BLi \cap\BH = \BLi$. Then $\BW_{\BM} = \BW_\BMo\times\BW_\BMi$.
In particular, we can write $w_{\BM} = w_\BMo w_\BMi$ with $w_\BMo \in \BW_\BMo$, $w_\BMi \in \BW_\BMi$.
Moreover, $\ji_*\nu$ is fixed by $\BW_\BMo$, hence the LHS is equal to
$$
\pm2(h^*\gamma_{t+},w_{\BM}^{-1}w_\BLi\ji_*\nu) =
\pm2(h^*\gamma_{t+},w_\BMo^{-1}w_\BMi^{-1}w_\BLi\ji_*\nu) =
\pm2(h^*\gamma_{t+},w'_\BLi\ji_*\nu),
$$
where $w'_\BLi = w_\BMi^{-1}w_\BLi$. Note that $w'_\BLi$ in the LHS runs through $\BW_\BLi$ independently
of $w_{\BM}$ in the RHS running through $\BW_{\BM}$. Hence, the inequality for all $w'_\BLi \in \BW_\BLi$,
$w_{\BM} \in \BW_{\BM}$ is equivalent to
$$
\max_{w'_\BLi \in \BW_\BLi} \{ \pm2(h^*\gamma_{t+},w'_\BLi\ji_*\nu) \} \le
\min_{w_{\BM} \in \BW_{\BM}} \{ (h^*(w_{\BM}\gamma_{t+} - \gamma),\rho_{\BH}) \}.
$$
The expression under the maximum can be rewritten as $\pm2((w'_\BLi)^{-1}i^*\gamma_{t+},\nu)$.
Since both $\nu$ and $i^*\gamma_{t+}$ are $\BLi$-dominant, the expression with ``$+$'' sign
is maximal when $w'_\BLi = 1$, and the expression with ``$-$'' sign is maximal when
$(w'_\BLi)^{-1}i^*\gamma_{t+} = i^*\gamma_{t-}$.
Thus, the LHS is
$$
\max \{ 2(i^*\gamma_{t+},\nu), -2(i^*\gamma_{t-},\nu) \}.
$$
Similarly, since $\rho_{\BH}$ is $\BM$-dominant, the expression in the RHS is minimal
when $w_{\BM}\gamma_{t+} = \gamma_{t-}$. The claim follows.
\end{proof}

Now let us rewrite more explicitly the definition of the outer part of the block $\SBo_\sj$.
Denote by $\hgamma_t$ the $\BLo$-dominant representative in the $\BW_\BLo$-orbit
of $h_*h^*\gamma_{t+}$. Also, set
\begin{equation}\label{dap}
\arraycolsep = 0pt%
\begin{array}{lrl}
d_\sj^{t,+} := &  \max & \{ (\ji^*\gamma_{t+},\nu)\ |\ \nu \in \SBi_\sj \},\\
d_\sj^{t,-} := & - \min & \{ (\ji^*\gamma_{t-},\nu)\ |\ \nu \in \SBi_\sj \}.
\end{array}
\end{equation}


\begin{proposition}\label{wtgt}
We have
\begin{equation}\label{bbo}
\HSBo_\sj = \{\lambda \in \Ker h^*\cap P^+_\BG \ |\
(\lambda,\hgamma_t) + d_\sj^{t,+} + d_\sj^{t,-} \le (\rho_{\BH},\gamma_{t-} - \gamma)\ \text{for all $0 \le t \le m$} \}.
\end{equation}
\end{proposition}
\begin{proof}
Take $\lambda \in \Ker h^* \cap P_\BG^+$.
By definition $\lambda \in \HSBo_\sj$ if and only if
$h^*(\rho - w_\BL\lambda) - w_\BLi\ji_*\nu + w'_\BLi \ji_*\nu' \in \BR_\delta$.
By definition of $\BR_\delta$ this is equivalent to
$$
(h^*(\rho - w_\BL\lambda) - w_\BLi\ji_*\nu + w'_\BLi \ji_*\nu', - v_{\BH} h^*\gamma) \le (\rho_{\BH}, -h^*\gamma)
$$
for all $\nu,\nu' \in \SBi_\sj$, $w_\BL \in \BW_\BL$, $w_\BLi,w'_\BLi \in \BW_\BLi$, and $v_{\BH} \in \BW_{\BH}$.
Note that $\BW_\BL = \BW_\BLo\times\BW_\BLi$ and that $\lambda$ is $\BW_\BLi$-invariant.
So we can rewrite the above condition as
$$
(h^*(\rho - w_\BLo\lambda) - w_\BLi \ji_*\nu + w'_\BLi \ji_*\nu', - v_{\BH} h^*\gamma) \le (\rho_{\BH}, -h^*\gamma)
$$
Since $h^*$ is $\BW_\BH$-equivariant, we have $v_{\BH} h^*\gamma = h^*(v_{\BH}\gamma)$.
Further, each weight $v_{\BH}\gamma$ can be written as $v_{\BM}\gamma_{t+}$ for some $0 \le t \le m$ and $v_{\BM} \in \BW_{\BM}$.
This allows to rewrite the condition as
$$
(h^*(\rho - w_\BLo\lambda) - w_\BLi \ji_*\nu + w'_\BLi \ji_*\nu', - v_{\BM} h^*\gamma_{t+}) \le (\rho_{\BH}, -h^*\gamma)
$$
for all $\nu,\nu' \in \SBi_\sj$, $w_\BLo \in \BW_\BLo$, $w_\BLi,w'_\BLi \in \BW_\BLi$, $v_{\BM} \in \BW_{\BM}$, and all $t$, $0 \le t \le m$.
Now recall that $h^*\rho = \rho_{\BH}$ and move it from the LHS to the RHS:
$$
(h^*(- w_\BLo\lambda) - w_\BLi \ji_*\nu + w'_\BLi \ji_*\nu', - v_{\BM} h^*\gamma_{t+}) \le (\rho_{\BH}, h^*(v_{\BM} \gamma_{t+} - \gamma)).
$$
Now, writing $v_{\BM} = v_\BMo v_\BLi$ in the LHS, taking into account that $h^*$ is $\BW_\BM$-equivariant,
and substituting
$v_\BMo^{-1}w_\BLo$ with $w_\BLo$, $v_\BLi^{-1}w_\BLi$ with $w_\BLi$, and
$v_\BLi^{-1}w'_\BLi$ with $w'_\BLi$
we rewrite the condition as
$$
(h^*(- w_\BLo\lambda) - w_\BLi \ji_*\nu + w'_\BLi \ji_*\nu', - h^*\gamma_{t+}) \le (\rho_{\BH}, h^*(v_{\BM} \gamma_{t+} - \gamma)).
$$
Finally, using the adjunction of $h^*$ and $h_*$ and of $\ji^*$ and $\ji_*$ we rewrite this as
$$
(w_\BLo\lambda, h_*h^*\gamma_{t+}) + (w_\BLi \nu, \ji^*\gamma_{t+}) + (- w'_\BLi \nu', \ji^*\gamma_{t+}) \le
(\rho_{\BH}, h^*(v_{\BM} \gamma_{t+} - \gamma)).
$$
Note that each term on both sides
contains an action of a Weyl group element, and these elements run through
the corresponding Weyl groups independently. Therefore, one can replace each summand by its
maximum (in the LHS) or minimum (in the RHS) to obtain an equivalent inequality.

The maximums of the second and the third summands in the LHS are given by $d^{t,\pm}_\sj$ by definition.
The first summand can be rewritten as $(\lambda, w_\BLo^{-1}h_*h^*\gamma_{t+})$ and since
$\lambda$ is $\BLo$-dominant, to achieve the maximum one should choose $w_\BLo^{-1}$ in such a way that
the corresponding weight is also $\BLo$-dominant. By definition, it is $\hgamma_t$, hence
the maximum of the first summand is $(\lambda,\hgamma_t)$.
Finally, as in Proposition~\ref{bia-exp}, we obtain that
the minimum in the RHS is equal to $(\rho_{\BH},\gamma_{t-} - \gamma)$.
Combining all of this together we obtain the result.
\end{proof}

\subsection{The small blocks}

Now we will give a description of the blocks $\SBS_\sj$.

Take $\sj,\sj' \in \SJ$ and assume that $\sj' < \sj$. As before
we write $\BH$ for $\BH_\sj$, $h$ for $h_\sj$, and $\BM$ for $\BM_\sj = \BL \cap \BH_\sj$.
In addition, we will write $\BH'$ for $\BH_{\sj'}$, $h'$ for $h_{\sj'}$, and $\BM'$ for $\BM_{\sj'} = \BL \cap \BH_{\sj'}$.
Similarly we denote by $\gamma$ and $\gamma'$ the weights such that $\delta = -h^*\gamma$ and $\delta' = - (h')^*\gamma'$
are the shapes of the corresponding cores. We number the orbits of $\BW_{\BM'}$
on  $\BW_{\BH'}\gamma'$
from $0$ to $m'$, and we denote
by $\gamma'_{t\pm}$ the $\BM'$-dominant
and antidominant representatives of these orbits.

The proof of the next two results is analogous to that of Propositions~\ref{bia-exp} and \ref{wtgt}.





\begin{proposition}\label{sbsie}
The inner part of the block $\SBS_\sj$ can be described by the following system of inequalities
\begin{equation}\label{sbti}
\SBSi_\sj = \left\{ \nu \in \SBi_\sj \ \left|\
\begin{array}{l}
\text{$(\sj - \sj')((h')^*\xi,(h')^*\gamma'_{t+}) + (\ji^*\gamma'_{t+},\nu) + \bd_{\sj'}^{t,-} <
(\rho_{\BH'},(h')^*(\gamma'_{t-} - \gamma'))$}\\
\qquad\text{for all $\sj' < \sj$ and for all $0 \le t \le m'$}
\end{array}
\right\}\right.,
\end{equation}
where for $\sj'<\sj$
\begin{equation}\label{daap}
\bd_{\sj'}^{t,-} := - \min  \{ (\ji^*\gamma'_{t-},\nu')\ |\ \nu' \in \SBSi_{\sj'} \}.
\end{equation}
\end{proposition}

%


\begin{proposition}\label{sbsoe}
The outer part of the block $\SBS_\sj$ can be described by the following system of inequalities
\begin{equation}\label{sbto}
\SBSo_\sj = \left\{ \lambda \in \SBo_\sj \ \left|\
\begin{array}{l}
\text{$(\lambda,\hgamma'_t) + (\sj - \sj')((h')^*\xi,(h')^*\gamma'_{t+}) + \bd_{\sj',\sj}^{t+} + \bd_{\sj'}^{t-} < (\rho_{\BH'},(h')^*(\gamma'_{t-} - \gamma'))$}\\
\qquad\text{for all $\sj' < \sj$ and for all $0 \le t \le m'$}
\end{array}
\right\}\right.,
\end{equation}
where
\begin{equation}\label{daam}
\bd_{\sj',\sj}^{t,+} := \max  \{ (\ji^*\gamma'_{t+},\nu)\ |\ \nu \in \SBSi_{\sj} \}
\end{equation}
and $\hgamma'_t$ is the $\BLo$-dominant representative in $\BW_{\BLo}h'_*(h')^*\gamma'_{t+}$.

\end{proposition}


\section{Explicit collections for classical groups}\label{ss-ex}

Now we will show that the construction of the previous section leads to
(conjecturally full) exceptional collections for isotropic Grassmannians of types $B$, $C$ and $D$, 
and potentially to many interesting collections in type $A$.


So, assume that $\BG$ is of type $B$, $C$ or $D$ and consider the standard numbering of the vertices of its Dynkin diagram.
$$
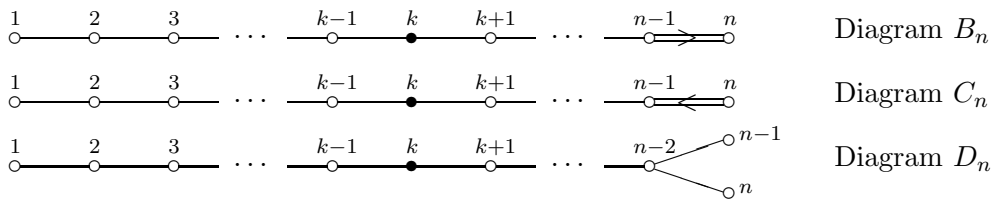
\begin{array}{ll}
\begin{picture}(300,10)
\multiput(0,0)(30,0){3}{\circle{4}}
\put(90,0){\hbox to 0mm{\hss\dots\hss}}
\put(120,0){\circle{4}}
\put(150,0){\circle*{4}}
\put(180,0){\circle{4}}
\put(210,0){\hbox to 0mm{\hss\dots\hss}}
\multiput(240,0)(30,0){2}{\circle{4}}
\multiput(2,0)(30,0){2}{\line(1,0){26}}
\multiput(122,0)(30,0){2}{\line(1,0){26}}
\multiput(62,0)(120,0){2}{\line(1,0){15}}
\multiput(118,0)(120,0){2}{\line(-1,0){15}}
\multiput(242,-1)(0,2){2}{\line(1,0){26}}
\put(250,-2.8){$>$}
\put(-2,5){$\scriptstyle{}1$}
\put(28,5){$\scriptstyle{}2$}
\put(58,5){$\scriptstyle{}3$}
\put(114,5){$\scriptstyle{}k-1$}
\put(148,5){$\scriptstyle{}k$}
\put(174,5){$\scriptstyle{}k+1$}
\put(234,5){$\scriptstyle{}n-1$}
\put(268,5){$\scriptstyle{}n$}
\end{picture}
&
\text{Diagram $B_n$}\\[10pt]
\begin{picture}(300,10)
\multiput(0,0)(30,0){3}{\circle{4}}
\put(90,0){\hbox to 0mm{\hss\dots\hss}}
\put(120,0){\circle{4}}
\put(150,0){\circle*{4}}
\put(180,0){\circle{4}}
\put(210,0){\hbox to 0mm{\hss\dots\hss}}
\multiput(240,0)(30,0){2}{\circle{4}}
\multiput(2,0)(30,0){2}{\line(1,0){26}}
\multiput(122,0)(30,0){2}{\line(1,0){26}}
\multiput(62,0)(120,0){2}{\line(1,0){15}}
\multiput(118,0)(120,0){2}{\line(-1,0){15}}
\multiput(242,-1)(0,2){2}{\line(1,0){26}}
\put(250,-2.8){$<$}
\put(-2,5){$\scriptstyle{}1$}
\put(28,5){$\scriptstyle{}2$}
\put(58,5){$\scriptstyle{}3$}
\put(114,5){$\scriptstyle{}k-1$}
\put(148,5){$\scriptstyle{}k$}
\put(174,5){$\scriptstyle{}k+1$}
\put(234,5){$\scriptstyle{}n-1$}
\put(268,5){$\scriptstyle{}n$}
\end{picture}
&
\text{Diagram $C_n$}\\
\begin{picture}(300,20)
\multiput(0,0)(30,0){3}{\circle{4}}
\put(90,0){\hbox to 0mm{\hss\dots\hss}}
\put(120,0){\circle{4}}
\put(150,0){\circle*{4}}
\put(180,0){\circle{4}}
\put(210,0){\hbox to 0mm{\hss\dots\hss}}
\put(240,0){\circle{4}}
\put(270,10){\circle{4}}
\put(270,-10){\circle{4}}
\multiput(2,0)(30,0){2}{\line(1,0){26}}
\multiput(122,0)(30,0){2}{\line(1,0){26}}
\multiput(62,0)(120,0){2}{\line(1,0){15}}
\multiput(118,0)(120,0){2}{\line(-1,0){15}}
\put(242,1){\line(3,1){26}}
\put(242,-1){\line(3,-1){26}}
\put(-2,5){$\scriptstyle{}1$}
\put(28,5){$\scriptstyle{}2$}
\put(58,5){$\scriptstyle{}3$}
\put(114,5){$\scriptstyle{}k-1$}
\put(148,5){$\scriptstyle{}k$}
\put(174,5){$\scriptstyle{}k+1$}
\put(234,5){$\scriptstyle{}n-2$}
\put(274,10){$\scriptstyle{}n-1$}
\put(274,-10){$\scriptstyle{}n$}
\end{picture}
&
\text{Diagram $D_n$}\\[10pt]
\end{array}
$$
To treat these cases simultaneously it is convenient to denote
\begin{equation}\label{e}
e = \begin{cases}
1/2, & \text{if $\BG$ is of type $B$,}\\
1, & \text{if $\BG$ is of type $C$,}\\
0, & \text{if $\BG$ is of type $D$.}
\end{cases}
\end{equation}
Then the weight lattice $P_\BG$ can be identified with the sublattice of $\QQ^n$ spanned by $\omega_i$
($1 \le i \le n$) with
$$
\begin{array}{lll}
\omega_i^{B,C,D} &=& (\underbrace{1,1,\dots,1}_i,\underbrace{0,0,\dots,0}_{n-i}),\quad 1\le i \le n-2+2e,\\
\omega_n^{B,D} &=& (1/2,1/2,\dots,1/2,1/2),\\
\omega_{n-1}^{D} &=& (1/2,1/2,\dots,1/2,-1/2).
\end{array}
$$
Let $k$ be the number of the vertex of the Dynkin diagram of $\BG$ corresponding
to the maximal parabolic subgroup $\BP$, so that $\xi = \omega_k$.

\subsection{Isotropic Grassmannians}\label{ss-ex-1}

First, we assume that
$$
k \le n + 2e - 2.
$$
In other words, $k \le n - 1$ for type $B$,
$k \le n$ for type $C$ and $k \le n-2$ for type $D$.
Then
$$
X := \BG/\BP = \begin{cases}
\OGr(k,2n+1), & \text{$k \le n-1$ ($\BG$ is of type $B_n$)}\\
\SGr(k,2n), & \text{$k \le n$\hphantom{${}-1$} ($\BG$ is of type $C_n$)}\\
\OGr(k,2n), & \text{$k \le n-2$ ($\BG$ is of type $D_n$)}
\end{cases}
$$
where $\OGr$ (resp., $\SGr$) denotes the orthogonal (resp., symplectic) isotropic Grassmannian.

Let $\Do$ be the component of $D \setminus \beta$ containing the vertices from $1$ to $k-1$. Then $b = k-1$ and $\Di$ is the component
containing vertices from $k+1$ to $n$. Note that $\ji^*$ is the projection onto the last $n-k$ coordinates
while $\jo^*$ is the projection onto the first $k-1$ coordinates with respect to the standard basis $\eps_1,\dots,\eps_n$ 
in $P_\BG = \QQ^n$. The simple roots are
$$
\begin{array}{l}
\alpha_i^{B,C,D} = \eps_i - \eps_{i+1},\quad 1\le i \le n-1,\\
\alpha_n^B = \eps_n,\qquad
\alpha_n^C = 2\eps_n,\qquad
\alpha_n^D = \eps_{n-1} + \eps_n.
\end{array}
$$
Note also that
$$
\rho = (n+e-1,n+e-2,\dots,e),
$$
thus $(\rho,\eps_i) = n+e-i$.

Now take any $a \le k-1$. Then the projection $h_a^*$ is the projection onto the last $n-a$ coordinates (it kills all $\eps_i$ with $i \le a$).
The simple root corresponding to $\BP$ is $\beta = \eps_k - \eps_{k+1}$,
so the maximal root of $\BH_a$ with the coefficient of $\beta$ equal to $1$
is $\bar\beta_a = \eps_{a+1} + \eps_{k+1}$,
so by Lemma~\ref{index} the index of the Grassmannian $\BH_a/(\BH_a \cap \BP)$ is
$$
r_a = (\rho,\beta + \bar\beta_a)/(\xi,\beta) = 2n + 2e - a - k - 1.
$$
In particular, we see that when $a$ decreases by 1, $r_a$ increases by $1$.
Also, note that $r_{k-1} = 2n + 2e - 2k$, while
$$
r = r_0 = 2n + 2e - k - 1.
$$
Further, the weight $\theta$ defined by~\eqref{theta} in this case is
$$
\theta = (\underbrace{0,0,\dots,0}_{k-1},1,\underbrace{0,0,\dots,0}_{n-k}).
$$
It follows that $(\theta,P_\BL) = \frac12\ZZ$ if $\BG$ is of type $B$ or $D$
and $(\theta,P_\BL) = \ZZ$ if $\BG$ is of type $C$ and
$$
\SJ = \begin{cases}
\frac12\ZZ \cap [0,2n-k-1/2], & \text{if $\BG$ is of type $B$}\\
\text{\hphantom{$\frac12$}}\ZZ \cap [0,2n-k], & \text{if $\BG$ is of type $C$}\\
\frac12\ZZ \cap [0,2n-k-3/2], & \text{if $\BG$ is of type $D$}\\
\end{cases}
$$
Applying~\eqref{aj} we conclude that
$$
a(\sj) = \begin{cases}
\lfloor\, \sj\,\rfloor, & \text{if $\sj < k$}\\
k-1, & \text{if $\sj \ge k$}
\end{cases}
$$

Now we are going to apply Propositions~\ref{bia-exp}, \ref{wtgt}, \ref{sbsie} and~\ref{sbsoe}.
We take
$$
\gamma_a = \alpha_a = \eps_a - \eps_{a+1}.
$$
Note that $\BW_{\BH_a}$ acts by permutations of the last $n-a$ coordinates and changes of signs of the
coordinates
(in case of type $D$ by pairwise changes of signs), while $\BW_{\BM_a}$ acts by permuting coordinates from $a+1$ to $k$
and from $k+1$ to $n$ separately and (pairwise) changes of signs only of the last $n-k$ coordinates.
Thus, the $\BW_{\BH_a}$-orbit of $\gamma_a$ consists of all vectors $\eps_a \pm \eps_i$, $a+1 \le i \le n$,
and it splits into three $\BW_{\BM_a}$-orbits:
$$
\{ \eps_a - \eps_i \}_{a+1 \le i \le k},\qquad
\{ \eps_a \pm \eps_i \}_{k+1 \le i \le n},\qquad\text{and}\qquad
\{ \eps_a + \eps_i \}_{a+1 \le i \le k}.
$$
Thus, using
the notation of section~\ref{ss-ed} we have $m = 2$ (unless $\BG$ has type $C$ and $k = n$
in which case the second orbit is empty and so $m = 1$), and the characteristic weights and
quantities from section~\ref{ss-ed} are given by the following table:
$$
\begin{array}{|c|c|c|c|c|c|c|c|c|c|}
\hline
t & \gamma_{t-} & (\rho_\BH,\gamma_{t-} - \gamma) & \gamma_{t+} & h_a^* \gamma_{t+} & \hgamma_t & (h_a^*\xi,h_a^*\gamma_{t+}) & \ji^*\gamma_{t+} & \ji^*\gamma_{t-} \\
\hline
0 & \eps_a - \eps_{a+1} & 0 & \eps_a - \eps_k & -\eps_k & -\eps_k & -1 & 0 & 0\\
\hline
1 & \eps_a - \eps_{k+1} & k-a & \eps_a + \eps_{k+1} & \eps_{k+1} & \eps_{k+1} & 0 & \eps_{k+1} & -\eps_{k+1} \\
\hline
2 & \eps_a + \eps_k & 2n+2e-a-k-1 & \eps_a + \eps_{a+1} & \eps_{a+1} & \eps_1 & 1 & 0 & 0 \\
\hline
\end{array}
$$
(if $\BG$ has type $C$ and $k = n$ then the line $t = 1$ should be omitted).

Applying Proposition~\ref{bia-exp} we obtain the following description of $\SBi_\sj$:
$$
\SBi_\sj = \{(\nu_{k+1},\dots,\nu_n) \in P_\BLi^+\ |\ 2\nu_{k+1} \le k-a(\sj)\ \text{and}\ \nu_i \equiv \sj \pmod \ZZ \}.
$$
Further, we apply~\eqref{dap} and compute
$$
d_\sj^{1,\pm} =
\{\sj\} + \lfloor (k - a(\sj))/2 - \{\sj\} \rfloor,
$$
(where $\{-\}$ stands for the fractional part)
and for other values of $t$ we have $d_\sj^{t,\pm} = 0$.

Now we can describe $\SBo_\sj$. Note that $(\hgamma_t,\Ker h_a^*) = 0$ unless $t = 2$. So, for $t = 0$ Proposition~\ref{wtgt}
gives an empty condition and for $t = 1$ we obtain the condition that $d_\sj^{1,+} + d_\sj^{1,-} \le k-a$ which holds by the definition
of $d_\sj^{1,\pm}$. Finally, the condition for $t = 2$ gives
$$
\SBo_\sj =  \{(\lambda_1,\dots,\lambda_{a(\sj)},0,\dots,0)\ |\ 2n + 2e - a(\sj) - k - 1 \ge \lambda_1 \ge \dots \ge \lambda_{a(\sj)} \ge 0\}.
$$
Note that the set $\SBo_\sj$ is the set of Young diagrams inscribed into the rectangle $a(\sj)\times (2n+2e-a(\sj)-k-1)$,
hence it is closed under taking subdiagrams. Thus, the second condition of Theorem~\ref{th-aw} is satisfied.
Since there are no very special elements by Lemma~\ref{vse}, the first condition is satisfied as well, so
Theorem applies, and we conclude that the block
$$
\SB_\sj =
\left\{(\lambda_1,\dots,\lambda_n) \in P_\BL^+\ \left|
\begin{array}{l}
2n + 2e + \sj - a(\sj) - k - 1 \ge \lambda_1 \ge \dots \ge \lambda_{a(\sj)} \ge \sj = \lambda_{a(\sj) + 1} = \dots = \lambda_k,\\
(k-a(\sj))/2 \ge \lambda_{k+1} \ge \dots \ge \lambda_n, \\
\lambda_1,\dots,\lambda_n \equiv \sj \pmod \ZZ
\end{array}
\right\}\right.
$$
is exceptional.

Now we are going to apply Proposition~\ref{sbsie}.
First, let us show that
\begin{equation}\label{Binn-equality-less-k}
\SBSi_\sj = \SBi_\sj\qquad\text{for $\sj < k$}
\end{equation}
and $\bd_\sj^{1-} = d_\sj^{1-} = \{\sj\} + \lfloor (k - a(\sj))/2 - \{\sj\} \rfloor$.
For this we can use induction on $\sj$. The base of induction, $\sj = 0$ is clear.
Assume that for all $\sj' < \sj$ the statement is proved. Then by Proposition~\ref{sbsie},
the additional condition defining $\SBSi_\sj$ is
$$
\nu_{k+1} + \{\sj'\} + \lfloor (k - a(\sj'))/2 - \{\sj'\} \rfloor < k - a(\sj').
$$
We claim that this condition is always satisfied for $\nu \in \SBi_\sj$. Indeed, we have
\begin{align}\label{long-inequality}
\{\sj\} + \lfloor (k - a(\sj))/2 - \{\sj\} \rfloor + \{\sj'\} + & \lfloor (k - a(\sj'))/2 - \{\sj'\} \rfloor \le \nonumber\\
& (k - a(\sj))/2 + (k - a(\sj'))/2 =
k - (a(\sj) + a(\sj'))/2 \le k - a(\sj'),
\end{align}
and the equality is possible only for if $a(\sj) = a(\sj')$ and both $(k-a(\sj))/2 - \{\sj\}$ and $(k-a(\sj))/2 - \{\sj'\}$ are integers.
But for $\sj',\sj < k$ one has $a(\sj) = \lfloor\, \sj \,\rfloor$, so the first condition shows that the integer parts of $\sj$ and $\sj'$ are equal,
while the second shows that the difference $\sj - \sj'$ is integer. This is possible only if
$\sj = \sj'$, which is a contradiction.
Hence, one of the inequalities in \eqref{long-inequality} is strict as we claimed. This finishes the proof of
\eqref{Binn-equality-less-k}.

Now let us check that
$$
\begin{array}{ll}
\SBSi_\sj = 0, & \text{for integer $\sj \ge k$ and}\\
\SBSi_\sj = \emptyset, & \text{for half-integer $\sj \ge k$.}
\end{array}
$$
Indeed, if $\sj$ is half-integer take $\sj' = k - 1/2$. Then
$\bd_{\sj'}^{1-} = \{\sj'\} + \lfloor (k-a(\sj')/2 - \{\sj'\} \rfloor = 1/2 + \lfloor 1/2 - 1/2 \rfloor = 1/2$,
so the inequality defining $\SBSi_\sj \subset \SBi_\sj$ is
$$
\nu_{k+1} + 1/2 < 1.
$$
On the other hand, $\nu_{k+1}$ should be a nonnegative half-integer, so we conclude that $\SBSi_\sj = \emptyset$.
For an integer $\sj \ge k$ we note that already $\SBi_\sj = 0$, so we only have to check that the inequality~\eqref{sbti}
is satisfied for $\nu = 0$. Indeed, if $\sj' < k$ then $a(\sj') \le k-1$, hence
$$
0 + \bd_{\sj'}^{1-} = \{\sj'\} + \lfloor (k-a(\sj'))/2 - \{\sj'\} \rfloor \le
\{\sj'\} + (k-a(\sj'))/2 - \{\sj'\} = (k-a(\sj'))/2 < k - a(\sj').
$$
Further, if $\sj' \ge k$ is a half-integer then as we already know $\SBSi_{\sj'}$ is empty, so $\bd_{\sj'}^{1,-} = -\infty$
and so we do not have a restriction on $\nu$. Finally, if $\sj' \ge k$ is integer then by induction hypothesis we have
$\bd_{\sj'}^{1-} = 0$ while $a(\sj') = k-1$, so the inequality $\nu_{k+1} + \bd_{\sj'}^{1-} < k - a(\sj')$ holds
in this case.

Now let us describe the outer parts of the blocks, $\SBSo_\sj$.
The inequality~\eqref{sbto} gives
$$
\lambda_1 + \sj - \sj' < 2n + 2e - a(\sj') - k - 1.
$$
It can be rewritten as
$$
\lambda_1 < 2n + 2e - \sj - k - 1 + (\sj' - a(\sj')).
$$
Since this inequality should hold for all $\sj' < \sj$, we can replace the last summand by its minimum,
which is equal to $0$. So, we conclude that the defining inequality of $\SBSo_\sj$ is
$\lambda_1 < 2n + 2e - \sj - k - 1$.
Since $\lambda_1$ should be an integer, this is equivalent to $\lambda_1 \le 2n + 2e - \lfloor\sj\rfloor - k - 2$.

Now we can write down the obtained answer. We denote by $\CA_\sj$ the subcategory in $\D(X)$
corresponding to the block $\SBS_\sj = \SBSo_\sj + \sj\xi + \SBSi_\sj$.

\begin{theorem}\label{bdnk}
Let $\BG$ be of type $B$ or $D$. Assume that $k \le n - 1$ for type $B$ and $k \le n-2$ for type $D$.
For each integer $t$, $0\le t\le k-1$, consider the subcategories $\CA_t$ and $\CA_{t+1/2}$
in $\D(X)$ defined by
$$
\renewcommand{\arraystretch}{1.5}\arraycolsep=2pt%
\begin{array}{lll}
\CA_t &=& \left\langle \CE^\lambda \ \left|\
\begin{array}{l}
2n + 2e - k - 2 \ge \lambda_1 \ge \dots \ge \lambda_{t} \ge t = \lambda_{t+1} = \dots = \lambda_k,\\
(k-t)/2 \ge \lambda_{k+1} \ge \dots \ge \lambda_n \ge (2e - 1)\lambda_{n-1},
\end{array}
\quad
\lambda_i \in \ZZ \right\rangle\right.,\\
\CA_{t+1/2} &=& \left\langle \CE^\lambda \ \left|\
\begin{array}{l}
2n + 2e - k - 3/2 \ge \lambda_1 \ge \dots \ge \lambda_{t} \ge t + 1/2 = \lambda_{t+1} = \dots = \lambda_k,\\
(k-t)/2 \ge \lambda_{k+1} \ge \dots \ge \lambda_n \ge (2e - 1)\lambda_{n-1},
\end{array}
\quad
\lambda_i \in 1/2 + \ZZ \right\rangle\right.,
\end{array}
$$
where $e$ is defined by~\eqref{e}. Also, for each integer $t$, $k \le t \le 2n+2e-k-2$, consider
the subcategory
$$
\CA_{t} = \left\langle \CE^\lambda \ \left|\
\begin{array}{l}
2n + 2e - k - 2 \ge \lambda_1 \ge \dots \ge \lambda_{k-1} \ge \lambda_k = t,\\
\lambda_{k+1} = \dots = \lambda_n = 0,
\end{array}
\quad
\lambda_i \in \ZZ \right\rangle\right..
$$
Then the collection of subcategories
$$
\begin{array}{c}
\CA_0,\CA_{1/2},\CA_1,\CA_{3/2},\dots,\CA_{k-1},\CA_{k-1/2},\CA_k,\CA_{k+1},\dots,\CA_{2n+2e-k-2}
\end{array}
$$
is semiorthogonal, and each subcategory is generated by an exceptional collection.
\end{theorem}

\begin{theorem}\label{cnk}
Assume $\BG$ is of type $C$ and $k \le n$. Consider the following subcategories in $\D(X)$
indexed by integers $t=0,\ldots,2n-k$:
$$
\renewcommand{\arraystretch}{1.5}%
\begin{array}{llll}
\CA_t &=& \left\langle \CE^\lambda \ \left|\
\begin{array}{l}
2n - k \ge \lambda_1 \ge \dots \ge \lambda_t \ge t = \lambda_{t+1} = \dots = \lambda_k,\\
\lfloor (k-t)/2 \rfloor \ge \lambda_{k+1} \ge \dots \ge \lambda_n \ge 0,
\end{array}
\right\rangle\right., & \text{for $t \le k-1$}\\
\CA_t &=& \left\langle \CE^\lambda \ \left|\
\begin{array}{l}
2n - k \ge \lambda_1 \ge \dots \ge \lambda_{k-1} \ge \lambda_k = t,\\
\lambda_{k+1} = \dots = \lambda_n = 0,
\end{array}
\right\rangle\right., & \text{for $t \ge k$}.
\end{array}
$$
Then the collection of subcategories
$$
\CA_0,\CA_1,\dots,\CA_{2n-k},
$$
is semiorthogonal, and each subcategory is generated by an exceptional collection.
\end{theorem}

\subsection{Orthogonal maximal isotropic Grassmannians}\label{ss-ex-2}

Note that if $\BG$ is of type $D$ and $k = n-1$ or $k = n$ then the Grassmannian $\BG/\BP$ is
isomorphic to the Grassmannian of type $B_{n-1}$ with $k = n-1$.
Thus, the only remaining case with $\BG$ classical
is when $\BG$ is of type $B_n$ and $k = n$, which we will now consider.
Note that in this case
$$
X = \BG/\BP = \OGr(n,2n+1).
$$
As before we take $\Do$ to be the component containing vertices from $1$ to $n-1$,
and thus $\Di = \emptyset$. Further, $\beta = \eps_n$, so $\bar\beta_a = \eps_{a+1}$
and
$$
r_a = (\rho,\beta + \bar\beta_a)/(\xi,\beta) = 2n - 2a.
$$
Hence, when $a$ increases by 1, the index decreases by 2.
In particular,
$$
r = r_0 = 2n.
$$
The weight $\theta$ defined by~\eqref{theta} is
$$
\theta = (0,0,\dots,0,2),
$$
hence $(\theta,P_\BL) = \ZZ$ and
$$
\SJ = \ZZ \cap [0,2n-1].
$$
Applying~\eqref{aj} we deduce that
$$
a(\sj) = \lfloor\, \sj/2\rfloor.
$$

As before we take $\gamma_a = \alpha_a = \eps_a - \eps_{a+1}$.
Note that $\BW_{\BH_a}$ acts by permutations of the last $n-a$ coordinates and by
changes of signs of the coordinates,
while $\BW_{\BM_a}$ acts just by permutations.
Thus, the $\BW_{\BH_a}$-orbit of $\gamma_a$ consists of all vectors $\eps_a \pm \eps_i$, $a+1 \le i \le n$,
and it splits into two $\BW_{\BM_a}$-orbits:
$$
\{ \eps_a - \eps_i \}_{a + 1 \le i \le n}
\qquad\text{and}\qquad
\{ \eps_a + \eps_i \}_{a + 1 \le i \le n}.
$$
Thus, using the notation of section~\ref{ss-ed} we have $m = 1$ and
$$
\begin{array}{|c|c|c|c|c|c|c|c|c|c|}
\hline
t & \gamma_{t-} & (\rho_\BH,\gamma_{t-} - \gamma) & \gamma_{t+} & h_a^* \gamma_{t+} & \hgamma_t & (h_a^*\xi,h_a^*\gamma_{t+}) \\
\hline
0 & \eps_a - \eps_{a+1} & 0 & \eps_a - \eps_n & -\eps_n & -\eps_n & -1/2 \\
\hline
1 & \eps_a + \eps_n & n-a & \eps_a + \eps_{a+1} & \eps_{a+1} & \eps_1 & 1/2 \\
\hline
\end{array}
$$

Since $P_\BLi = 0$ and $a(\sj) < k = n$ for all $\sj \in \SJ$, we have
$$
\SBi_\sj = 0\qquad\text{for all $\sj \in \SJ$.}
$$
In particular, $d_\sj^{t,\pm} = 0$ and thus
$$
\SBo_\sj = \{ n-a(\sj) \ge \lambda_1 \ge \dots \ge \lambda_{a(\sj)} \ge 0\}.
$$
Note that this is the set of Young diagrams inscribed into the rectangle $a(\sj) \times (n - a(\sj))$,
hence it is closed under taking subdiagrams. Thus, the second condition of Theorem~\ref{th-aw} is satisfied.
Since there are no very special elements by Lemma~\ref{vse}, the first condition is satisfied as well, so
Theorem applies, and we conclude that the block
$$
\SB_\sj = \SBo_\sj + \sj\xi = \left\{ (\lambda_1,\dots,\lambda_n)\ \left|\
\begin{array}{l}
n + \sj/2 - a(\sj) \ge \lambda_1 \ge \dots \ge \lambda_{a(\sj)} \ge \sj/2 = \lambda_{a(\sj)+1} = \dots = \lambda_n,\\
\lambda_i \equiv \sj/2 \pmod \ZZ
\end{array}
\right\}\right.
$$
is exceptional.

On the other hand, the condition~\eqref{sbto} gives
$\lambda_1 + (\sj - \sj')/2 < n - a(\sj') = n - \lfloor \sj'/2\rfloor$.
It can be rewritten as
$$
\lambda_1 < n - \sj/2 + \{\sj'/2\}.
$$
Since this should be satisfied for all $\sj' < \sj$, we conclude that
$\lambda_1 < n - \sj/2$. On the other hand, $\lambda_1$ should be an integer,
so we obtain $\lambda_1 \le n - 1 - \lfloor \sj/2 \rfloor$.

Now we can write down the obtained answer.
Recall that $\CA_\sj$ is the subcategory of $\D(X)$
corresponding to the block $\SBS_\sj = \SBSo_\sj + \sj\xi + \SBSi_\sj$.

\begin{theorem}\label{bnn}
Assume $\BG$ is of type $B_n$ and $k = n$.
Consider the following subcategories in $\D(X)$ (where $t$ is a nonnegative integer):
$$
\renewcommand{\arraystretch}{1.5}%
\begin{array}{lll}
\CA_{2t} &=& \langle \CE^\lambda \ |\
n - 1 \ge \lambda_1 \ge \dots \ge \lambda_t \ge t = \lambda_{t+1} = \dots = \lambda_n,
\quad
\lambda_i \in \ZZ \rangle,\\
\CA_{2t+1} &=& \langle \CE^\lambda \ |\
n - 1/2 \ge \lambda_1 \ge \dots \ge \lambda_t \ge t + 1/2 = \lambda_{t+1} = \dots = \lambda_n,
\quad
\lambda_i \in 1/2 + \ZZ \rangle.
\end{array}
$$
Then the collection of subcategories
$$
\CA_0,\CA_1,\dots,\CA_{2n-1},
$$
is semiorthogonal, and each subcategory is generated by an exceptional collection.
\end{theorem}


\subsection{Purity for maximal isotropic Grassmannians}\label{lag-purity-sec}

Recall that for an exceptional block $\SB$ the exceptional collection
$(\CE^{\lambda})_{\lambda\in\SB}$ is strong if and only if it
consists of vector bundles (see Proposition~\ref{E-vec-bun-prop}).
Using the explicit form of the blocks we can check
that this is true in the case of maximal isotropic Grassmannians (symplectic or orthogonal).

\begin{theorem} The exceptional collections of Theorem~\ref{cnk} for $k=n$ and of
Theorem~\ref{bnn} consist of vector bundles.
\end{theorem}
\begin{proof} By Proposition~\ref{path-closed-prop},
it is enough to check that for each of the blocks $\SB$ appearing in the collection the subquiver
$\CQ_\SB\subset \CQ$ contains entirely any path that starts and ends in $\CQ_\SB$.

First, let us consider the case when $\BG$ is of type $C_n$ and $k=n$ (so $\BG/\BP$ is
the Lagrangian Grassmannian $\SGr(n,2n)$). In this case $\BL=\GL_n$, so the quiver
$\CQ$ has vertices numbered by dominant weights of $\GL_n$ and there is an arrow
$\lambda\to\mu$ if and only if
$$\Hom_{\GL_n}(V^\mu,V^\lambda\otimes (V^{2\omega_1})^\vee)=
\Hom_{\GL_n}(V^\mu\otimes V^{2\omega_1}, V^\lambda)\neq 0.$$
Thus, if $\mu$ corresponds to a Young diagram then so does $\lambda$ and
$\mu$ is contained in $\lambda$ as a subdiagram. Since all the blocks
consist of Young diagrams and are closed under passing to subdiagrams, this implies
that they satisfy our condition on paths.

In the case when $\BG$ is of type $B_n$ and $k=n$ the Levi group $\BL$ is a twofold covering of $\GL_n$.
If $\sj$ is integer then all $\lambda$ and $\mu$ from this block are restricted from $\GL_n$
and the arrow $\lambda\to\mu$ in $\CQ$ exists if and only if
$$
\Hom_{\GL_n}(V^\mu\otimes V^{\omega_1}, V^\lambda)\neq 0,
$$
so the above argument shows that the block $\SB_\sj$ satisfies the condition on paths.
If $\sj$ is half-integer then $\SB_\sj = \SB_{\sj - 1/2} + \xi$, and since
the twist by $\xi$ is an autoequivalence, we conclude that the block $\SB_\sj$
satisfies the condition on paths as well.
\end{proof}


\begin{example}\label{eee}
Assume that $\BG$ is of type $C_4$ and $k = 3$, i.e. $X = \BG/\BP = \SGr(3,8)$,
and take the block
$$
\SB_1 = \{ 5 \ge \lambda_1 \ge 1 = \lambda_2 = \lambda_3,\ 1 \ge \lambda_4 \ge 0 \}.
$$
Note also that $\BL = \GL_3\times\SL_2$ and $V_\BL^{-\beta} = V_\BL^{0,0,-1;1}$.
In particular, we have a path 
$$
(3,1,1;1) \to (2,1,1;2) \to (1,1,1;1)
$$
in the quiver $\CQ$
that starts and ends in the block $\SB_1$, while its second vertex is not in the block.
So, the assumption of Proposition~\ref{path-closed-prop} does not hold. On the other hand,
the assumption of Proposition~\ref{Koszul-alg-prop}(i) is not satisfied as well. Indeed,
if $\lambda = (4,1,1;0)$ and $\mu = (1,1,1;1)$ and $v = s_3s_4 \in \SR_\BG^\BL$
then $v\rho - \rho = (0,0,-3;1)$ hence $V_\BL^\mu \subset V_\BL^\lambda \otimes V_\BL^{v\rho - \rho}$,
so by Proposition~\ref{extuu-g}(ii) we have $\Ext^2(V_\BL^\lambda,V_\BL^\mu) \ne 0$.
On the other hand, $\xi = (1,1,1,0)$, so $(\xi,\lambda) - (\xi,\mu) = 6 - 3 = 3$.
So, in the algebra $A_{\SB_1}$ its bigrading is $(2,3)$, while the first (in the cohomological grading)
component of the algebra has bigrading $(1,1)$ by Lemma~\ref{uugord}. Thus, the algebra cannot
be one-generated, and in particular, it is not Koszul.
%
%
%
%

On the other hand, one can check that the objects $\CE^\lambda$ with $\lambda \in \SB_1$ are still vector bundles. 
To illustrate what goes on let us consider the case $\lambda = (4,1,1;0)$. By definition, $\CE^{(4,1,1;0)}$ is the right
mutation of $\CU^{(4,1,1;0)}$ through the subcategory generated by $\CU^\mu$ with smaller $\mu$. This mutation
is a composition of several simple mutations. The first simple mutation is the right mutation through $\CU^{(3,1,1;1)}$.
It is easy to see that $\Ext^\bullet(\CU^{(4,1,1;0)},\CU^{(3,1,1;1)}) = \kk[-1]$, i.e.\ $\Ext^1$ is one-dimensional
and $\Ext^i = 0$ for $i \ne 1$. This means that the result of the first mutation $R_1$ fits into an exact sequence
\begin{equation*}
0 \to \CU^{(3,1,1;1)} \to R_1 \to \CU^{(4,1,1;0)} \to 0.
\end{equation*}
The second simple mutation is the right mutation of $R_1$ through $\CU^{(2,1,1;0)}$.
It is easy to see that $\Ext^\bullet(\CU^{(4,1,1;0)},\CU^{(2,1,1;0)}) = 0$ and $\Ext^\bullet(\CU^{(3,1,1;1)},\CU^{(2,1,1;0)}) = \kk[-1]$, 
hence $\Ext^\bullet(R_1,\CU^{(2,1,1;0)}) = \kk[-1]$, so the second mutation is again given by the extension
\begin{equation*}
0 \to \CU^{(2,1,1;0)} \to R_2 \to R_1 \to 0,
\end{equation*}
where $R_2$ is the result of the mutation. 
The last simple mutation is the right mutation of $R_2$ through $\CU^{(1,1,1;1)}$.
It is easy to see that $\Ext^\bullet(\CU^{(4,1,1;0)},\CU^{(1,1,1;1)}) = \kk[-2]$,
$\Ext^\bullet(\CU^{(3,1,1;1)},\CU^{(1,1,1;1)}) = 0$,  and 
$\Ext^\bullet(\CU^{(2,1,1;0)},\CU^{(1,1,1;1)}) = \kk[-1]$.
It follows that there is an exact sequence
\begin{multline}\label{4-term-ex-seq}
\qquad
0 \to \Ext^1(R_2,\CU^{(1,1,1;1)}) \to \Ext^1(\CU^{(2,1,1;1)},\CU^{(1,1,1;1)}) \\
\to \Ext^2(\CU^{(4,1,1;0)},\CU^{(1,1,1;1)}) \to \Ext^2(R_2,\CU^{(1,1,1;1)}) \to 0,
\qquad
\end{multline}
and that all other $\Ext$ spaces from $R_2$ to $\CU^{(1,1,1;1)}$ vanish.
The map in the middle is a map $\kk \to \kk$, and a direct computation shows that it is an isomorphism.
Thus, $\Ext^\bullet(R_2,\CU^{(1,1,1;1)}) = 0$, so the last mutation changes nothing and
$\CE^{(4,1,1;0)} = R_2$ has a filtration of length $3$ with factors being $\CU^{(2,1,1;0)}$, $\CU^{(3,1,1;1)}$, and $\CU^{(4,1,1;0)}$.
In particular, it is a vector bundle.
\end{example}

It is clear from the above argument that the key point is the surjectivity of the middle morphism 
in the 4-term exact sequence \eqref{4-term-ex-seq}. 
In fact, it is equivalent to the surjectivity of the Massey triple product
\begin{equation*}
\Ext^1(\CU^{(4,1,1;0)},\CU^{(3,1,1;1)}) \otimes
\Ext^1(\CU^{(3,1,1;1)},\CU^{(2,1,1;0)}) \otimes
\Ext^1(\CU^{(2,1,1;0)},\CU^{(1,1,1;1)}) 
\to \Ext^2(\CU^{(4,1,1;0)},\CU^{(1,1,1;1)}).
\end{equation*}
Since the Massey products are induced by the higher products in the natural $A_\infty$-structure of the algebra~$A_{\SB_1}$,
this surjectivity can be reinterpreted as the fact that the algebra $A_{\SB_1}$ is one-generated as an $A_\infty$-algebra.
This leads to the following Conjecture.

\begin{conjecture}\label{ainfty}
The algebra $A_\SB$ is one-generated as an $A_\infty$ algebra. Its Koszul dual is a usual algebra.
\end{conjecture}

This Conjecture implies the purity and strongness of the collections $\CE^\lambda$.

\subsection{Numbers of objects}

It turns out that the collections constructed in sections~\ref{ss-ex-1} and~\ref{ss-ex-2}
contain the maximal possible number of objects.
It is well known that the rank of Grothendieck group of $\BG/\BP$ is equal to the cardinality of 
$\BW_\BG/\BW_\BL$
%
%
(this rank is equal to the rank of the homology group of $X$ due to the Bruhat cell decomposition,
and the homology of $X$ was computed in \cite[Prop.\ 5.2]{BGG}).
In the case of the series $B$, $C$ and $D$ these ranks are given by
$$
r(n,k)={n\choose k}\cdot 2^k,
$$
where in the case of type $D$ we assume that  $k \le n-2$
(as was explained before, for type $D$ we do not need to consider the case $k=n-1$ or $n$).

\nc{\la}{\lambda}

\begin{proposition}\label{nbcd}
The total number of objects in the collections of Theorems~$\ref{bdnk}$, $\ref{cnk}$
and~$\ref{bnn}$ equals the rank of the Grothendieck group of the corresponding Grassmannian.
\end{proposition}
\begin{proof}
Let us denote
$$c_k(n)=|\{n\ge\la_1\ge\ldots\ge\la_k\ge 0, \ \la_i\in\ZZ\}|={n+k\choose k}.$$
We will consider the types $B$, $C$ and $D$ separately.

{\bf 1. Type $B_n$, $k\le n-1$}. In this case we have
$$
\begin{array}{llll}
|\SBS_t| &=& c_{t}(2n-k-1-t)c_{n-k}(\lfloor (k-t)/2 \rfloor),\qquad & \text{for integer $0 \le t \le k-1$}\\
|\SBS_{t+1/2}| &=& c_{t}(2n-k-1-t)c_{n-k}(\lfloor (k-t-1)/2 \rfloor),\qquad & \text{for integer $0 \le t \le k-1$}\\
|\SBS_t| &=& c_{k-1}(2n-k-1-t),\qquad & \text{for integer $k \le t \le 2n-k-1$}\\
\end{array}
$$
Hence, the total number of objects in the collection of Theorem~\ref{bdnk} in this case is
$$
N^B(n,k)=
\sum_{t=0}^{k-1}c_{t}(2n-k-1-t)\cdot\left(c_{n-k}(\lfloor (k-t)/2 \rfloor)+c_{n-k}(\lfloor (k-t-1)/2 \rfloor)\right) +
\sum_{t=k}^{2n-k-1}c_{k-1}(2n-k-1-t).
$$
But
$$
\sum_{t=k}^{2n-k-1}c_{k-1}(2n-k-1-t)=
\sum_{i=0}^{2n-2k-1}c_{k-1}(i)=
\sum_{i=0}^{2n-2k-1}{k-1+i\choose k-1}=
{2n-k-1\choose k}=c_k(2n-2k-1).
$$
Thus,
\begin{multline*}
N^B(n,k) = \\
\sum_{t=0}^{k-1}c_{t}(2n-k-1-t)\cdot\left(c_{n-k}(\lfloor (k-t)/2 \rfloor)+c_{n-k}(\lfloor (k-t-1)/2 \rfloor)
\right)+c_k(2n-2k-1)=\\
\sum_{t=0}^{k}{2n-k-1\choose t}\cdot\left(c_{n-k}(\lfloor (k-t)/2 \rfloor)+c_{n-k}(\lfloor (k-t-1)/2 \rfloor)
\right).
\end{multline*}
Hence, $N^B(n,k)$ is the coefficient of $x^k$ in $(1+x)^{2n-k-1}f^B_{n-k}(x)$, where
$$f^B_{n-k}(x)=\sum_{i\ge 0}\left(c_{n-k}(\lfloor i/2 \rfloor)+c_{n-k}(\lfloor (i-1)/2 \rfloor)\right)x^i=
(1+2x+x^2)\cdot \sum_{j\ge 0}c_{n-k}(j)x^{2j}=\frac{(1+x)^2}{(1-x^2)^{n-k+1}}.$$
Therefore, $N^B(n,k)$ is the coefficient of $x^k$ in
$$\frac{(1+x)^{2n-k+1}}{(1-x^2)^{n-k+1}}=\frac{(1+x)^n}{(1-x)^{n-k+1}}=
(1+x)^n\cdot\sum_{i\ge 0}{n-k+i\choose i}x^i.$$
Finally, this gives
$$N^B(n,k)=\sum_{i=0}^k{n\choose k-i}{n-k+i\choose i}=\sum_{i=0}^k\frac{n!}{(k-i)!i!(n-k)!}=
{n\choose k}\cdot \sum_{i=0}^k{k\choose i}={n\choose k}\cdot 2^k.$$

{\bf 1'. Type $B_n$, $k=n$}.
In this case
$$
|\SBS_{2t}|=|\SBS_{2t+1}|=c_t(n-t-1)={n-1\choose t},
$$
and the total number of objects is
$$
N^B(n,n)=2\sum_{t=0}^{n-1}{n-1\choose t}=2\cdot 2^{n-1}=2^n.
$$

{\bf 2. Type $C_n$}.
We have
$$
\begin{array}{llll}
|\SBS_t| &=& c_t(2n-k-t)\cdot c_{n-k}(\lfloor (k-t)/2 \rfloor),\qquad & \text{for integer $0 \le t \le k-1$}\\
|\SBS_t| &=& c_{k-1}(2n-k-t),\qquad & \text{for integer $k \le t \le 2n - k$}
\end{array}
$$
Thus, the total number of objects is
\begin{multline*}
N^C(n,k) =
\sum_{t=0}^{k-1}c_{t}(2n-k-t)\cdot c_{n-k}(\lfloor (k-t)/2 \rfloor)
+\sum_{t=k}^{2n-k}c_{k-1}(2n-k-t)=\\
\sum_{t=0}^{k-1}c_{t}(2n-k-t)\cdot c_{n-k}(\lfloor (k-t)/2 \rfloor)+c_k(2n-2k)=
\sum_{t=0}^{k}c_{t}(2n-k-t)\cdot c_{n-k}(\lfloor (k-t)/2 \rfloor).
\end{multline*}
In other words, $N^C(n,k)$ is the coefficient of $x^k$ in $(1+x)^{2n-k}f^C_{n-k}(x)$, where
$$f^C_{n-k}(x)=\sum_{i\ge 0}c_{n-k}(\lfloor i/2 \rfloor)x^i=
(1+x)\sum_{j\ge 0}c_{n-k}(j)x^{2j}=\frac{(1+x)}{(1-x^2)^{n-k+1}}.$$
Therefore, $N^C(n,k)$ is the coefficient of $x^k$ in
$(1+x)^{2n-k+1}\cdot(1-x^2)^{-(n-k+1)}$, so we get
$$N^C(n,k)=N^B(n,k)={n\choose k}\cdot 2^k.$$

{\bf 3. Type $D_n$, $k\le n-2$}.
First, we observe that
\begin{multline*}
s_k(n):=|\{n\ge\la_1\ge\ldots\ge \la_k\ge -\la_{k-1},\ \la_i\in\ZZ\}|= \\
\sum_{p\ge 0}
|\{n\ge\la_1\ge\ldots\ge \la_{k-1}=p, \la_i\in\ZZ\}|\cdot (2p+1)=
\sum_{p\ge 0}(2p+1)c_{k-2}(n-p),
\end{multline*}
and so
$$\sum_{n\ge 0}s_k(n)x^n=
\left(\sum_{p\ge 0}(2p+1)x^p\right)\cdot\frac{1}{(1-x)^{k-1}}=\frac{1+x}{(1-x)^{k+1}}.$$
Similarly,
\begin{multline*}
t_k(n):=|\{n+1/2\ge\la_1\ge\ldots\ge \la_k\ge -\la_{k-1},\ \la_i\in 1/2+\ZZ\}|=\\
\sum_{p\ge 0}
|\{n+1/2\ge\la_1\ge\ldots\ge \la_{k-1}=p+1/2, \la_i\in 1/2+\ZZ\}|\cdot (2p+2)=
\sum_{p\ge 0}(2p+2)c_{k-2}(n-p),
\end{multline*}
and so
$$\sum_{n\ge 0}t_k(n)x^n=
\left(\sum_{p\ge 0}(2p+2)x^p\right)\cdot\frac{1}{(1-x)^{k-1}}=\frac{2}{(1-x)^{k+1}}.$$
Now
$$
\begin{array}{llll}
|\SBS_t| &=& c_{t}(2n-k-2-t)s_{n-k}(\lfloor (k-t)/2 \rfloor),\qquad & \text{for integer $0 \le t \le k-1$}\\
|\SBS_{t+1/2}| &=& c_{t}(2n-k-2-t)t_{n-k}(\lfloor (k-t-1)/2 \rfloor),\qquad & \text{for integer $0 \le t \le k-1$}\\
|\SBS_{t}| &=& c_{k-1}(2n-k-2-t),\qquad & \text{for integer $k \le t \le 2n-k-2$}
\end{array}
$$
Hence, the total number is
\begin{multline*}
N^D(n,k)= \\
\sum_{t=0}^{k-1}c_{t}(2n-k-2-t)\cdot\left(s_{n-k}(\lfloor (k-t)/2 \rfloor)+t_{n-k}(\lfloor (k-t-1)/2 \rfloor)
\right)
+\sum_{t=k}^{2n-k-2}c_{k-1}(2n-k-2-t)= \\
\sum_{t=0}^{k}c_{t}(2n-k-2-t)\cdot\left(s_{n-k}(\lfloor (k-t)/2 \rfloor)+t_{n-k}(\lfloor (k-t-1)/2 \rfloor)
\right).
\end{multline*}
Thus, $N^D(n,k)$ is the coefficient of $x^k$ in $(1+x)^{2n-k-2}f^D_{n-k}(x)$, where
\begin{multline*}
f^D_{n-k}(x)=
\sum_{i\ge 0}\left(s_{n-k}(\lfloor i/2 \rfloor)+t_{n-k}(\lfloor (i-1)/2 \rfloor)\right)x^i= \\
(1+x)\cdot \sum_{j\ge 0}s_{n-k}(j)x^{2j} +x(1+x)\cdot \sum_{j\ge 0}t_{n-k}(j)x^{2j}= \\
\frac{(1+x)(1+x^2)}{(1-x^2)^{n-k+1}}+
\frac{2(1+x)x}{(1-x^2)^{n-k+1}}=
\frac{(1+x)^3}{(1-x^2)^{n-k+1}}.
\end{multline*}
Therefore, $N^D(n,k)$ is the coefficient of $x^k$ in
$(1+x)^{2n-k+1}(1-x^2)^{-(n-k+1)}$
which gives
$$N^D(n,k)=N^B(n,k)={n\choose k}\cdot 2^k.
$$
This completes the proof.
\end{proof}

\subsection{Proofs}\label{Proofs-sec}

Here we explain how the results of the paper imply the Theorems from the Introduction.

\begin{proof}[Proof of Theorem~\ref{main-th}]
The exceptional collections are constructed in Theorems~\ref{bdnk}, \ref{cnk}, and~\ref{bnn}.
They have equivariant structure by construction. The number of objects equals the rank
of the Grothendieck group by Proposition~\ref{nbcd}.
\end{proof}

\begin{proof}[Proof of Corollary~\ref{dbfib}]
Recall that $Y = \CG\times_\BG(\BG/\BP) = (\CG \times (\BG/\BP)) /\BG$,
with respect to the natural right action of $\BG$ on $\CG$ and the left action on $\BG/\BP$.
By~\cite{El}, Theorem 9.6, the derived category $\D(Y)$ is equivalent to $\D(\CG\times(\BG/\BP))^\BG$,
the category of $\BG$-equivariant objects in $\D(\CG\times(\BG/\BP))$.
Consider the object $\CO_\CG \boxtimes \CE^\lambda \in \D(\CG\times(\BG/\BP))$ with its natural
$\BG$-equivariant structure. By the above observation it gives an object $\CE_Y^\lambda \in \D(Y)$
such that for any point $x \in X$ we have
$$
(\CE_Y^\lambda)_{|p^{-1}(x)} \cong \CE^\lambda.
$$
Thus we can apply Theorem 3.1 from~\cite{Sam2} and conclude that the functors
$$
\Phi^\lambda : \D(X) \to \D(Y),
\qquad
F \mapsto p^*F \otimes \CE_Y^\lambda
$$
are fully faithful and subcategories $\Phi_\lambda(\D(X)) \subset \D(Y)$ are semiorthogonal.
This means that we have a semiorthogonal decomposition
$$
\D^b(Y) = \langle \{ \Phi^\lambda(\D(X)) \}_{\lambda \in \SB}, \CA \rangle,
$$
where $\CA = \cap_{\lambda \in \SB} {}^\perp \Phi^\lambda(\D(X))$. Now if $X$ has an exceptional collection $F_i$
of length $N = \rk K_0(X)$ then the objects $p^*F_i \otimes \CE_Y^\lambda$ form an exceptional collection of length
$N\cdot \#\SB$ in $\D(Y)$, so if $\#\SB = \rk K_0(\BG/\BP)$ then this number equals $\rk K_0(X)\cdot \rk K_0(\BG/\BP) = \rk K_0(Y)$,
so we have an exceptional collection of expected length on $Y$.
\end{proof}

\begin{proof}[Proof of Theorem~\ref{mainthm}]
Part $(i)$ is given by Theorem~\ref{aap-ec}.
Part $(ii)$ follows from Proposition~\ref{crit-eb} combined with Proposition~\ref{inv-cond-prop} and Theorem~\ref{th-aw}.
Part $(iii)$ is a combination of Theorems~\ref{bdnk}, \ref{cnk}, and~\ref{bnn} with Proposition~\ref{nbcd}.
\end{proof}

\begin{proof}[Proof of Theorem~\ref{sp-equiv}]
This is just Proposition~\ref{E-vec-bun-prop}.
\end{proof}

\subsection{Usual Grassmannians}\label{s-typea}

In this section we speculate that our construction might still work
with a certain weakening of the assumption~\eqref{a-ass} (so that $\Do$ is not necessarily
connected).
Namely, we consider the case $X = \Gr(k,n)$, the usual Grassmannian,
and apply formally the procedure of section~\ref{ss-ckb} to the data for which
\eqref{a-ass} does not hold
to construct collections of expected length in $\D^b(X)$.
Of course, our proof of part~(b)
of the criterion of exceptionality (see Proposition \ref{crit-eb})
does not work in this situation, so we do not
have a proof of the exceptionality of this collection. However, we believe
that all these collections are exceptional and full.

Since the result of this section is only conjectural, we skip the intermediate calculations
(which are analogous to those for isotropic Grassmannians) and only state the final answer.


Let $\BG = \SL_n$ and $\BL = (\GL_k\times\GL_l) \cap \SL_n$ ($n = k + l$).
In the framework of the paper we could take $\Do$ to be either of the two connected components of $D_\BG \setminus \beta$.
Let us take instead $\Do$ to be {\em the union} of both, that is $\Do = D_\BG \setminus \beta$.
Of course we violate here the assumption~\eqref{a-ass}.

%

Moreover, we arbitrarily renumber the vertices of $D_\BG$ in such a way that
$D_a=D_\BG \setminus \{1,\dots,a\}$
is always connected and contains $\beta = \alpha_{n-1}$. In other words, to obtain
from $D_\BG$ the chain of Dynkin diagrams $D_a$
we keep chopping off one of the end-points of
the diagram until only $\beta$ is left.

It is clear that such renumberings are in a bijection
with isotopy classes of monotone curves $C$ in a $k \times l$ rectangle on an integer grid
going from the point $(k,l)$ to the point $(0,0)$ and not passing through integer points.
We will describe a conjectural exceptional collection corresponding to an isotopy class of such a curve.

Moreover, in fact we will allow the curve to pass through integer points (this corresponds to
allowing to chop off both end-points simultaneously).

So, assume we are given such a curve $C$. Consider the sequence of points $Q_0,Q_1,\dots,Q_m$
of intersection of $C$ with the edges of the grid squares (some of the points $Q_i$ can lie
at the vertices of the squares) and let $(x_i,y_i)$ be the coordinates of $Q_i$.
Set
$$
a_i = \lfloor x_i \rfloor,\qquad
b_i = \lfloor y_i \rfloor,\qquad
c_i = k - \lceil x_i \rceil,\qquad
d_i = l - \lceil y_i \rceil.
$$
Then consider the blocks
\begin{equation}\label{ba}
\SB_i = \left\{ \begin{array}{l} d_i + i \ge \lambda_1 \ge \dots \ge \lambda_{a_i} \ge i = \lambda_{a_i+1} = \dots = \lambda_k,\\
\lambda_{k+1} = \dots = \lambda_{n-b_i} = 0 \ge \lambda_{n-b_i+1} \ge \dots \ge \lambda_n \ge -c_i \end{array} \right\}
\end{equation}
(in particular, $\SB_0 = \{0\}$).
Note that the total number of weights in those blocks is
$$
\# \left( \SB_0 \sqcup \SB_1 \sqcup \dots \sqcup \SB_m \right) =
\sum_{i=0}^m{a_i + d_i \choose a_i}{b_i + c_i \choose b_i} = {k + l \choose k},
$$
which is the rank of the Grothendieck group of $X = \Gr(k,n)$.
The equality above has a simple combinatorial proof --- the RHS is the number of Young diagrams inscribed in the rectangle,
we divide the set of all such diagrams into subsets numbered by the point of intersection of the border of the diagram
with the curve $C$, the summands in the LHS correspond to the parts of this decomposition.

We have the following

\begin{conjecture}
The blocks $\SB_i$ given by~\eqref{ba} are exceptional and the collection
$\langle \CA_0,\CA_1,\dots,\CA_m \rangle$ with subcategories $\CA_i = \langle \CU^\lambda \rangle_{\lambda \in \SB_i}$
is a semiorthogonal decomposition of $\D^b(\Gr(k,n))$,
each component of which is generated by an exceptional collection.
\end{conjecture}

\begin{remark}
One special case is interesting. Assume $l = k$, and take for $C$ the segment of the straight line
from $(k,k)$ to $(0,0)$. Then $m = k$ and $Q_i = (i,i)$ so that $a_i = b_i = i$, $c_i = d_i = k - i$.
The corresponding exceptional collection is invariant with respect to the outer automorphism
of $\Gr(k,2k)$ (passing to orthogonal complement with respect to a nondegenerate bilinear form).
\end{remark}



\section{Appendix. Key technical Proposition}\label{ss-kp}

In this Appendix we prove a certain auxiliary result on $\GL_n$-representations.

For a dominant weight $\lambda=(\lambda_1\ge \ldots\ge \lambda_n)$ of $\GL_n$ we 
denote by $V^\lambda$ the corresponding irreducible $\GL_n$-representation.
We write $\lambda\ge 0$ (and say that $\lambda$ is {\sf nonnegative}) if $\lambda_n\ge 0$. 
Such weights correspond to partitions with at most $n$ parts.
Let ${w}_0$ denote the longest element of the symmetric group $\SSS_n$, i.e.\ the permutation which takes $i$ to $n+1-i$ for all $i$.
 
For an integer $a$, $0\le a\le n$, and an integer $l\ge 0$, let $\Pi^{a}_{-l}$ be the projector on the category
of $\GL_n$-representations which
acts identically on $V^{\la}$, where $\la_{a+1}=\ldots=\la_n=-l$,
and sends all other irreducible representations to zero.
We say that a map of $\GL_n$-representations is a $\Pi^a_{-l}$-isomorphism (resp.\ $\Pi^a_{-l}$-injection)
if applying $\Pi^a_{-l}$ to this map we get an isomorphism (resp.\ injection).

The main result of this Appendix is the following

\begin{proposition}\label{GL-lem-gen}
Fix an integer $a$, $0\le a\le n$.
Let $\kappa$ be a partition with at most $a$ parts, and
let $\tau$ be a partition with at most $n-a$ parts (both viewed as weights of $\GL_n$).
Finally, let $W$ be a representation which is a direct summand of $V^{\otimes N}$,
where $V$ is the standard $n$-dimensional representation of $\GL_n$.
Then the natural map
\begin{equation}\label{st0}
V^{\kappa-{w}_0\tau}\otimes W \to V^{\kappa}\otimes V^{-{w}_0\tau}\otimes W \to
V^{\kappa}\otimes \Pi^a_0(V^{-{w}_0\tau}\otimes W)  
\end{equation}
is a $\Pi^a_0$-isomorphism.
\end{proposition}

The following Corollary of this Proposition is used in section~\ref{ss-proof-cc}.

\begin{corollary}\label{GL-lem-cor} Fix $a$, $0\le a\le n-1$.
Let $\kappa$ be a partition with at most $a$ parts, $\tau$ a partition with at most $n-a$ parts,
and $\mu$ a partition with at most $n$ parts.
Then the natural map
$$
V^{\kappa-{w}_0\tau}\otimes V^\mu \to
V^{\kappa}\otimes V^{-{w}_0\tau}\otimes V^\mu \to
V^{\kappa}\otimes \Pi^{a}_{-l}(V^{-{w}_0\tau}\otimes V^\mu)$$
induces an isomorphism after applying $\Pi^a_{-l}$.
\end{corollary}
\begin{proof}
Denote by $(l)$ the autoequivalence of the category of representations of $\GL_n$ that takes 
a representation with a highest weight $\lambda = (\lambda_1,\dots,\lambda_n)$ 
to the representation with the highest weight $(\lambda_1+l,\dots,\lambda_n+l)$.
In other words, it is the twist by $(\det V)^{\otimes l}$.
Then for $W = V^\mu(l)$ we have
\begin{align*}
& \Pi^a_{-l}(V^{\kappa-w_0\tau}\otimes V^\mu)(l) =
\Pi^a_0(V^{\kappa-w_0\tau}\otimes W),\\
& \Pi^a_{-l}(V^{\kappa} \otimes V^{-w_0\tau}\otimes V^\mu)(l) =
\Pi^a_0(V^{\kappa} \otimes V^{-w_0\tau}\otimes W),\\
& \Pi^a_{-l}(V^{\kappa} \otimes \Pi^a_{-l}(V^{-w_0\tau}\otimes V^\mu))(l) =
\Pi^a_0(V^{\kappa} \otimes \Pi^a_{-l}(V^{-w_0\tau}\otimes V^\mu)(l)) =
\Pi^a_0(V^{\kappa} \otimes \Pi^a_0(V^{-w_0\tau}\otimes W)),
\end{align*}
so applying $\Pi^a_{-l}$ to the map in the Corollary and twisting by $(l)$ we obtain the map~\eqref{st0}
acted upon by~$\Pi^a_0$. The latter is an isomorphism by Proposition~\ref{GL-lem-gen},
hence the former is an isomorphism as well.
%
\end{proof}

We start the proof of Proposition~\ref{GL-lem-gen} with the following numerical observation.

\begin{lemma}\label{tableau-lem}
Under the assumptions of Proposition \ref{GL-lem-gen} one has
\begin{equation*}
\dim \Pi^a_0(V^{\kappa-{w}_0\tau}\otimes W)=\dim 
\Pi^a_0\left(V^{\kappa}\otimes \Pi^a_0(V^{-{w}_0\tau}\otimes W)\right).
\end{equation*}
\end{lemma}
\begin{proof}
It is enough to check that the multiplicities of $V^\mu$,
where $\mu$ is a partition with at most $a$ parts, in $V^{\kappa-{w}_0\tau}\otimes W$
and in $V^{\kappa}\otimes \Pi^a_0(V^{-{w}_0\tau}\otimes W)$
are equal.
To this end we replace $W$ with any of its irreducible summand of the form $V^{\lambda}$,
where $\lambda\ge 0$, and apply the Littlewood--Richardson rule.
The dimension of the space $\Hom(V^{\mu},V^{\kappa-w_0\tau}\ot V^{\lambda})$ is given by the number
of semistandard skew tableaux $S$ of shape $(\mu)\setminus(\kappa-w_0\tau)$
with the content of weight $\lambda$, satisfying the lattice permutation condition.
Every such skew tableau contains a skew subtableau $S'$ of shape $\mu\setminus\kappa$
that still satisfies the lattice permutation condition. Let $\nu$ be the weight of the content of $S'$.
Then to give $S$ is the same as to give $\nu\subset\lambda$ together with a pair:

\noindent
(i) a semistandard skew tableau of shape $\mu\setminus\kappa$ with content of weight $\nu$,

\noindent
(ii) a semistandard skew tableau of shape $\nu\setminus(-w_0\tau)$ with content $\lambda$.

\noindent
Let $N_1$ (resp., $N_2$) be the number of choices in (i) (resp., in (ii)).
We have
$$N_1=\dim\Hom(V^{\mu},V^{\kappa}\ot V^{\nu}).$$
On the other hand,
$$N_2=\dim\Hom(V^{\nu},V^{-w_0\tau}\ot V^{\lambda}).$$
Thus, the above argument gives the equality
\begin{equation}\label{dimension-decomposition}
\dim \Hom(V^{\mu},V^{\kappa-w_0\tau}\ot V^{\lambda})=
\sum_{\nu\ge 0, \nu\subset\mu,\nu\subset\lambda}
\dim(\Hom(V^{\mu},V^{\kappa}\ot V^{\nu}))\cdot \dim(\Hom(V^{\nu},V^{-w_0\tau}\ot V^{\lambda})).
\end{equation}
Note that the condition $\nu\subset\mu$ here is automatic since otherwise
$\Hom(V^{\mu},V^{\kappa}\ot V^{\nu})$ is zero.
On the other hand, we have a decomposition
\begin{equation}\label{kappa-tau-lambda-tensor-decomposition}
\Hom(V^{\mu},V^{\kappa}\ot\Pi^a_0(V^{-w_0\tau}\ot V^{\lambda}))=
\bigoplus_{\nu\ge 0, \nu\subset\mu, \nu\subset\lambda}
\Hom(V^{\mu},V^{\kappa}\ot V^{\nu})\ot \Hom(V^{\nu},V^{-w_0\tau}\ot V^{\lambda}).
\end{equation}
Indeed, the summation over $\nu\ge 0$ in the right-hand side appears from decomposing
$\Pi^a_0(V^{-w_0\tau}\ot V^{\lambda})$ into irreducibles. The condition $\nu\subset\mu$
can be added for the same reason as before, and the condition $\nu\subset\lambda$ is added
because otherwise $\Hom(V^{\nu}, V^{-w_0\tau}\ot V^{\lambda})\neq 0$ vanishes.
According to the definition of $\Pi^a_0$ we also have to requre $\nu$ to have at most $a$ parts,
but this follows from the inclusion $\nu\subset\mu$. Comparing the dimensions in
\eqref{kappa-tau-lambda-tensor-decomposition} with \eqref{dimension-decomposition}, we get
the required equality.
\end{proof}


The above Lemma reduces the proof of Proposition \ref{GL-lem-gen}
to showing that the map \eqref{st0} is $\Pi^a_0$-injective.
We will deduce this injectivity from a more general Proposition~\ref{glprop} below.
To state it we need more notation.

Let us define the {\sf depth} of a dominant weight $\lambda=(\lambda_1\ge \ldots\ge \lambda_n)$ of $\GL_n$ 
as the sum of absolute values of all its negative entires. In other words, we take $1 \le i \le n$ such that
$\lambda_i \ge 0 \ge \lambda_{i+1}$, and set
\begin{equation*}
\depth(\lambda) = - \lambda_{i+1} - \dots - \lambda_n.
\end{equation*}
Note that the depth is always nonnegative, and it is zero if and only if $\lambda\ge 0$.

Let $\Pi_d$ be the the projector on the category of representations of $\GL_n$ which acts identically
on all $V^\lambda$ with $\depth(\lambda) = d$, and sends all other irreducible representations to zero. 
Also, set $\Pi_{\ge d_0} := \sum_{d \ge d_0} \Pi_d$.

Consider the $\GL_n$-representations
\begin{equation*}
V_p := V^{\otimes p}
\qquad\text{and}\qquad
V_{p,q} := V^{\otimes p} \otimes (V^*)^{\otimes q}.
\end{equation*}
We will derive the $\Pi^a_0$-injectivity of \eqref{st0} from the following result.

\begin{proposition}\label{glprop}
Fix integers $k,t,N \ge 0$. The natural map
\begin{equation}\label{mainmap}
\Pi_t(V_{k,t}) \otimes V_{N} \to V_{k+N,t} \to V_{k} \otimes \Pi_0(V_{N,t})
\end{equation}
is $\Pi_0$-injective, i.e.\ it becomes injective after applying $\Pi_0$.
\end{proposition}

To prove Proposition~\ref{glprop} we will use some simple facts about the partial contraction maps between
the $\GL_n$-representations $V_{p,q}$.
First, let us consider the partial trace map $\Tr_{i,j}:V_{p,q} \to V_{p-1,q-1}$ given by
\begin{equation}\label{trij}
\Tr_{i,j}((v_1\otimes \dots \otimes v_p) \otimes (f_1\otimes \dots \otimes f_q)) = 
f_j(v_i) v_1 \otimes \dots \otimes \widehat{v_i} \otimes \dots \otimes v_p \otimes f_1 \otimes \dots \otimes \widehat{f_j} \otimes \dots \otimes f_q.
\end{equation}
Clearly it is $\GL_n$-equivariant. 

\begin{lemma}\label{kertrij}
The maximal depth of an irreducible representation occuring in $V_{p,q}$ is equal to $q$.
The intersection of the kernels of all maps $\Tr_{i,j}$ for $1 \le i \le p$, $1 \le j \le q$ contains
the direct sum of all irreducibles of depth $q$ in $V_{p,q}$:
\begin{equation*}
\Pi_q(V_{p,q}) \subset \bigcap_{1 \le i \le p,\ 1\le j \le q} \Ker\Tr_{i,j}.
\end{equation*}
\end{lemma}
\begin{proof}
The first assertion follows easily from the Littlewood-Richardson rule.
The second follows immediately from the first, as $V_{p-1,q-1}$ does not contain irreducible 
representations of depth $q$.
\end{proof}

Next, for $p \ge q$ and a permutation $\sigma \in \SSS_p$
let us define the corresponding contraction map
\begin{equation}\label{trsi}
\Tr_\sigma:V_{p,q}  \to V_{p-q},\qquad
(v_1\otimes \dots \otimes v_p) \otimes (f_1\otimes \dots \otimes f_q)  \mapsto f_1(v_{\sigma_p})\cdots f_q(v_{\sigma_{p-q+1}}) v_{\sigma_{1}} \otimes \dots \otimes v_{\sigma_{p-q}}.
\end{equation} 
In other words, $\Tr_\sigma$ 
is the composition of the action of $\sigma\otimes\id_{V_{0,q}}$ followed by $q$
consecutive contractions of the factors $V\otimes V^*$.

\begin{lemma}
The intersection of the kernels of all maps $\Tr_\sigma$ for $\sigma \in \SSS_p$ contains
the direct sum of all irreducibles of positive depth in $V_{p,q}$:
\begin{equation*}
\Pi_{\ge 1}(V_{p,q}) \subset \bigcap_{\sigma \in \SSS_p} \Ker\Tr_\sigma = \Ker\left(\sum_{\sigma\in\SSS_p} \Tr_\sigma \right).
\end{equation*}
\end{lemma}
\begin{proof}
This follows from the fact that irreducible representations of positive depth do not occur
in $V_{p-q}$.
\end{proof}

\begin{lemma}\label{traces-lem} 
Any $\GL_n$-map $V_{p,q} \to V_{p-q}$, 
where $p\ge q$, is a linear combination $\sum_{\sigma \in \SSS_p} a_\sigma\Tr_\sigma$ of the 
contraction maps \eqref{trsi}.
Moreover, the kernel of the map
\begin{equation*}
\sum\Tr_\sigma : V_{p,q} \xrightarrow{\ \ }
\bigoplus_{\sigma \in \SSS_p} V_{p-q} 
\end{equation*}
is $\Pi_{\ge 1}(V_{p,q})$. In particular, the restriction of this map 
to $\Pi_0(V_{p,q})$ is injective.
%
\end{lemma}
\begin{proof} 
The first part follows immediately from the first fundamental theorem on invariants of $\GL_n$
(see e.g. \cite[Sec.\ 12]{CB}).
For the second part, since we already know that $\Pi_{\ge1}(V_{p,q})$ is in the kernel, we have to check that 
for each irreducible summand $V^\mu \subset V_{p,q}$
with $\mu\ge 0$ the map $\sum\Tr_\sigma$ is injective on $V^\mu$.

So, let $V^\mu \subset V_{p,q}$ be an irreducible summand with $\mu \ge 0$.
Note that $\mu$ is a partition of $p-q$, in particular,
$V^\mu$ is a direct summand of $V_{p-q}$. Choose a splitting $V_{p,q} \to V^\mu$
of the given embedding and an embedding $V^\mu \to V_{p-q}$. Then the composition
\begin{equation*}
V^\mu \to V_{p,q} \to V^\mu \to V_{p-q}
\end{equation*}
is an embedding. On the other hand, the composition of the second and the third arrows is a linear combination
of the maps $\Tr_\sigma$. It follows that for some $\sigma$ the map $\Tr_\sigma$ restricted to $V^\mu$ is nonzero,
hence injective. Therefore the map $\sum\Tr_\sigma$ is also injective on $V^\mu$.
%
\end{proof}

\begin{proof}[Proof of Proposition~\ref{glprop}]
If $N < t$ then by the Littlewood--Richardson rule,
$\Pi_0(V_{N,t}) = 0$, hence the third term in~\eqref{mainmap} is zero.
Similarly, in this case $\Pi_0(V^\lambda \otimes V_N) = 0$ for any $\lambda$ of depth $t$, 
hence the first term in~\eqref{mainmap}
becomes zero after applying $\Pi_0$. Thus, the composition~\eqref{mainmap} is $\Pi_0$-injective.

From now on assume that $N \ge t$. By Lemma~\ref{traces-lem}, we have a left exact sequence
\begin{equation*}
0 \to \Pi_{\ge 1}(V_{N,t}) \to V_{N,t} \xrightarrow{\ \sum\Tr_\sigma\ } \bigoplus_{\sigma \in \SSS_N} V_{N-t}.
\end{equation*}
Since the complement of $\Pi_{\ge 1}(V_{N,t})$ in $V_{N,t}$ is $\Pi_0(V_{N,t})$, this means
that the projection $V_{N,t} \to \Pi_0(V_{N,t})$
fits into a commutative diagram
\begin{equation*}
\xymatrix{
V_{N,t} \ar[rr] \ar[dr]_{\sum_{\sigma\in\SSS_N}\Tr_\sigma} && 
\Pi_0(V_{N,t}) \ar@{^{(}->}[dl] \\
&
\displaystyle\bigoplus_{\sigma\in\SSS_N} V_{N-t}
}
\end{equation*}
with an injective right bottom arrow.
Tensoring it with $V_{k}$ we obtain the following commutative diagram
\begin{equation*}
\xymatrix{
V_{k+N,t} = V_k \otimes V_{N,t} \ar[rr] \ar[dr]_{\sum_{\sigma\in\SSS_N} \id_{V_{k}}\otimes\Tr_\sigma\quad} && 
V_{k} \otimes \Pi_0(V_{N,t}) \ar@{^{(}->}[dl] \\
&
\displaystyle\bigoplus_{\sigma\in\SSS_N} V_{k+N-t}
}
\end{equation*}
Now consider the composition~\eqref{mainmap} and assume that $V^\mu \subset \Pi_t(V_{k,t})\otimes V_{N}$ is an irreducible summand
with $\mu \ge 0$ such that its image in $V_{k+N,t}$ is mapped to zero by the projection
$V_{k+N,t} \to V_{k}\otimes \Pi_0(V_{N,t})$. By the above commutative diagram it means that
\begin{equation}\label{trs1}
(\id_{V_{k}}\otimes\Tr_\sigma)(V_\mu) = 0
\end{equation} 
for all $\sigma \in \SSS_N$.
%
On the other hand, since $V_\mu \subset \Pi_t(V_{k,t})\otimes V_{N} \subset V_{k+N,t}$,
by Lemma~\ref{kertrij} we have
\begin{equation}\label{trs2}
\Tr_{i,j}(V_\mu) = 0
\end{equation} 
for all $1\le i \le k$ and $1 \le j \le t$.
Let us show that~\eqref{trs1} and~\eqref{trs2} lead to a contradiction. Indeed, since $\mu \ge 0$,
we know by Lemma~\ref{traces-lem} that for some $\sigma \in \SSS_{k+N}$
the trace map $\Tr_\sigma:V_{k+N,t} \to V_{k+N-t}$ is injective on $V^\mu$. 
Fix such $\sigma$. There are two possibilities:
\begin{enumerate}
\item for each $1\le i \le k$ we have $\sigma_i \le k+N-t$, or
\item for some $1 \le i \le k$ we have $\sigma_i > k+N-t$. 
\end{enumerate}
%
In the first case the map $\Tr_\sigma$ can be rewritten as the composition of $\id_{V_{k}}\otimes \Tr_{\sigma'}$ with some $\sigma' \in \SSS_N$,
followed by an appropriate permutation acting on $V_{k+N-t}$. In particular, by~\eqref{trs1} it vanishes on~$V^\mu$.
In the second case the map $\Tr_\sigma$ factors through the map $\Tr_{i,j}:V_{k+N,t} \to V_{k+N-1,t-1}$ for $j = N+k+1-\sigma_i$,
and so it vanishes on~$V^\mu$ by~\eqref{trs2}.
This contradiction finishes the proof.
\end{proof}

Now we can finish the proof of our key technical Proposition.

\begin{proof}[Proof of Proposition \ref{GL-lem-gen}]
By Lemma \ref{tableau-lem}, it is enough to prove that the map \eqref{st0} is $\Pi^a_0$-injective.
Let $k$ be the sum of parts of $\kappa$, and let $t$ be the sum of parts of $\tau$. Note that
the representation $V^{\kappa-w_0\tau}$ has depth $t$.
Let us choose some embeddings
$V^\kappa \subset V_{k}$, $V^{-w_0\tau} \subset V_{0,t}$ and $W \subset V_{N}$. Their tensor
product gives an embedding $V^\kappa \otimes V^{-w_0\tau} \otimes W \subset V_{k+N,t}$ that fits into a commutative diagram
\begin{equation*}
\xymatrix{
V^{\kappa-{w}_0\tau}\otimes W \ar[r] \ar@{..>}[d] &
V^{\kappa}\otimes V^{-{w}_0\tau}\otimes W \ar[r] \ar[d] &
V^{\kappa}\otimes \Pi^a_0(V^{-{w}_0\tau}\otimes W) \ar@{..>}[d] \\
\Pi_t(V_{k,t}) \otimes V_{N} \ar[r] &
V_{k+N,t} \ar[r] &
V_{k} \otimes \Pi_0(V_{N,t}) 
}
\end{equation*}
(the left dotted arrow comes from the embedding $V^{\kappa-{w}_0\tau}\subset \Pi_t(V_{k,t})$
 and the right dotted arrow 
is obtained by the functoriality of the projector $\Pi^a_0$). Note that all vertical arrows 
are injective. Applying the projector $\Pi_0$ (and dropping the middle terms) 
we obtain a commutative square
\begin{equation*}
\xymatrix{
\Pi_0(V^{\kappa-{w}_0\tau}\otimes W) \ar[r] \ar[d] &
\Pi_0(V^{\kappa}\otimes \Pi_0(V^{-{w}_0\tau}\otimes W)) \ar[d] \\
\Pi_0(\Pi_t(V_{k,t}) \otimes V_{N}) \ar[r] &
\Pi_0(V_{k} \otimes \Pi_0(V_{N,t}))
}
\end{equation*}
with injective vertical arrows. The bottom line is injective by Proposition~\ref{glprop}, hence so is the top line.
Applying additionally the projector $\Pi^a_0$ we conclude that the map
\begin{equation*}
\Pi^a_0(V^{\kappa-{w}_0\tau}\otimes W) \to
\Pi^a_0(V^{\kappa}\otimes \Pi_0(V^{-{w}_0\tau}\otimes W)) 
\end{equation*}
is also injective. But 
$$\Pi^a_0(V^{\kappa}\otimes \Pi_0(V^{-{w}_0\tau}\otimes W)) = \Pi^a_0(V^{\kappa}\otimes \Pi^a_0(V^{-{w}_0\tau}\otimes W)).$$
Indeed, by the Littlewood--Richardson rule, the tensor product of $V^\kappa$ with $V^\mu$ for nonnegative $\mu$ has a summand $V^\lambda$
with $\lambda_{a+1} = \dots = \lambda_n = 0$ only if $\mu_{a+1} = \dots = \mu_n = 0$. We conclude that the map
\begin{equation*}
\Pi^a_0(V^{\kappa-{w}_0\tau}\otimes W) \to
\Pi^a_0(V^{\kappa}\otimes \Pi^a_0(V^{-{w}_0\tau}\otimes W)) 
\end{equation*}
is injective.
\end{proof}



\end{document}